%% file: 1016.tex
\documentclass[12pt]{amsart}
\usepackage[colorlinks=true]{hyperref}
\hypersetup{
    colorlinks   = true,
    linkcolor    = blue,  
    urlcolor     = blue,
    citecolor    = blue
}
\usepackage{xcolor}
\usepackage{import}
\usepackage{xifthen}
\usepackage{pdfpages}
\usepackage{transparent}

\usepackage{amssymb}
\usepackage{amsmath}
\usepackage{amscd}
\usepackage{curves}
\usepackage{epsfig}
\usepackage{bbm}
\usepackage{geometry}
\usepackage{graphicx}
\usepackage{pst-plot}
\usepackage{amsfonts}
\usepackage{enumerate}
\usepackage{mathtools}
\usepackage{enumitem}
\usepackage{subcaption}

\usepackage{tikz}
\usetikzlibrary{patterns,decorations.pathmorphing}

\numberwithin{equation}{section}
\allowdisplaybreaks

\newtheorem{thm}{Theorem}[section]

\newtheorem{lm}{Lemma}[section]
\newtheorem{pp}{Proposition}[section]
\newtheorem{df}{Definition}[section]

\newtheorem{example}[thm]{Example}
\newtheorem{remark}[thm]{Remark}

\newtheorem{problem}{Problem}

\newcommand{\sff}{\textrm{I}\!\textrm{I}}

\newcommand{\diff}{\mathrm{d}}
\newcommand{\tv}{\tilde{v}}

\newcommand{\Tset}{{T}_{-1}\partial M}
\newcommand{\tM}{\tilde{M}}

\begin{document}
\title[]{Time Separation and Scattering Rigidity for Analytic Lorentzian Manifolds}


\author[]{Yuchao Yi, Yang Zhang}
\begin{abstract}

In this work, we prove the following three rigidity results: (i) in a real-analytic globally hyperbolic spacetime $(M,g)$ without boundary, the time separation function restricted to a thin exterior layer of a unknown compact subset $K \subset M$ determines $K$ up to an analytic isometry, assuming no lightlike cut points in $K$; (ii) in a real-analytic globally hyperbolic spacetime $(M,g)$  with timelike boundary, the boundary time separation function determines $M$ up to an analytic isometry, assuming no lightlike cut points near $M$ and lightlike geodesics are non-trapping; (iii) in a real-analytic Lorentzian manifold $(M,g)$ with timelike boundary, the interior and complete scattering relations near the light cone, each determines $M$ up to an analytic isometry, assuming that lightlike geodesics are non-trapping. 
We emphasize in all of these three cases we do not assume the convexity of the boundary of the subset or the manifold.  
Moreover, in (iii) we do not assume causality of the Lorentzian manifold, and allow the existence of cut points. 
Along the way, we also prove some boundary determination results, the connections between the interior and complete scattering relations, and the connections between the lens data and the scattering relation, for Riemannian manifolds and Lorentzian manifolds with boundaries.
\end{abstract}
\maketitle

\tableofcontents

\section{Introduction}\label{sec: introduction}

Let $(M, g)$ be a Lorentzian manifold with timelike boundary $\partial M$, where $g$ has the signature $(-, +, \cdots, +)$. {A geodesic is said to be \textit{non-trapping} if it exits the manifold in finite time, so that no information remains trapped inside.} In this paper, we study several rigidity problems for Lorentzian manifolds whose lightlike geodesics are non-trapping, in the analytic category.

Boundary rigidity refers to the determination of metric from its boundary distance function. For Riemannian manifolds and Lorentzian manifolds, these distance functions behave quite differently. In a Riemannian manifold, the distance between two arbitrary points $x$ and $y$ is defined as the infimum of the lengths of all piecewise smooth curves connecting $x$ and $y$. 
It endows the manifold a metric structure that is compatible with its topology. 
In a Lorentzian manifold, since it is not physically reasonable to travel faster than the speed of light, 
the distance function is usually defined only for causal curves. Moreover, because in physical applications most Lorentzian models are time-orientable, we will assume throughout the manifold is time-oriented when discussing the distance function. 
A time-oriented connected smooth Lorentzian manifold is also called a \textit{spacetime}, see \cite[Definition 3.1]{BEE96}.
Let \[
LM = \{(x, v) \in TM \setminus 0: g(v, v) = 0\}
\]
be the light cone bundle.  
The time-orientablility allows a global and continuous choice of the future and past directions. 
For \(x \in M\),
let $J^+(x)$ be its causal future, for more details see Section \ref{sec: lorentzian geometry}.
The\textit{ Lorentzian distance function}, or more commonly known as the \textit{time separation function}, is defined as follows: if $y \in J^+(x)$, then
\[
d(x, y) = \sup\{L(\alpha): \alpha \text{ future pointing piecewise smooth causal curve from $x$ to $y$}\};
\]
otherwise we set $d(x, y) = 0$. 
Here $L$ denotes the Lorentzian arc length, see Section \ref{sec: null and timelike cut locus}.
Then, we consider the following boundary rigidity problem.
\begin{problem}[\textbf{Lorentzian boundary rigidity}]
    For $j=1,2$, let $(M_j, g_j)$ be a spacetime with timelike boundary. Suppose there exists a boundary diffeomorphism $\varphi_0: \partial M_1 \to \partial M_2$, such that
    \[
    d_1(x, y) = d_2(\varphi_0(x), \varphi_0(y)), \quad \forall x, y \in \partial M_1.
    \]
    Then can $\varphi_0$ be extended to a global isometry?
\end{problem}

Another commonly studied rigidity problem is the scattering rigidity problem. Since many definitions coincide with the ones for Riemannian manifolds, here we may view $(M, g)$ as either a compact Riemannian manifold with boundary, or a Lorentzian manifold with timelike boundary. We denote the inward $(-)$ and outward $(+)$ pointing vectors on the boundary as
\[
\partial_\pm TM = \{(x, v) \in \partial TM: g(v, \nu) > 0\},
\]
where $\nu$ is the unit outward pointing vector field. Clearly
\[
\partial TM = \partial_-TM \cup \partial_+TM \cup T\partial M, \quad \overline{\partial_\pm TM} = \partial_\pm TM \cup T\partial M.
\]
There are two similar but different definitions of scattering information in the literature: one records the information when the geodesic first leaves $M^\circ$, and the other records the information when the geodesic fully leaves $M$. See Figure \ref{fig: interior and complete}.


\begin{figure}
\begin{tikzpicture}[scale=0.8, line cap=round, line join=round]
\draw[thick]
    plot[smooth cycle, tension=0.5]
        coordinates{
            (-3,-0.8) (-2.8,1.4) (-1.6,2.7) (0.0,3.0) (1.6,2.7) (2.8,1.5) (3,-0.5)
            (2.2,-1.4) (1.0,-1.3) (0, -0.3) (-1.2,-1.5) (-2.3,-1.5)
        };
\draw[line width=1.2pt, blue]
    (-3,-0.9) .. controls (-2.3,-0.7) and (-1.1,-0.35)
    .. (0.2,-0.25) .. controls (0.9,-0.2) and (2.0,-0.55)
    .. (3,-0.95);
    \draw[->,red,line width=1.0pt] (-3,-0.9) -- ++(0.6,0.17) node[above] {$v$};
    \draw[->,red,line width=1.0pt] (2.8, -0.88) -- ++(0.5,-0.2) node[right] {$u$};
    \draw[->,red,line width=1.0pt] (0,-0.25) -- ++(0.7,0.03) node[right] {$w$};
    \fill (-3,-0.9) circle (2pt) node[below left=-2pt] {$x$};
    \fill (0,-0.3)   circle (2pt) node[above=2pt] {$y$};
    \fill (2.8,-0.88)   circle (2pt) node[above=2pt] {$z$};
\end{tikzpicture}

    \caption{For $(x, v) \in \partial_-TM$, denote the corresponding geodesic by $\gamma$, then $(y, w)$ is the point and direction at which $\gamma$ leaves $M^\circ$ for the first time, or equivalently its first time reaching $\partial M$; and $(z, u)$ is where the geodesic fully leaves $M$. The scattering relation $(x,v) \to (z, u)$ and $(y, w) \to (z, u)$ will be included in the complete scattering relation, but interior scattering relation will only record $(x, v) \to (y, w)$. The $(y, w) \to (z, u)$ part will not be recorded in the interior scattering relation as $(y, w)$ is tangential to the boundary.}
    \label{fig: interior and complete}
\end{figure}
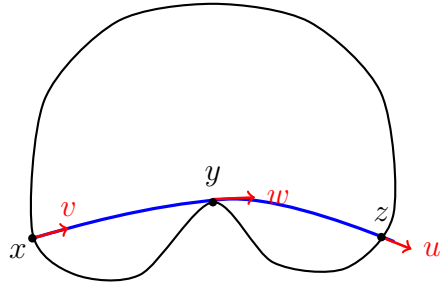

For $(x, v) \in \partial_-TM$, we denote the corresponding geodesic by $\gamma$ below. Define the \textit{interior travel time} as
\[
\tau^{in}(x, v) = \sup\{t>0: \gamma((0, t)) \subset M^\circ\}.
\]
If $\tau^{in}(x, v)$ is finite, then we define
\[
S^{in}(x, v) = (\gamma(\tau^{in}(x, v)), \dot{\gamma}(\tau^{in}(x, v))), \quad \ell^{in}(x, v) = \tau^{in}(x, v) \cdot |g(v, v)|^{1/2}
\]
as the \textit{interior scattering relation} and the \textit{interior length function}. The tuple $(S^{in}, \ell^{in})$ is called the \textit{interior lens data}, and the tuple $(S^{in}, \tau^{in})$ is called the \textit{interior travel time data}. 
Note that the domain of the interior scattering information is a subset of $\partial_- TM$, excluding the trapping directions and boundary tangential directions.

Now let $(\tM, \tilde{g})$ be any extension of $(M, g)$. 
For any $(x, v) \in \overline{\partial_- TM}$, 
as before, let $\gamma$ be the corresponding geodesic. We define the \textit{complete travel time} as
\[
\tau(x, v) = \sup \{t>0: \gamma([0, t]) \subset M\}.
\]
If $\tau(x, v)$ is finite, then we define
\[
S(x, v) = (\gamma(\tau(x, v)), \dot{\gamma}(\tau(x, v))), \quad \ell(x, v) = \tau(x, v) \cdot |g(v, v)|^{1/2}
\]
as the \textit{complete scattering relation} and the \textit{complete length function}. The tuple $(S, \ell)$ is called the \textit{complete lens data}, and the tuple $(S, \tau)$ is called the \textit{complete travel time data}. Note that the domain of the complete scattering information is a subset of $\overline{\partial_- TM}$, which may include the boundary tangential directions when they are non-trapping.

For both Riemannian manifold and Lorentzian manifold with timelike boundary, the unit outward normal vector field $\nu$ is well-defined. In this paper, an $(x, v) \in T\partial M$ is called a \textit{strictly convex direction}, if the second fundamental form is positive, i.e.,
\[
\sff (v, v) := g(\nabla_v \nu, v) > 0.
\]
In particular, if all boundary tangential vectors are strictly convex directions, we say the boundary is {strictly convex}.\footnote{More commonly, for a Riemannian manifold with boundary, strict convexity is defined by the property that any two sufficiently close points on the boundary are connected by a distance minimizing geodesic that lies in the interior except the two end points. This implies $\sff$ is positive semi-definite. Conversely, positive definite $\sff$ implies this property. See for example \cite[Section 3.2]{GM25}. We will only use the definition of a strictly convex direction in this paper.} In this case, the two kinds of scattering information coincide because any geodesic in $M$ can only intersect the boundary transversally. Moreover, in a Lorentzian manifold with timelike boundary, if all boundary tangential lightlike vectors are strictly convex directions, then we say the manifold is strictly null-convex (see \cite{HU19}). In this case, the two scattering relations for lightlike geodesics also coincide.

In general, interior and complete scattering information are different, and recovering one from the other is not obvious. 
Complete scattering information appeared, for example, in \cite{Var09, GMT21, SUV21}. In applications, it models through-transmission setups where rays are detected only after they fully exit the body, such as medical or industrial CT, cross-well seismics, and through-thickness ultrasound. 
In these setups, the recorded data are the entry and final exit states (possibly with travel time). Meanwhile, interior scattering information has also made its appearance, for example, in \cite{SU09, Cro14, Wen15, SUVZ19, CW15}.
Interior scattering relation matches first-arrival measurements made on the boundary, as in reflection seismology, ground-penetrating radar, or surface-mounted acoustic sensing, where a signal is registered as soon as it first returns to the boundary even if the ray later glides along it or exits elsewhere.
We study how the two definitions are related to each other in Section \ref{sec: interior and complete scattering information}.

For $j = 1, 2$, suppose $(M_j, g_j)$ are either both compact Riemannian manifolds with boundaries or both Lorentzian manifolds with timelike boundaries. Furthermore, suppose there exists a boundary isometry $\varphi_0: \partial M_1 \to \partial M_2$, that is, $\bar{g}_1 = \varphi_0^*\bar{g}_2$, where $\bar{g}_j = g_j|_{T\partial M_j \times T\partial M_j}$ is the boundary metric. In the boundary normal coordinates, one may write $(x, v)$ as $(x, v', v^n)$, where the boundary is locally given by $x^n = 0$ and is the interior corresponds to $x^n > 0$. 
By abuse of notation, we view $(\varphi_0)_*$ as
\[
(\varphi_0)_*: \partial TM_1 \to \partial TM_2, \quad (x, v', v^n) \mapsto (\varphi_0(x), (\varphi_0)_*v', v^n).
\]
Then we have the following two rigidity problems.
\begin{problem}[\textbf{Interior scattering rigidity}]
    For $j=1,2$, let $(M_j, g_j)$ be either both compact Riemannian manifolds with boundaries or both Lorentzian manifolds with timelike boundaries. Suppose there exists a boundary isometry $\varphi_0: \partial M_1 \to \partial M_2$. Suppose furthermore that in boundary normal coordinates,
    \[
    (\varphi_0)_* \circ S^{in}_1 = S^{in}_2 \circ (\varphi_0)_*
    \]
    on the domain of $S^{in}_1$ (which is usually the entire $\partial_-TM_1$ for Riemannian manifolds with the non-trapping assumption). Then can $\varphi_0$ be extended to a global isometry?
\end{problem}
\begin{problem}[\textbf{Complete scattering rigidity}]
    For $j=1,2$, let $(M_j, g_j)$ be either both compact Riemannian manifolds with boundaries or both Lorentzian manifolds with timelike boundaries. 
    Suppose there exists a boundary isometry $\varphi_0: \partial M_1 \to \partial M_2$. Suppose furthermore that in boundary normal coordinates,
    \[
    (\varphi_0)_* \circ S_1 = S_2 \circ (\varphi_0)_*
    \]
    on the domain of $S_1$ (which is usually the entire $\overline{\partial_-TM_1}$ for Riemannian manifolds with the non-trapping assumption). Then can $\varphi_0$ be extended to a global isometry?
\end{problem}


In this paper, we focus on Lorentzian manifolds with timelike boundaries. 
As a result, it is natural to focus on scattering information associated with causal geodesics. Specifically, we will assume the lightlike geodesics are non-trapping, and work with both the interior and complete scattering information for causal geodesics sufficiently close to light cones.

In much of the literature, the \textit{projected scattering relation} is considered instead for Riemannian manifolds, see for example \cite{SU09,ILS23}. In this paper, when we talk about scattering rigidity problems, we always assume the boundary metric is already given. Knowledge of the boundary metric makes the scattering relation and the projected scattering relation equivalent, so we do not work with the projected version. Nevertheless, we include a brief discussion of it for completeness. 
For both the interior and complete scattering relation, suppose $(x, v) \in \partial SM$ is a unit vector mapped to $(y, w) \in \partial SM$, where $SM$ is the spherical bundle of the Riemannian manifold. 
Then the projected scattering relation maps $(x, \kappa(v))$ to $(y, \kappa(w))$, where $\kappa: \partial SM \to \overline{B\partial M}$ is the orthogonal projection in boundary normal coordinates and $\overline{B_x\partial M}$ is the closed unit ball in $T_x\partial M$. 
Hence, the domain and range of the projected scattering relation are both subsets of $\overline{B\partial M}$. Similarly, one can consider the projected scattering relation in the Lorentzian setting. More precisely, for any $(x, v') \in T\partial M$ such that $g(v', v') < -1$, there exists a unique inward pointing unit timelike vector $(x, v)$ such that $\kappa(v) = v'$, given by 
$
v = v' + \sqrt{-1-g(v',v')}\partial_n,
$
in boundary normal coordinates.
Then the projected timelike scattering relation can thus be defined as a map from a subset of $\Tset$ to $\Tset$, where we write
$\Tset  = \{(x,v') \in T\partial M: g(v', v') \leq -1\}.$
However, there also exists a unique inward pointing lightlike vector $(x, \tilde{v})$ such that $\kappa(\tilde{v}) = v'$, given by $\tilde{v} = v' + |v'|_g\partial_n$. 
Therefore, when discussing the projected scattering relation for Lorentzian manifolds with timelike boundaries, it is important to distinguish timelike and lightlike ones. We emphasize again that throughout this work, we do not work with the projected version to avoid any confusion, as we assume the boundary metric is already known.

\subsection{Main results}\label{sec: main results}
We first study the case when the unknown region is embedded in a larger analytic globally hyperbolic Lorentzian manifold, and we assume lightlike geodesics do not have cut points there. We show that given the time separation function supported in a thin exterior layer of the unknown region, the unknown region can be determined up to isometry.
\begin{thm}
\label{thm: exterior data no cut points}
    For $j = 1, 2$, let $(N_j, g_j)$ be an analytic globally hyperbolic Lorentzian manifold of dimension $n \geq 3$ and let $d_j$ be the corresponding time separation function. Let $K_j \subset \subset \tM_j \subset N_j$ be such that $K_j$ is a compact subset and $\tM_j$ is an open neighborhood of $K_j$. Denote by $K^c_j = \tM_j \backslash K_j$. 
    Suppose there exists a diffeomorphism $\varphi_0: K^c_1 \to K^c_2$ such that
    \[
    d_1(x, y) = d_2(\varphi_0(x), \varphi_0(y)), \quad \forall x, y \in K^c_1,
    \]
    and $\varphi_0$ extends continuously to a bijection between $\partial K_1$ and $\partial K_2$.
    Suppose every lightlike geodesic segment in $K_j$ does not have cut point. Then there exists an analytic isometry $\varphi: \tM_1 \to \tilde{M_2}$ such that $\varphi|_{K^c_1} = \varphi_0$.
\end{thm}

We emphasize that no regularity assumption is imposed on the unknown region $K_j$ for $j = 1,2$; it may be any compact subset. 
The exterior region $K^c_j$ can also be an arbitrarily thin layer.
To prove this theorem, 
we use the complete travel time data for lightlike geodesics with respect to a smaller neighborhood of $K_j$, for $j = 1,2$.
This smaller neighborhood is more convenient to work with, since the boundary of $K_j$ may have low regularity. 
We reconstruct the lightlike complete travel time data from the exterior time separation function, see Section \ref{sec: lightlike complete travel time data}.
Then in Proposition \ref{prop: complete light scattering recover metric}, inspired by \cite{Var09}, we show that this lightlike information establishes a correspondence between the two neighborhoods, which is proved to preserve the geometric structure and can be extended to an analytic isometry that matches the given exterior identification.
For more details, see Section \ref{sec: exterior time separation function}.

Next, we study the boundary rigidity problem for analytic globally hyperbolic spacetimes with timelike boundaries. 
For the definition of globally hyperbolic spacetimes with timelike boundaries, see Section \ref{sec: lorentzian geometry}. By an analytic spacetime with timelike boundary, we mean a smooth manifold with boundary whose transition maps are analytic, the metric is analytic, and the boundary is timelike. Certainly the boundary and boundary metric are analytic with respect to the induced analytic structure. We require that after a small analytic collar neighborhood extension, see Section \ref{sec: normal exponential extension}, it is still globally hyperbolic, and lightlike geodesics do not have cut points in the extension. We also require the existence of a strictly convex direction on the boundary, for the recovery of the jet of boundary metric.

\begin{thm}
\label{thm: boundary distance main thm}
    For $j=1,2$, let $(M_j, g_j)$ be an analytic globally hyperbolic spacetime with timelike boundary of dimension $n \geq 3$, and suppose lightlike geodesics are non-trapping. 
    Assume there exists an analytic collar neighborhood extension,
    see Section \ref{sec: normal exponential extension}, which is globally hyperbolic and contains no lightlike cut points.
    Suppose there exists an analytic diffeomorphism $\varphi_0: \partial M_1 \to \partial M_2$ such that
    \[
    d_1(x, y) = d_2(\varphi_0(x), \varphi_0(y)), \quad 
    \forall x, y \in \partial M_1.
    \]
    Moreover, assume for any connected component $C$ of $\partial M_1$, there exists a causal direction $(x, v) \in TC$ such that $(x, v)$ and $(\varphi_0(x), (\varphi_0)_*v)$ are strictly convex directions in $(M_1, g_1)$ and $(M_2, g_2)$, respectively. Then there exists an analytic isometry $\varphi: M_1 \to M_2$ such that $\varphi|_{\partial M} = \varphi_0$.
\end{thm}

Comparing to the exterior case in Theorem \ref{thm: exterior data no cut points}, 
boundary rigidity problem is more subtle due to the presence of the boundary. 
Since we do not assume the timelike boundary to be strictly convex, the distance maximizing curve may no longer be a geodesic.
To prove this theorem, by Proposition \ref{prop: complete light scattering recover metric},
it suffices to reconstruct the lightlike complete travel time data with respect to $M_j$, for $j =1,2$, from the boundary time separation function $d_j$. 
For this purpose, we first determine the jet of the metric in an analytic collar neighborhood of each $M_j$.
Then we determine the exterior time separation function $\tilde{d}_j$ in this neighborhood, from the knowledge of $d_j$. 
However, since $\tilde{d}_j$ is the time separation function only defined within this neighborhood, 
the maximal distance is not necessarily realized by a causal geodesic. 
We emphasize this situation differs from the exterior case and a different reconstruction procedure is required to recover the lightlike complete travel time data, see Section \ref{sec: boundary rigidity}.
In addition, in this setting, we assume there are no lightlike cut points in the extension of $(M,g)$. 
This simplifies the recovery, avoiding the case that a boundary point of $M$ becomes a cut point in the extended manifold.

In \cite{LOY16}, the authors show that the Lorentzian universal covering of an analytic Lorentzian spacetime $(M,g)$ is determined by the time separation function restricted to a small timelike submanifold contained in $M$, assuming that $(M,g)$ is geodesically complete modulo scalar curvature singularities. 
As only local information is assumed, one can only expect uniqueness up to an isometry of the universal covering space, not of the manifold itself.
In contrast, in this work, we study the boundary rigidity problem for globally hyperbolic Lorentzian manifolds with timelike boundary. As a manifold with boundary, it is always geodesically incomplete. 
Moreover, since we have global information, we can completely determine the metric up to isometry.
In \cite{Ste25}, the local lightlike scattering rigidity for analytic metrics is established.
For a more comprehensive overview of previous results, see Section \ref{sec: history}.


Finally, we study the scattering rigidity problem for analytic Lorentzian manifolds with timelike boundaries. 
In this setting, we do not require any causality conditions. 
To state the result, we first introduce a mild non-conjugacy condition, which comes from \cite{SU09}.

\begin{df}
    Let $(M,g)$ be a semi-Riemannian manifold  with boundary. 
    Consider a boundary tangential direction $(x, v) \in T\partial M$. Let $\gamma$ be the corresponding geodesic. The direction $(x, v)$ is said to satisfy the \textit{non-conjugacy condition}, if $\gamma$ is non-trapping and $x$ is not conjugate to any point in $\gamma \cap \partial M$.
\end{df}
Note that in a Lorentzian manifold with strictly null-convex timelike boundary, any boundary tangential timelike vector sufficiently close to the light cone will satisfy this condition.
Recall by abuse of notation, we can view the differential of a boundary isometry $\varphi_0$ as a map from $\partial TM_1$ to $\partial TM_2$, in boundary normal coordinates. 
For simplicity, we denote by 
\[
JM = \{(x, v) \in TM \setminus 0: g(v, v) \leq 0\}
\]
the set of causal vectors in $M$.
Then let $\partial_\pm JM$ be the inward $(-)$ and outward $(+)$ pointing ones on the boundary. We show that both the interior and complete scattering rigidity hold.

\begin{thm}
\label{thm: interior scattering main thm}
    For $j=1,2$, let $(M_j, g_j)$ be an analytic Lorentzian manifold of dimension $n \geq 3$ with analytic timelike boundary. Assume all null geodesics are non-trapping.
    Then, for sufficiently small conic neighborhoods $\mathcal{U}_j \subset \partial TM_j$ containing $\partial LM_j$, the interior scattering relation $S^{in}_j$ is well-defined in $\mathcal{V}_j^{in} \coloneqq \mathcal{U}_j \cap \partial_-JM_j$. Moreover, assume for each connected component $C$ of $\partial M_1$, there exists $(x, v) \in TC \ \cap\  \overline{\mathcal{V}_1^{in}}$ such that the non-conjugacy condition holds. If there exists a boundary isometry $\varphi_0: \partial M_1 \to \partial M_2$ such that
    \[
    (\varphi_0)_*(\mathcal{V}_1^{in}) = \mathcal{V}_2^{in}, \quad (\varphi_0)_* \circ S^{in}_1|_{\mathcal{V}_1^{in}} = S^{in}_2|_{\mathcal{V}_2^{in}} \circ (\varphi_0)_*,
    \]
    then there exists an analytic isometry $\varphi: M_1 \to M_2$ such that $\varphi|_{\partial M_1} = \varphi_0$.
\end{thm}

\begin{thm}\label{thm: complete scattering main thm}
    For $j=1,2$, let $(M_j, g_j)$ be an analytic Lorentzian manifold of dimension $n \geq 3$ with analytic timelike boundary. Assume all null-geodesics are non-trapping.
    Then, for sufficiently small conic neighborhoods $\mathcal{U}_j \subset \partial TM_j$ containing $\partial LM_j$, the complete scattering relation $S_j$ is well-defined in ${\mathcal{V}_j} \coloneqq \mathcal{U}_j \cap \overline{\partial_-JM_j}$.
    Moreover, assume for each connected component $C$ of $\partial M_1$, there exists $(x, v) \in TC \ \cap\ \mathcal{V}_1 $ such that the non-conjugacy condition holds.
    If there exists a boundary isometry $\varphi_0: \partial M_1 \to \partial M_2$ such that
    \[
    (\varphi_0)_*(\mathcal{V}_1) = \mathcal{V}_2, \quad
    (\varphi_0)_* \circ S_1|_{\mathcal{V}_1} = S_2|_{\mathcal{V}_2} \circ (\varphi_0)_*,
    \]
    then there exists an analytic isometry $\varphi: M_1 \to M_2$ such that $\varphi|_{\partial M_1} = \varphi_0$.
\end{thm}

The proof of Theorem \ref{thm: interior scattering main thm} and Theorem \ref{thm: complete scattering main thm} rely on certain boundary determination results, which we develop in Section \ref{sec: determination of jet}, following the ideas in \cite{SU09}.
In order to prove the rigidity result for both interior and complete scattering relations, we study their connections in Section \ref{sec: interior and complete scattering information}.
In particular, in Section \ref{sec: lorentzian setting}, we prove that both the lightlike interior and complete travel time data can be recovered from either the interior or the complete scattering relation for timelike vectors sufficiently close to the light cone, see Lemma \ref{lm: lorentzian recover travel time from interior} - \ref{lm: lorentzian recover complete from interior}.
The results in Section \ref{sec: interior and complete scattering information} do not rely on any causality or analyticity assumptions on the manifold, or on any convexity of the timelike boundary. 
Instead, we derive these results using geometric arguments and the first variation of the travel time in Section \ref{sec: the first variation of the travel time}.
Then apply Proposition \ref{prop: complete light scattering recover metric} again, we complete the proof of these two theorems in Section \ref{sec: scattering rigidity}.

\subsection{Previous literature}\label{sec: history}
In Riemannian manifolds, a  geometric formulation of boundary distance rigidity goes back to Michel \cite{Mic81}.
In \cite{PU05}, it is proved two dimensional compact simple Riemannian manifolds are boundary distance rigid. 
Besides, this result is known to hold for simple subdomains in Euclidean spaces, in an open hemisphere in two dimensions, in symmetric spaces of constant negative curvature, and for two dimensional spaces of negative curvature, see \cite{Gro83, Mic81,BCG95,Cro90}.
When one metric is sufficiently close to the Euclidean metric, boundary rigidity was established in \cite{LSU03} and later improved in \cite{BI10}.
In particular, for smooth metrics in dimensions $n \geq 3$, local and global boundary rigidity results under foliation conditions are established, see \cite{UV16, SUV16, SUV21}.
Lens and scattering data provide alternative boundary measurements beyond the boundary distance. Local lens rigidity with incomplete data for classes of non-simple manifolds was proved in \cite{SU09}. 
In the presence of trapped geodesics, scattering rigidity and lens rigidity were obtained 
under different geometric assumptions, see \cite{Cro14, CH16, Gui17}.
In the analytic category, lens rigidity was established for non-trapping Riemannian manifolds in \cite{Var09}, and rigidity problems for such manifolds with an analytic magnetic field are considered as well, see \cite{DPSU07, HV11}.
For other related works and generalizations, see 
\cite{Mun24a, Mun24b,AZ15,LS19,ILS23,dHILS22,PSU13,Cro04,Muh81,BCG95}, and we refer the readers to surveys \cite{SUVZ19, IM19}.
These rigidity problems are closely related to the (anisotropic) Calderón problem for the Laplace–Beltrami operator, for example, see \cite{LTU03, LU01, LU89} for real-analytic Riemannian metrics and surveys \cite{Uhl12, Sal13} for more results.

In Lorentzian manifolds, boundary distance rigidity is stated in terms of the time separation function.
The recovery of stationary metrics from the time separation function is studied in
\cite{ADH96} for two-dimensional product Lorentzian manifolds, and in \cite{LOY16} for universal covering spaces of real-analytic Lorentzian manifolds, in \cite{UYZ21} for manifolds with dimensions three and higher.
The recovery of stationary metrics from the lens relation is studied in \cite{Ste24}. Then it is considered in \cite{Mun24c} via timelike geodesics and MP-systems.
Analogously, these Lorentzian rigidity problems are closely related to the Lorentzian Calder\'{o}n problems, where we consider the recovery of the Lorentzian manifold $(M,g)$ from the Dirichlet-to-Neumann map associated with a wave equation.
For ultra-static manifold with stationary metrics or metrics real-analytic in $t$,
this problem has been extensively studied in the literature, including
\cite{Bel87, Esk07,KKL01,KOP18,LO14}.
The most recent progress is in \cite{AFO22, AFO24}, where the assumption of real-analyticity in $t$ is replaced by bounds on the Lorentzian curvature. 

In both Riemannian and Lorentzian settings, boundary determination problems are considered usually as the first step towards interior recovery.
In Riemannian geometry, boundary determination has been studied, with well-established results on uniqueness, stability, and constructive reconstruction methods, see \cite{LSU03, SU05, UW03} for cases with strictly convex boundaries, and \cite{SU09, Zho12} for cases with concave points.
In Lorentzian geometry, 
the determination of the jet of the metric on a timelike submanifold from the time separation function is considered in \cite{LOY16}.
Boundary determination of the metric, the magnetic field, and the potential from the knowledge of the DN map has been studied in \cite{SY18}, and later discussed in \cite{YZ25} using partial data in disjoint boundary sets.

\subsection{Outline}
In Section \ref{sec: preliminaries}, we include some preliminary results and some useful lemmas. In Section \ref{sec: exterior time separation function}, we prove Theorem \ref{thm: exterior data no cut points}. In Section \ref{sec: boundary rigidity}, we prove Theorem \ref{thm: boundary distance main thm}. Section \ref{sec: determination of jet} to Section \ref{sec: scattering rigidity} are related to the scattering rigidity results. Specifically, in Section \ref{sec: determination of jet}, we study the determination of the jet of metric on the boundary for both Riemannian manifold and Lorentzian manifold with timelike boundary. In Section \ref{sec: interior and complete scattering information}, we study how the interior and complete scattering information relates to each other. The results from these two sections will be used in Section \ref{sec: scattering rigidity} to prove Theorem \ref{thm: interior scattering main thm} and Theorem \ref{thm: complete scattering main thm}. Finally, Appendix \ref{sec: transversal geodesics} includes some transversal intersection results, and Appendix \ref{sec: construction of jet with strictly convex direction} includes the constructive recovery of the jet of metric on the boundary near strictly convex directions. 

\subsection*{Acknowledgment}
The authors thank Plamen Stefanov for bringing Lorentzian rigidity problems in the analytic setting to their attention during a discussion. Motivated by that conversation, we explored a related setting, leading to the results presented here.
The authors thank Gunter Uhlmann for numerous discussions and helpful suggestions. YY was supported in part by the National Science Foundation.

\section{Preliminaries}\label{sec: preliminaries}
In this section, we briefly recall some basic definitions and properties of Lorentzian manifolds that will be used in the paper.

\subsection{Lorentzian geometry}\label{sec: lorentzian geometry}
Recall that a smooth Lorentzian manifold $(M,g)$ is a spacetime if it is connected and has a time orientation.
We say a smooth path $\mu:(a,b) \rightarrow M$ is timelike, if $g(\dot{\mu}(s), \dot{\mu}(s)) < 0$ for all $s \in (a,b).$
We say a smooth path $\mu:(a,b) \rightarrow M$ is causal, if $g(\dot{\mu}(s), \dot{\mu}(s)) \leq 0$ with $\dot{\mu}(s) \neq 0$ for all $s \in (a,b).$
Given $x, y \in M$, we say $x \ll y$ if there exists a future directed piecewise smooth timelike curve from $x$ to $y$.
We say $x < y$ if there is a future directed piecewise smooth causal curve from $x$ to $y$. 
We write $x \leq y$ if $x = y$ or $x < y$.
The chronological future and past of $x \in M$ are defined as $I^+(x) = \{y \in M: x \ll y\}$ and 
$I^-(x) = \{y \in M: y \ll x\}$ respectively.
The causal future and past of $x \in M$ are defined as
$J^+(x) = \{y \in M: x \leq y\}$ and 
$J^-(x) = \{y \in M: y \leq x\}$ respectively.
We denote by $J(x, y) = J^+(x) \cap J^-(y)$ the causal diamond set. 
For a subset $U \subseteq M$, its causal futures (or past) are defined as
$
    J^\pm(U) = \bigcup_{x \in U} J^\pm(x).
$

In a spacetime $(M, g)$, 
a \emph{pregeodesic} is a smooth curve whose image is a geodesic, possibly with a non-affine parameterization.
An open set $U \subset M$ is called a \textit{geodesically convex neighborhood}, provided $U$ is a normal neighborhood of each of its points, see \cite[Definition 5.5]{One83}.
In particular, for any two points $x, y \in U$ there is a unique geodesic segment $\gamma:[0, 1] \rightarrow N$ that lies entirely in $U$, while there are possibly other geodesics from $x$ to $y$ that do not remain in $U$.

We say a spacetime is \textit{causal} if it does not have any closed causal curves. 
For a spacetime $(M, g)$ without boundary, 
we say it is \textit{globally hyperbolic}, if it is strongly causal (see \cite[Definition 14.11]{One83}) and every causal diamond is compact.
In fact, in this definition, it suffices to assume causal instead of strongly causal. 
This provides one possible definition for a subset to be globally hyperbolic (see \cite[Definition 14.20]{One83}), i.e., $H \subset M$ is said to be a \textit{globally hyperbolic subset}, if it satisfies the strong causality condition, and every causal diamond is compact in $H$. With this definition, one can then study subsets of $M$ with timelike boundaries, and refer to them as globally hyperbolic Lorentzian manifolds with timelike boundaries. However, this definition is not intrinsic, as the strong causality depends on the ambient manifold. 
Throughout this paper, we instead use the following intrinsic definition of being globally hyperbolic for manifolds with timelike boundaries. 
\begin{df}{\cite[Definition 2.15]{AFS21}}
    Let $(M, g)$ be a Lorentzian manifold with timelike boundary. It is globally hyperbolic if it is causal and the causal diamond set $J(x, y)$ is compact for any $x, y \in M$. 
\end{df}
The definition is intrinsic as we only consider causal curves in $M$. By \cite[Corollary 5.8]{AFS21}, every globally hyperbolic Lorentzian manifold with timelike boundary can be isometrically embedded into a globally hyperbolic Lorentzian manifold without boundary of the same dimension. 
Moreover, globally hyperbolic Lorentzian manifolds with or without boundaries are always time-orientable, so we may refer to them as globally hyperbolic spacetime.

\subsection{Null and timelike cut locus}\label{sec: null and timelike cut locus}
In a spacetime $(M,g)$ without boundary, for $x < y$, we define the Lorentzian arc length
\[
L(\alpha) = \int_0^1 \sqrt{-g(\dot{\alpha}(s), \dot{\alpha}(s))}\,\mathrm{d}s,
\]
for a  piecewise smooth causal path $\alpha: [0,1] \to M$ from $x$ to $y$.
Note that this definition is independent of the parameterization of $\alpha$.
In addition, by \cite[Remark 3.35]{BEE96}, the Lorentzian arc length functional $L$ is upper semicontinuous.

Recall the time separation function $d$ for a Lorentzian manifold is the supremum of the arc length of all piecewise smooth causal curves from $x$ to $y$, if $x \leq y$.
It satisfies the reverse triangle inequality
\[
d(x,y) + d(y,z) \leq d(x,z), \qquad \text{for } x \leq y \leq z.
\]
For $(x,v) \in L^+ M$, recall the \emph{null cut locus function}
\[
\rho(x,v) = \sup \{\, s \in [0, \mathcal{T}(x,v)): \tau(x, \gamma_{x,v}(s)) = 0 \,\},
\]
where $\gamma_{x,v}(s)$ is the unique null geodesic starting from $x$ in the direction $v$, and $\mathcal{T}(x,v)$ is supremum of the parameter value for which $\gamma_{x,v}(s) \in M$ is defined.
For a timelike vector $(x,v) \in TM$, recall the \textit{timelike cut locus function}
\[
\rho(x,w) = \sup \{\, s \in [0, \mathcal{T}(x,v)): d(x, \gamma_{x,w}(s)) = s\,\},
\]
where $\gamma_{x,v}(s)$ is the unique timelike geodesic starting from $x$ in the direction $v$,
see \cite[Definition 9.3]{BEE96}.
In both cases, when $\rho(x,v) < \mathcal{T}(x,v)$, we call $\gamma_{x,v}(\rho(x,v))$ the first cut point of $x$ along the null or timelike geodesic  $\gamma_{x,v}$.
By \cite[Theorem 9.12]{BEE96} and \cite[Theorem 9.15]{BEE96},  
in a globally hyperbolic Lorentzian manifold without boundaries, the first cut point $\gamma_{x,v}(\rho(x,v))$ is either the first conjugate point to $x$ along $\gamma_{x,v}$, or the first point on $\gamma_{x,v}$ at which there exists another distinct null or timelike geodesic from $x$ to $\gamma_{x,v}(\rho(x,v))$.

Moreover, we include the following lemma about the boundary of the causal future in a globally hyperbolic Lorentzian manifold without boundary. 
\begin{lm}\label{lm: boundary of causal future}
    {Let $(M, g)$ be a globally hyperbolic Lorentzian manifold.}
    {Let $x, y \in M$ and $\partial J^+(x) \coloneqq J^+(x) \setminus I^+(x)$ be the boundary of the causal future of $x$.
    The following statements are equivalent.
    \begin{enumerate}
        \item $y \in \partial J^+(x)$;
        \item $y \in J^+(x)$ and $d(x, y) = 0$;
        \item there exists $y_j \to y$ such that $d(x, y_j) > 0$ and $d(x, y_j) \to 0$;
        \item there exists a future pointing lightlike geodesic from $x$ to $y$ entirely lying on $\partial J^+(x)$, with no cut points strictly before $y$.
    \end{enumerate}
    In particular, when there are no lightlike cut points, one has $\partial J^+(x) = \mathcal{L}^+_x$, where $\mathcal{L}^+_x$ is the forward light cone defined as the union of all lightlike geodesics starting from $x$.
    }
\end{lm}
\begin{proof}
First, we prove the statement (1) is equivalent to (2). 
Indeed, with $y \in J^+(x)$, one has $d(x, y) \geq 0$.
Recall $y \in I^+(y)$ exactly when $d(x, y) > 0$. 
Then $y \in J^+(x)$ with $d(x, y) = 0$  if and only if 
$y \in J^+(x) \setminus I^+(x)$.

Next, we prove (2) is equivalent to (3). 
Indeed, as $(M,g)$ is globally hyperbolic, the causal future $J^+(x)$ is closed and is the closure of the open set $I^+(x)$.
Moreover, the time separation function $d(x, y)$ is continuous in $M \times M$. 
Suppose $y \in J^+(x)$ with $d(x, y) = 0$. Then we can find a sequence $y_j \in I^+(x)$ such that $y_j \rightarrow y$. By continuity, one has $d(x, y_j) \rightarrow d(x, y) =0$. 
On the other hand, if (3) holds, then $y_j \in I^+(x)$ and therefore we must have $y \in J^+(x)$ with $d(x,y) = 0$ by continuity again.

Finally, using \cite[Theorem 1]{ABHR18}, the statement (1) is equivalent to (4), see also \cite[Lemma 2.3]{KLU18}.
\end{proof}

We emphasize these definitions above about cut locus do not work well in manifolds with boundaries.
Indeed, using these definitions, in a globally hyperbolic Lorentzian manifold $(M,g)$ with timelike boundary, $\rho(x,v)$ may give us the exit time $s_0$ 
instead of the true cut locus. 
Therefore, we define the cut locus in $M$ as the infimum time when the geodesic fails to be maximal, see the definitions in (\ref{def: null cut bd}) and (\ref{def: time cut bd}).  
Moreover, suppose $(M,g)$ is embedded in a globally hyperbolic Lorentzian manifold $(N, g)$ with or without timelike boundary, such that $M \subset N^\circ$. Let $d_N$ be the Lorentzian distance function and $\rho_N$ be the cut locus in $N$. 
Then one has $d(x, y) \leq d_N(x,y)$ for any $x, y \in N$ and 
$\rho(x,v) \geq \rho_N(x,v)$ for any lightlike or timelike $(x,v) \in TM$.

\subsection{Normal exponential extension}\label{sec: normal exponential extension}
Next, we briefly talk about the extension of a Lorentzian manifold with boundary, by using a collar neighborhood. 
This is usually a crucial first step for the recovery of interior metric. 
Given an analytic Lorentzian manifold $(M, g)$ with timelike boundary, one can always extend the smooth manifold slightly to a larger smooth manifold $\tM$. When the extension is sufficiently small, one can analytically extend the metric $g$ to $\tilde{g}$ on $\tM$. 
For every point $x \in \partial M$, 
let $U_x$ be a small neighborhood of $x$ where the metric can be written in the boundary normal coordinates. 
Choose a subset of $\{U_x: x\in \partial M\}$ that is locally finite and covers $\partial M$.
Note such subset can be made countable and we denote it by $\{U_j\}_1^\infty$. 
Take $V = \cup_j U_j$, and we shrink $V$ if necessary. For any $y \in V$, there exists a unique $x \in \partial M$ such that $y$ is on the normal geodesic from $x$ in $V$. This is possible by shrinking $V$, since the cover is locally finite. Then the following map is well-defined:
\[
\exp^{-1}_{\nu}: V \to \partial M \times (-\epsilon, \epsilon), \quad y= \exp_x(s\nu) \mapsto (x, s),
\]
where $\nu$ is the unit normal vector field on $\partial M$. We refer to its inverse, $\exp_{\nu}$, as the \textit{normal exponential map}, whose domain is the range of $\exp_\nu^{-1}$. Comparing to the compact Riemannian manifold case in \cite[Section 2]{Var09}, the domain of $\exp_\nu$ in our setting may not be uniform in size, since $M$ is not compact. Without loss of generality, we may assume $\tM = V \cup M$ in the first place. In this paper, we refer to such an extension as the \textit{analytic collar neighborhood extension}. Finally, we can always make sure that after the extension, the outer boundary is still timelike.





\subsection{Limit curves}\label{sec: limit curves}
In a globally hyperbolic spacetime $M$ with timelike boundary, we recall the following concept of limit curves, which are only used in Section \ref{sec: boundary rigidity}.
\begin{df}[{\cite[Definition 3.28]{BEE96}}]
A curve $\gamma$ is a limit curve of the sequence $\{\gamma_n\}$ if there is a subsequence $\{\gamma_m\}$ such that for all $p$ in the image of $\gamma$, each neighborhood of $p$ intersects all but a finite number of curves of the subsequence $\{\gamma_m\}$. 
\end{df}
Note that a sequence of curves $\{\gamma_n\}$ may not have limit curves or may have many limit curves.
Moreover, we consider the convergence of curves in $C^{0}$ topology as below.
\begin{df}[{\cite[Definition 3.33]{BEE96}}]
Let $\gamma$ and all sequence curves $\gamma_n$ be defined on the closed interval $[a,b]$. 
Then $\{\gamma_n\}$ is said to
{converge to $\gamma$ in the $C^{0}$ topology on curves} if
$\gamma_n(a)\to \gamma(a)$, $\gamma_n(b)\to \gamma(b)$,
and given any open set $\Gamma$ containing $\gamma$, we have $\gamma_n \subseteq \Gamma$ for sufficiently large $n$.
\end{df}
As is explained in \cite{BEE96}, in any spacetime, there exists a sequence $\{\gamma_n\}$ that has a limit curve $\gamma$, while it doe not converge to $\gamma$ in the $C^0$ topology. 
However, if the spacetime without boundary is strongly causal, then two types of convergence are almost equivalent for sequences of causal curves. 
Additionally, in a strongly causal spacetime $(M,g)$ with timelike boundary, the same results holds, see Lemma \ref{lm: limit_C0}. 

Although the definitions of the causal structure and the Lorentzian distance function are usually defined using piecewise smooth curve, the limit curve of a sequence of piecewise smooth curves may have much lower regularity. Thus we need the following definition about $H^1$-curves, when considering limit curves. 
\begin{df}[{\cite[Definition 2.16]{AFS21}}]
Let $(M,g)$ be a spacetime with timelike boundary and $I \subset \mathbb{R}$ be any interval. 
A continuous curve $\gamma: I \rightarrow M$ is an $H^1$-causal curve if 
\begin{itemize}
    \item in any local coordinates, $\gamma$ restricted to any compact interval $J$ belongs to $H^1(J;\mathbb{R}^n)$; and
    \item its almost everywhere derivative is causal.
\end{itemize}
Here the Sobolev space $H^1(J;\mathbb{R}^n)$ is the set of curves that are absolutely continuous and whose derivatives are $L^2$ integrable, in a compact interval.
\end{df}

One has the following lemmas about the limit curve of a sequence of $H^1$-causal curves. 
\begin{lm}[{\cite[Proposition 2.19 (1)]{AFS21}}]\label{lm: limit_C0}
Suppose that $\{\gamma_n\}$ is a sequence of causal curves defined on $[a,b]$
such that $\gamma_n(a)\to p$, $\gamma_n(b)\to q$.
A causal curve $\gamma\colon[a,b]\to M$, with $\gamma(a)=p$ and $\gamma(b)=q$,
is a limit curve of $\{\gamma_n\}$  if and only if there is a subsequence
$\{\gamma_m\}\subset\{\gamma_n\}$ which converges to $\gamma$ in the $C^{0}$ topology.
\end{lm}
\begin{lm}{\cite[Proposition 2.19 (2)]{AFS21}}\label{lm: limit_H1}
Let $(M,g)$ be a globally hyperbolic spacetime with timelike boundary.
Suppose that $\{p_n\}$ and $\{q_n\}$ are sequences in $M$ converging
to $p$ and $q$ in $M$, respectively, with $p\neq q$, and
$p_n \le q_n$ for each $n$. Let $\gamma_n$ be a future pointing causal curve
from $p_n$ to $q_n$ for each $n$. Then there exists a future pointing causal
limit curve $\gamma$ which joins $p$ to $q$.
\end{lm}


Recall in a spacetime $(M,g)$ with timelike boundary, the causal future $J^+(p)$ and past $J^-(p)$, for some $p \in M$, is defined using piecewise smooth causal curves. 
One may also consider their definitions using $H^1$-causal curves, i.e., $J^+_{H^1}(p)$ and $J^-_{H^1}(p)$. 
\begin{lm}[{\cite[Proposition 2.22]{AFS21}}]\label{lm_causalH1}
Let $(M,g)$ be a globally hyperbolic Lorentzian manifold with timelike boundary (in our conventions, always defined by causal futures and pasts using piecewise smooth casual curves). 
Then $J^\pm(p) = J^\pm_{H^1}(p)$ for all $p \in M$. 
\end{lm}
This leads to the following lemma. 
\begin{lm}\label{lm: H1 curve to piecewise}
Let $(M,g)$ a globally hyperbolic spacetime with timelike boundary. 
Let $x,y \in M$ and $\mu:[0,1] \rightarrow M$ be an $H^1$-causal curve from $x$ to $y$.
Then there exists a piecewise smooth casual curve connecting $x$ and $y$. 
\end{lm}
\begin{proof}
Indeed, with the $H^1$-causal curve $\mu$, we know $y \in J^+_{H^1}(x)$. 
It follows that $y \in J^+(x)$ and therefore there is a piecewise smooth causal curve from $x$ to $y$.

\end{proof}

\section{Exterior time separation function}\label{sec: exterior time separation function}

We prove Theorem \ref{thm: exterior data no cut points} in this section. Recall that the ambient manifold $N_j$ is globally hyperbolic, $K_j$ is the unknown compact region, $\tM_j$ is an open neighborhood of $K_j$, and $K^c_j = \tM_j \backslash K_j$ is the exterior region. The steps are as follows:
\begin{enumerate}
    \item First we determine the metric in $K^c_j$, and find a middle layer $K_j \subset\subset M_j \subset\subset \tM_j$ with piecewise analytic boundary.
    \item Then we use the exterior time separation function to recover the complete travel time data $(S, \tau)$ for lightlike geodesics with respect to $M_j$.
    \item Use the travel time data, we construct a map from $L^+M_1$ and $M_2$, and show that it is actually fiber-preserving and therefore can be upgraded to a bijective function between $M_1$ and $M_2$.
    \item We show that the map can be extended and upgraded to an analytic isometry between $\tM_1$ and $\tM_2$, which agrees with $\varphi_0$ on the exterior region.
\end{enumerate}

We summarize Step (3) and Step (4) into the following proposition, as it will also be used in the proof of Theorem \ref{thm: boundary distance main thm}, Theorem \ref{thm: interior scattering main thm} and Theorem \ref{thm: complete scattering main thm}.
In this proposition, we do not require $M_j$ to be contained in a globally hyperbolic Lorentzian manifold, as assumed in Theorem 1.1. Instead, it suffices to assume that all lightlike geodesics in $M_j$ are non-trapping. 
In the setting of Theorem \ref{thm: exterior data no cut points}, this non-trapping property follows automatically from the global hyperbolicity assumption and compactness of $M_j$.
\begin{pp}\label{prop: complete light scattering recover metric}
    For $j=1,2$, let $(\tM_j, g_j)$ be analytic Lorentzian manifolds of dimension $n \geq 3$. Let $M_j \subset \tM_j$ be a closed subset in the interior with piecewise analytic boundary. Assume the lightlike geodesics are non-trapping in $M_j$. Suppose there exists an isometry $\varphi_0: M^c_1 \sqcup \partial M_1 \to M_2^c \sqcup \partial M_2$, where $M^c_j = \tM_j \backslash M_j$. Assume the lightlike complete travel time data are equivalent, that is
    \[
    (\varphi_0)_* \circ S_1 = S_2 \circ (\varphi_0)_*, \quad \tau_1 = \tau_2 \circ (\varphi_0)_* \quad \text{on } \overline{\partial_-LM_1}.
    \]
    Then there exists an analytic isometry $\varphi: \tM_1 \to \tM_2$ such that $\varphi|_{M^c_1} = \varphi_0$.
\end{pp}

In this section, we will only work with complete travel time data, so we sometimes omit the complete and simply refer to it as the travel time data. Complete travel time data is more natural in the exterior case, because measurements are made outside the unknown region, which requires the information to fully pass through.

\subsection{Setup}\label{sec: setup}
First, we prove the metric in the exterior region is determined in the following lemma.
\begin{lm}\label{lm: exterior metric}
    Let $\varphi_0$ be defined as in Theorem \ref{thm: exterior data no cut points}. Then $\varphi_0$ is an analytic isometry on $K^c_j$.
\end{lm}
\begin{proof}
    We first claim that the for any geodesically convex open set $U$, the time separation function $d_U$ in $U$ determines the metric there. 
    Indeed, for any $x \in U$, 
    the intersection $I^-({x}) \cap U = \{y \in U: d_U(y, x) > 0\}$ is a non-empty open set. 
    Now we pick some $y$ in it. 
    By \cite[Lemma 5]{LOY16}, $d_U(y, \cdot)$ is smooth and $\nabla_g d_U(y, \cdot)|_x$ gives a unit timelike vector at $x$. 
    Computing this for any $y \in I^-(x) \cap U$, we recover all future pointing timelike vectors at $x$ and therefore the metric at $x$.

    Now consider some $x \in K^c_1$. Let $U \subset K^c_1$ be a geodesically convex neighborhood of $x$. 
    By strong causality, there exists a smaller neighborhood $V \ni x$ such that for any $y \in V$, any causal curve connecting $x$ and $y$ are fully contained in $U$. As a result, $d_1(y, x) = d_U(y, x)$ for any $y \in V$, where recall $d_1$ is the time separation function restricted to $K_1^c$.
    Since $\varphi_0$ is a diffeomorphism on $K^c_j$ and the strong causality holds on $K^c_2$ as well, we may assume that $d_2(\varphi_0(y), \varphi_0(x)) = d_{\varphi_0(U)}(\varphi_0(y), \varphi_0(x))$ for any $y \in V$ by shrinking the size of $V$.
    Use the fact that $d_1 = \varphi^*d_2$, one has $I^-(\varphi_0(x)) \cap \varphi_0(V) = \varphi_0(I^-(x) \cap V)$. Then the set of (past pointing) unit timelike vector at $\varphi_0(x)$ is given by
    \begin{align*}
        &\{\nabla_{g_2} d_2(z, \cdot)|_{\varphi_0(x)}: z \in I^-(\varphi_0(x)) \cap \varphi_0(V)\}\\
        = & \{\nabla_{g_2} d_2(\varphi_0(y), \cdot)|_{\varphi_0(x)}: y \in I^-(x) \cap V\}\\
        = & (\varphi_0)_*\{\nabla_{g_1}d_1(y, \cdot)|_x: y \in I^-(x) \cap V\}.
    \end{align*}
    That is, $\varphi_0$ preserves the length of timelike vectors. As a result, $g_1|_{K^c_1} = \varphi_0^*(g_2|_{K^c_2})$. Since metrics are analytic and $\varphi_0$ is an isometry, it is automatically analytic.
\end{proof}

Since $K_j$ may have very rough boundary, we now find a middle layer $M_j$ with piecewise analytic boundary such that $K_j \subset \subset M_j \subset\subset \tM_j$.
\begin{lm}\label{lm: find middle layer}
    There exists $M_j \subset\subset \tM_j$ such that $K_j \subset\subset M_j$ and $M_j$ has piecewise analytic boundary. 
    In particular, we can write $\partial M_j \subset \bigcup_{k=1}^N \partial B_k^j$, where $B_k^j \subset M^c_j$ are compact sets with analytic boundary.
    Moreover, denoting by $M^c_j = \tM_j \backslash M_j$, we have $\varphi_0(M^c_1) = M^c_2$ and $\varphi_0(\partial M_1) = \partial M_2$.
\end{lm}
\begin{proof}
     We start by considering some open neighborhood $U_1$ of $K_1$ such that $K_1 \subset\subset U_1 \subset\subset \tM_1$. For every $x \in U_1$, we may choose local coordinates around $x$ whose domain is contained entirely in $K^c_1$. 
     Denote the closed unit ball in that coordinate by $B_x$. 
     By our construction, each $B_x$ has an analytic boundary. By the compactness, there are finitely many such $B_{x_j}$ covering $\partial U_1$, and we take $M_1 = U_1 \cup (\bigcup_{j=1}^N B_{x_j})$. Now $M_1$ satisfies the desired property. In addition, since $\varphi_0$ is an analytic isometry on $K^c_j$, the set $\varphi_0(B_{x_j})$ also has an analytic boundary. In particular, $M_2 = \tM_2 \backslash \varphi_0(\tM_1 \backslash M_1)$ also satisfies the desired property.
\end{proof}

As $M_j$ has piecewise analytic boundary (here piecewise refers to finitely many pieces), the intersection between the boundary and lightlike geodesics become much simpler.

\begin{lm}\label{lm: finite intersection}
    Let $M_j$ be given by Lemma \ref{lm: find middle layer}.
    Then any causal geodesic only intersects $\partial M_j$ finitely many times. In particular, for any $(x, v) \in \partial L^+M_j$, there exists some $\epsilon > 0$ such that $\gamma_{x, v}((-\epsilon, \epsilon)) \cap \partial M_j = \{x\}$.
\end{lm}
\begin{proof}
    By Lemma \ref{lm: find middle layer}, we may assume $\partial M_j \subset \cup_{k=1}^N B_{k}^j$, where each $B_k^j \subset M^c_j$ is a compact set with analytic boundary. Consider a causal geodesic $\gamma$.
    By the global hyperbolicity, it eventually leaves any compact set. Suppose $\gamma$ intersects with $\partial M_j$ infinitely many times. 
    Then by the pigeonhole principle, it intersects $B_k^j$ infinitely many times for some $k$. By the compactness, their intersection has accumulation point, which forces the geodesic to lie completely inside of $\partial B_k^j$.
    However, this would contradict $\gamma$ leaving any compact set. The second statement immediately follows from finite intersection points.
\end{proof}

\subsection{Lightlike complete travel time data}\label{sec: lightlike complete travel time data}
In this subsection, for convenience we ignore the subscript $j$ and work with $K, M, \tM, N, g$. The goal of this subsection is to explicitly compute the lightlike complete travel time data $(S, \tau)$ for $M$, using $d$ and $g$ on $K^c \sqcup \partial K$. Recall that even though $d(x, y)$ is only given for all $x, y \in K^c$, the time separation function $d(x, y)$ for all $x, y \in K^c \sqcup \partial K$ is actually known via continuity by \cite[Lemma 14.21]{One83}. (Here we used the assumption that $\varphi_0$ extends continuously to $\partial K_j$.)

\begin{lm}\label{lm: exterior recover light cone boundary}
    For any $x \in K^c$, we have
    \[
    \partial(\{ y \in K^c: d(x, y) > 0 \}) \cap K^c = \partial J^+(x) \cap K^c.
    \]
\end{lm}
\begin{proof}
    On a globally hyperbolic spacetime we have
    \begin{align*}
        \partial (\{y \in K^c: d(x, y) > 0\}) \cap K^c &= \partial(I^+(x) \cap K^c) \cap K^c \\
        &= \Big((\partial I^+(x) \cap \overline{K^c}) \cup (\partial K^c \cap \overline{I^+(x)})\Big) \cap K^c\\
        &= \partial J^+(x) \cap K^c.
    \end{align*}
\end{proof}

\begin{lm}\label{lm: determine geodesic before cut}
    For every $x \in K^c$ and $v \in L^+_x\tM$, 
    consider some $x' = \gamma_{x,v}(\epsilon) \in K^c$ in a geodesically convex neighborhood of $x$. Then
    \[
    \partial J^+(x) \cap \partial J^+(x') \cap K^c = \gamma_{x,v}([\epsilon, \rho(x, v)]) \cap K^c.
    \]
\end{lm}
\begin{proof}
    Indeed, for $\epsilon \leq t \leq \rho(x, v)$, then $d(x, \gamma_{x, v}(t)) = d(x', \gamma_{x, v}(t)) = 0$. 
    As $\gamma_{x, v}(t)$ is contained in $J^+(x)$ and $J^+(x')$, this implies $\gamma_{x, v}(t) \in \partial J^+(x) \cap \partial J^+(x')$, by the global hyperbolicity of the spacetime. On the other hand, if $y \in \partial J^+(x) \cap \partial J^+(x')$, then there exists $v' \in L^+_{x'}\tM$ such that $\gamma_{x', v'}$ passes through $y$. If $v' \neq \dot{\gamma}_{x,v}(\epsilon)$ up to rescaling, then $x$ and $y$ are connected by $\gamma_{x, v}([0, \epsilon]) \circ \gamma_{x',v'}$, which is a causal path that is not a pregeodesic. By Lemma \ref{lm: boundary of causal future}, $d(x, y) > 0$, contradicting $y \in \partial J^+(x)$. Thus $y = \gamma_{x, v}(t)$ for some $t \geq \epsilon$. Again use $y \in \partial J^+(x)$ to deduce $d(x, y) = 0$, so $\epsilon \leq t \leq \rho(x, v)$.
\end{proof}

\begin{lm}\label{lm: lower bound for rho}
    Let $\gamma:[0, T] \to \tM$ be a lightlike geodesic. Then $\rho(\gamma(t), \dot{\gamma}(t))$ has a uniform positive lower bound for all $t \in [0, T]$.
\end{lm}
\begin{proof}
    Since $\tM$ is globally hyperbolic, $\rho$ is lower semi-continuous. Clearly $\rho$ is always positive, then on a compact set it has positive lower bound.
\end{proof}

We now start to recover the lightlike complete travel time data for $M$. Specifically, for any lightlike geodesic $\gamma$ in $M$, $\gamma \cap K^c$ and its travel time data are computed step by step. We show the procedure terminates after finitely many steps by proving that there is a uniform lower bound for every two steps. See Figure \ref{fig: 2step procedure} for the different cases of the two step recovery appeared in the proof.

\begin{figure}[ht]
  \centering
  \begin{subfigure}{0.48\textwidth}
    \centering
    \def\svgwidth{0.8\linewidth}
    \def\svgwidth{\columnwidth}
    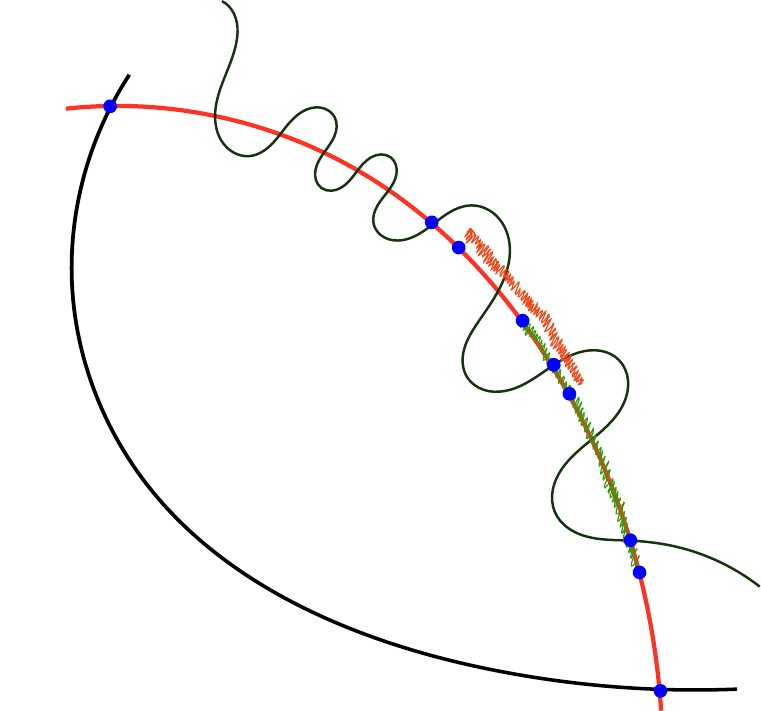

  \end{subfigure}
  \hfill
  \begin{subfigure}{0.48\textwidth}
    \centering
    \def\svgwidth{0.8\linewidth}
    \def\svgwidth{\columnwidth}
    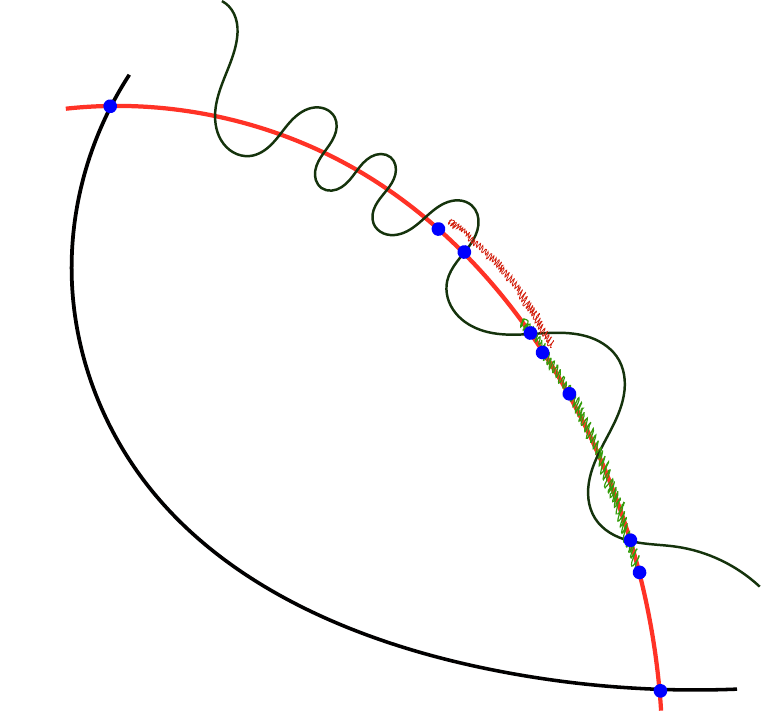

  \end{subfigure}
  \caption{The first step goes from $t_1$ to $t_2'$, but there are two cases based on whether $\gamma(s_1)$ is in $K$ (left graph) or $K^c$ (right graph). The second step goes to $t_3'$, which can be shown is strictly larger than $s_1$ in both cases. Since $s_1 - t_1$ has a uniform lower bound, this shows for every two steps we gain a uniform lower bound for the recovery.}
    \label{fig: 2step procedure}
\end{figure}

\begin{pp}\label{prop: lightlike travel time data}
    The lightlike complete travel time data $(S, \tau)$ for $M$ can be computed. In particular, this means for any $(x, v) \in \partial L^+M_1$:
    \begin{equation}
        (S_1, \tau_1)(x, v) = (y, w; T) \iff (S_2, \tau_2)(\varphi_0(x), (\varphi_0)_*v) = (\varphi_0(y), (\varphi_0)_*w; T).
    \end{equation}
\end{pp}
\begin{proof}
    Pick some $(x, v) \in \partial L^+M$, denote by $\gamma$ the corresponding lightlike geodesic. By Lemma \ref{lm: finite intersection}, we know $\gamma$ intersects with $\partial M$ only at finitely many points. Suppose $\gamma([0, T]) \subset M$ and $\gamma((T, T+\epsilon)) \subset M^c$. Then $(S, \tau)(x, v) = (y, w)$ where $(y, w) = (\gamma(T), \dot{\gamma}(T))$ and $(S, \tau)$ is the complete travel time data with respect to $M$. The goal is to determine $(y, w; T)$.

    Since $g|_{K^c \sqcup \partial K}$ is given, $(y, w; T)$ is directly recovered if $\gamma$ never entered $K$ before exiting $M$, so suppose $\gamma$ intersects with $K$ before leaving $M$. Then there exists $0 < t_1 < T$ such that $\gamma(t_1) \in \partial K$ and $\gamma([0, t_1)) \subset K^c$. Certainly we can determine $t_1$, since $g|_{K^c \sqcup \partial K}$ is given. We show that $\gamma \cap K^c$ can be recovered step by step, where for each step we can recover the travel time along $\gamma$, and $T$ will be reached in finitely many steps.

    Denote $i_0$ the uniform positive lower bound for $\rho(\gamma(t), \dot{\gamma}(t))$ for $t \in [0, T]$ from Lemma \ref{lm: lower bound for rho}. Pick sufficiently small $0<\epsilon< i_0/2$. Denote $z_1 = \gamma(t_1-\epsilon)$ and $z_1' = \gamma(t_1 - \epsilon/2)$, then both are in $K^c$. By Lemma \ref{lm: exterior recover light cone boundary} and Lemma \ref{lm: determine geodesic before cut}, $\gamma([t_1 - \epsilon/2, s_1]) \cap K^c$ can be determined where $s_1 - (t_1 - \epsilon) = \rho(\gamma(t_1 -\epsilon), \dot{\gamma}(t_1 -\epsilon))$. In fact, $\gamma([t_1 - \epsilon/2, s_1)) \cap K^c$ can also be determined. Indeed, $\gamma([t_1 - \epsilon/2, s_1]) \cap K^c$ is union of line segments, $\gamma(s_1) \in K^c$ if and only if one of the segment has an end point in $K^c$ that is not $\gamma(t_1-\epsilon/2)$. The benefit of removing $\gamma(s_1)$ is that, for any $\gamma(t) \in \gamma([t_1 - \epsilon/2, s_1)) \cap K^c$, the first cut point for $z_1$ comes after $\gamma(t)$. As a result, $\gamma$ is the unique causal geodesic connecting $z_1$ and $\gamma(t)$.

    We now recover $t$ and $\dot{\gamma}(t)$ for any $\gamma(t) \in \gamma([t_1 - \epsilon/2, s_1)) \cap K^c$. Since $\gamma(t) \in K^c$, we can pick a sequence of $y_j \in I^+(\gamma(t)) \cap K^c$ converging to $\gamma(t)$. Then $d(z_1, y_j) > 0$. We claim that $d(z_1, \cdot) = |\exp_{z_1}^{-1}(\cdot)|_g$ locally around $y_j$ and $d(\cdot, y_j) = |\exp_{y_j}^{-1}(\cdot)|_g$ locally around $z_1$, provided $j$ is sufficiently large. Indeed, this comes from the fact that $\gamma(t)$ is before the first cut point for $z_1$. By lower semi-continuity of $\rho$ and the fact that there is no conjugate point before the first cut point, $\exp_{z_1}$ and $\exp_{y_j}$ are local diffeomorphisms around $y_j$ and $z_1$, respectively. Moreover, the time separation function is positive around $z_1$ and $y_j$ because they are connected by a unique distance maximizing timelike geodesic.

    Similar computation as \cite[Lemma 5]{LOY16} gives
    \[
    v_j := d(z_1, y_j) \cdot \nabla_g d(z_1, \cdot)|_{y_j}, \quad u_j := -d(z_1, y_j) \cdot \nabla_g d(\cdot, y_j)|_{z_1}
    \]
    satisfy $\exp_{z_1}(u_j) = y_j$ and $d\exp_{z_1}|_{u_j}u_j = v_j$. Since $y_j \to \gamma(t)$, denote $u = \lim u_j$ and $v = \lim v_j$, we must have $\gamma(t) = \exp_{z_1}(u)$ and $v = d\exp_{z_1}|_u u$. So the travel time data from $z_1$ to $\gamma(t)$ along $\gamma$ is $(z_1, u; \gamma(t), v; 1)$, which is equivalent to $(\gamma(t_1-\epsilon), \dot{\gamma}(t_1-\epsilon); \gamma(t), \dot{\gamma}(t); t-(t_1-\epsilon))$. As a result, there exists $k > 0$ such that $u = k \dot{\gamma}(t_1-\epsilon)$, so $t = (t_1-\epsilon)+k$ and $\dot{\gamma}(t) = \frac{1}{k}v$ are recovered.
    
    As $t$ and $\dot{\gamma}(t)$ can be recovered for any $\gamma(t) \in \gamma([t_1-\epsilon/2, s_1))\cap K^c$, we are done if $\gamma([t_1-\epsilon/2, s_1))\cap M^c \neq \emptyset$. Otherwise it lies in $M$, we denote the supremum of all such $t$ by $t_2'$, and $\dot{\gamma}(t_2')$ can be computed by taking limit. Since we know $t_2'$, $\gamma(t_2')$ and $\dot{\gamma}(t_2')$, we keep extending $\gamma$ in $K^c$ until one of the following two things happen: either it leaves $M$, in which case we are done since this means $T$, $\gamma(T)$ and $\dot{\gamma}(T)$ are recovered; or it hits $\partial K$ again, say at time $t_2 \geq t_2'$. Again pick $0< \epsilon < i_0/2$ sufficiently small, the exact same argument shows that one can recover the travel time data on $\gamma([t_2-\epsilon/2, s_2)) \cap K^c$ where $s_2 - (t_2-\epsilon) = \rho(\gamma(t_2-\epsilon), \dot{\gamma}(t_2-\epsilon))$. Denote $t_3'$ the supremum of all $t$ such that $\gamma(t) \in \gamma([t_2-\epsilon/2, s_2)) \cap K^c$. We claim that $t_3' > s_1$. Note that then $t_3' - t_1 > s_1 - (t_1-\epsilon) - \epsilon > i_0/2$ is bounded below.

    To prove the claim, consider two cases: $\gamma(s_1) \in K^c$ or $\gamma(s_1) \notin K^c$. If $\gamma(s_1) \in K^c$, then $t_2' = s_1$ and $t_2 > t_2'$. Since $\epsilon < i_0/2$, $\gamma(t_2-\delta) \in \gamma([t_2-\epsilon/2, s_2)) \cap K^c$ for all sufficiently small $\delta$. In particular, this means $t_3' \geq t_2 > t_2' = s_1$. On the other hand, if $\gamma(s_1) \notin K^c$, then $\gamma([t_2', s_1]) \subset K$, hence $t_2 = t_2'$. Let $s'$ be the first time $\gamma$ leaves $K$ after $t_2'$, clearly $s' \geq s_1$. By the assumption that lightlike geodesic in $M$ does not have cut point, we know $\rho(\gamma(t_2), \dot{\gamma}(t_2)) > s'-t_2$. In particular, there exists $\delta > 0$ sufficiently small such that $\gamma(s'+\delta) \in K^c$ and $s'+\delta-t_2 < \rho(\gamma(t_2), \dot{\gamma}(t_2))$. Thus $t_3' \geq s' + \delta > s_1$.

    Thus we have proved that, after two steps, the amount of recovery of the travel time along $\gamma$ is $t_3'-t_1$, which is bounded below uniformly by $i_0/2$. We repeat the two-step procedure, by the uniform lower bound of the two-step size, $T$, $\gamma(T)$ and $\dot{\gamma}(T)$ can be recovered in finitely many steps. As $(x, v) \in \overline{\partial_-L^+M}$ is arbitrarily chosen, the entire lightlike complete travel time data for $M$ is recovered.
\end{proof}

Thus we have proved $M_1$ and $M_2$ admit equivalent lightlike complete travel time data (identified via $\varphi_0$). It now suffices to prove Proposition \ref{prop: complete light scattering recover metric} in Section \ref{sec: construction of bijection} and Section \ref{sec: analytic isometry}.

\subsection{Construction of bijection}\label{sec: construction of bijection}

As explained at the beginning of Section \ref{sec: exterior time separation function}, we now start the construction of the extension. The construction is mostly similar to \cite{Var09}, the difference being we only have lightlike directions and the fact that Lorentzian distance function is not a true distance. Recall that in the assumption of Theorem \ref{thm: exterior data no cut points}, the manifolds are time-orientable, so $L^+M_j$ is well-defined. 

For any $(x, v) \in L^+M_1$, for clarity we denote $\gamma^1_{x, -v}$ as the unique past pointing lightlike geodesic, where the superscript $1$ indicates it is a geodesic in $\tM_1$. Let
\[
t(x, v) = \sup\{t > 0: \gamma^1_{x, -v}([0, t]) \subset M_1\}
\]
be the first exit time for the past pointing lightlike geodesic $\gamma^1_{x,-v}$. Denote
\[
(y, w) = (\gamma^1_{x, -v}(t(x, v)), -\dot{\gamma}^1_{x, -v}(t(x, v))).
\]
Define
\[
\varphi: L^+M_1 \to M_2, \quad (x, v) \mapsto \exp^{g_2}_{\varphi_0(y)}(t(x, v)(\varphi_0)_*w).
\]
\begin{lm}\label{lm: well-defined}
    The map $\varphi$ is well-defined.
\end{lm}
\begin{proof}
    By non-trappingness of null-geodesics, $t(x, v)$ is finite. By Proposition \ref{prop: lightlike travel time data}, there exists travel time data $((y, w); (z, u); T)$ on $\tM_1$ such that
    \begin{enumerate}
        \item $T \geq t(x, v)$;
        \item it is the travel time data corresponding to $\gamma^1_{x,v}$;
        \item and $((\varphi_0(y), (\varphi_0)_*w); (\varphi_0(z), (\varphi_0)_*u); T)$ is a travel time data on $\tM_2$.
    \end{enumerate}
    As a result $\exp_{\varphi_0}^{g_2}(t(x, v)(\varphi_0)_*w) \in M_2$.
\end{proof}
Next we show that the map $\varphi(x, \cdot)$ is locally constant for any fixed $x \in M_1$. We first recall a technical lemma for analytic Riemannian metric.
\begin{lm}[Lemma 1, \cite{Var09}]\label{lm: riemannian analytic}
    Let $g^R$ be an analytic Riemannian metric on $\tilde{M}$. If $K$ is a compact subset contained in the subset of $\tilde{M}$, then there is an open $O \subset \tilde{M}$ containing $K$ and a positive number $r$ such that the squared distance function of $g^R$ is analytic on the set
    \[
    \Delta_{O,r}(K) = \{(x, y): x \in K, d_R(x, y) < r\}.
    \]
\end{lm}


\begin{figure}
    \centering
    \begin{subfigure}{0.3\textwidth}
    \centering
    \def\svgwidth{0.9\linewidth}
    \def\svgwidth{\columnwidth}
    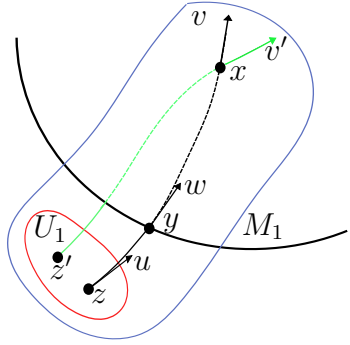

  \end{subfigure}  
    \caption{Locally consider Fermi coordinates around a fixed geodesic, and construct the analytic Riemannian metric $g^R$ there. Around $z$ and $\varphi_0(z)$, the metrics agree (preserved by $\varphi_0$).}
    \label{fig: fermi coord}
\end{figure}

\begin{pp}\label{prop: locally constant}
    For any $(x, v) \in L^+M_1$, $\varphi(x, \cdot)$ is locally constant. Hence $\varphi(x, \cdot)$ is a constant function. 
\end{pp}
\begin{proof}
    Fix some $(x, v) \in L^+M_1$, for notation simplicity we denote $\gamma^1_{x, v}$ by $\gamma^1$, then $(\gamma^1(0), \dot{\gamma}^1(0)) = (x, v)$. By definition of $t(x, v)$ and Lemma \ref{lm: finite intersection}, there exists some $\epsilon>0$ sufficiently small such that $\gamma^1((-t(x, v)-\epsilon, -t(x, v))) \subset M^c_1$. Denote $s = t(x, v)+\epsilon$ and $(z, u) : =(\gamma^1(-s), \dot{\gamma}^1(-s))$. 

    Consider now a Fermi coordinate around $\gamma^1$ given by the following procedures. Let $e_0 = u$, $e_i \in T_z\tilde{M}_1$ for $i = 1,\dots,n-1$ be such that $e_0,\dots,e_{n-1}$ are linearly independent, and consider their parallel transport $E_i$ along $\gamma^1$. The Fermi coordinate in a neighborhood $W_1$ of $\gamma^1$ is given by
    \[
    (r^0, \cdots, r^{n-1}) \mapsto \exp^{g_1}_{\gamma^1(-s+r^0)}(\sum_{i=1}^{n-1}r^iE_i).
    \]
    One may choose $e_i$ such that the metric $g_1$ along $\gamma^1$ is given by
    \[
    g_1|_{\gamma^1} = 2dr^0\otimes dr^1 + \sum_{j,k=2}^{n-1}\theta_{jk}dr^j\otimes dr^k,
    \]
    see for example \cite[Section 4]{FO22}. Let us assume that the Fermi coordinate is constructed on $\gamma^1([-s-\delta, \delta])$ for some sufficiently small $\delta$, so that $\gamma^1([-s, 0])$ is in the interior of $W_1$. See Figure \ref{fig: fermi coord}.

    We now proceed to construct a corresponding Fermi coordinate in $\tilde{M}_2$ along $\gamma^2$, where
    \[
    \gamma^2(-s+t) = \exp^{g_2}_{\varphi_0(z)}(t(\varphi_0)_*u).
    \]
    First of all, $\gamma^2$ is well-defined on $(-s-\delta, \delta)$ for small $\delta$, since 
    \begin{align}
        &\gamma^2([-s,-t(x, v)]) = \varphi_0(\gamma^1([-s,-t(x, v)])) \subset \tM_2,\\
        &\gamma^2(0) = \exp^{g_2}_{\varphi_0(z)}(s(\varphi_0)_*u) = \exp^{g_2}_{\varphi_0(y)}(t(x, v)(\varphi_0)_*w) = \varphi(x, v) \in M_2,\label{eq: gamma_2}
    \end{align}
    where $(y, w) = (\gamma^1(-t(x, v)), -\dot{\gamma}^1(-t(x, v)))$. We then construct the Fermi coordinate using the parallel transport of $(\varphi_0)_*e_i \in T_{\varphi_0(z)}\tilde{M}_2$, denoted by $E_i'$. Note that $\varphi_0$ being isometry on $M^c_j$ implies for all $t$ around $-s$, $\varphi_0(\gamma^1(t)) = \gamma^2(t)$ and $(\varphi_0)_*E_i(t) = E_i'(t)$. The Fermi coordinate in a neighborhood $W_2$ around $\gamma^2([-s-\delta, \delta])$ is thus given by
    \[
    (r^0, \cdots, r^{n-1}) \mapsto \exp^{g_2}_{\gamma^2(-s+r^0)}(\sum_{i=1}^{n-1}r^iE_i').
    \]
    By the construction of $E_i$ and $E_i'$, we have
    \[
    g_2|_{\gamma^2} = 2dr^0\otimes dr^1 + \sum_{j,k=2}^{n-1}\theta'_{jk}dr^j\otimes dr^k.
    \]

    For $j = 1, 2$, define $X_j = \partial_{r^0} - \partial_{r^1}$ in the corresponding Fermi coordinates. Then by shrinking $W_1$ and $W_2$, $X_1$ and $X_2$ are both analytic timelike vector fields. Moreover, there exists $U_1 \subset W_1 \cap M^c_1$ neighborhood of $z$ and $U_2 \subset W_2\cap M^c_2$ neighborhood of $\varphi_0(z)$ such that $\varphi_0(U_1) = U_2$ and $(\varphi_0)_*(X_1|_{U_1}) = X_2|_{U_2}$. Consider the analytic Riemannian metric on $W_j$ given by
    \[
    g^R_j = g_j -2(g_j(X_j, X_j))^{-1}X_j^\flat \otimes X_j^\flat,
    \]
    we thus have $g^R_1|_{U_1} = \varphi_0^*(g^R_2|_{U_2})$, denote $d_{R,j}$ the corresponding distance function on $W_j$.
    
    We are now ready to show $\varphi(x, \cdot)$ is constant around $v$. By continuity there exists a sufficiently small neighborhood $\mathcal{U}$ of $v$ in $L^+_xM_1$ such that for all $v' \in \mathcal{U}$,
    \[
    \exp^{g_1}_x(-sv') \in U_1, \quad \{\exp_x^{g_1}(tv'): t\in[-s,0]\} \subset W_1.
    \]
    Denote $(z', u') := (\exp_x^{g_1}(-sv'), d\exp_x^{g_1}|_{-sv'}v')$. Define
    \begin{align*}
        &\rho_1(t) = d_{R,1}^2(\exp^{g_1}_z(tu), \exp^{g_1}_{z'}(tu')),\\
        &\rho_2(t) = d_{R,2}^2(\exp^{g_2}_{\varphi_0(z)}(t(\varphi_0)_*u), \exp^{g_2}_{\varphi_0(z')}(t(\varphi_0)_*u')).
    \end{align*}
    By Lemma \ref{lm: riemannian analytic}, $\rho_1$ and $\rho_2$ are well-defined analytic function on $[0, s]$, provided that $\mathcal{U}$ is sufficiently small. Since $\varphi_0$ preserves both Riemannian and Lorentzian metrics on $U_j$, $\rho_1(t) = \rho_2(t)$ for all $t$ sufficiently small. Hence $\rho_1 = \rho_2$ on $[0, s]$. By definition of $\rho_1$, $\rho_2(s) = \rho_1(s) = 0$; together with \eqref{eq: gamma_2}, we have
    \[
    \varphi(x, v) = \exp^{g_2}_{\varphi_0(z)}(s(\varphi_0)_*u) = \exp^{g_2}_{\varphi_0(z')}(s(\varphi_0)_*u').
    \]
    It now suffices to show
    \[
    \exp^{g_2}_{\varphi_0(z')}(s(\varphi_0)_*u') = \varphi(x, v').
    \]
    
    This follows immediately from the same argument as in \cite[Lemma 1]{Var09}, we include here for completeness. By Lemma \ref{lm: finite intersection} every lightlike geodesic only intersects $\partial M_j$ finitely many times, so the set of lightlike travel time data with respect to $\tilde{\gamma}^2(t) = \exp_{\varphi_0(z')}^{g_2}(t(\varphi_0)_*u')$ is given by
    \[
    ((\tilde{\gamma}^2(s_k), \dot{\tilde{\gamma}}^2(s_k));(\tilde{\gamma}^2(s_k+T_k), \dot{\tilde{\gamma}}^2(s_k+T_k)); T_k), \quad k = 1 \dots N
    \]
    with $0<s_1 \leq s_1+T_1 <  s_2 \leq s_2+T_2 < \dots < s_N \leq s \leq s_N+T_N$ and
    \[
    \tilde{\gamma}^2((0, s_1)) \subset M^c_2, \quad \tilde{\gamma}^2([s_k, s_k+T_k]) \subset M_2, \quad \tilde{\gamma}^2((s_k+T_k, s_{k+1})) \subset M^c_2.
    \]
    Denote $\tilde{\gamma}^1(t) = \exp^{g_1}_{z'}(tu')$. By Proposition \ref{prop: lightlike travel time data} and the fact that $\varphi_0$ preserves metrics on $M^c_j$, a simple induction shows that
    \[
    \left(\varphi_0(\tilde{\gamma}^1(t)), (\varphi_0)_*(\dot{\tilde{\gamma}}^1(t))\right) = \left(\tilde{\gamma}^2(t), \dot{\tilde{\gamma}}^2(t)\right), \quad \forall t = s_k, s_k+T_k.
    \]
    Denote $(y', w') = (\tilde{\gamma}^1(s_N), \dot{\tilde{\gamma}}^1(s_N))$ the entering point for $(x, v')$, we know $t(x, v') = s - s_N$. This implies
    \begin{align*}
        \exp^{g_2}_{\varphi_0(z')}(s(\varphi_0)_*u') &= \tilde{\gamma}_2(s_N+(s-s_N)) \\
        &= \exp^{g_2}_{\varphi_0(\tilde{\gamma}^1(s_N))}((s-s_N)(\varphi_0)_*\dot{\tilde{\gamma}}^1(s_N)) \\ 
        &=\exp^{g_2}_{\varphi_0(y')}(t(x, v')(\varphi_0)_*w') \\
        &= \varphi(x, v').
    \end{align*}
    Thus $\varphi(x, \cdot)$ is locally constant. For Lorentzian manifold of dimension at least 3 the set of future pointing lightlike directions is connected, so we can define $\varphi(x) = \varphi(x, \cdot)$.
\end{proof}

Next we show that it is a bijection.
\begin{pp}\label{prop: bijection}
    $\varphi$ is a bijection from $M_1$ to $M_2$.
\end{pp}
\begin{proof}
    Considering the same construction of $\psi: L^+M_2 \to M_1$. Proposition \ref{prop: locally constant} shows that $\psi: M_2 \to M_1$ is also well-defined. For any $(x, v) \in L^+M_1$, by Proposition \ref{prop: lightlike travel time data}, $\exp^{g_2}_{\varphi_0(y)}(t(\varphi_0)_*w) \in M_2$ for all $0 \leq t \leq t(x, v)$ where $(y, w)$ is the first entry point for $(x, v)$. Then
    \[
    \psi(\varphi(x)) = \psi(\exp^{g_2}_{\varphi_0(y)}(t(x, v)(\varphi_0)_*w)) = \exp^{g_1}_y(t(x, v)w) = x.
    \]
    So $\varphi$ is a bijection with $\varphi^{-1} = \psi$.
\end{proof}

\subsection{Analytic isometry}\label{sec: analytic isometry}
Define $\tilde{\varphi}(x) = \varphi_0(x)$ if $x \in M^c_1$, and $\tilde{\varphi}(x) = \varphi(x)$ if $x \in M_1$.
\begin{pp}\label{prop: preserve causality}
    $\gamma$ is a future pointing lightlike geodesic in $\tilde{M}_1$ if and only if $\tilde{\varphi}(\gamma)$ is a future pointing lightlike geodesic in $\tilde{M}_2$. In particular, $\tilde{\varphi}$ and $\tilde{\varphi}^{-1}$ preserves causality.
\end{pp}
\begin{proof}
    Let $\gamma$ be a (connected) future pointing lightlike geodesic in $\tilde{M}_1$, by extending it we might as well assume $\gamma(0) \notin M_1$. By Lemma \ref{lm: finite intersection} it can be decomposed into finitely many segments: $\gamma((s_k, s_{k+1})) \subset M^c_1$ for $k$ odd, and $\gamma([s_k, s_{k+1}]) \subset M_1$ for $k$ even. Denote $(y_k, w_k) = (\gamma(s_k), \dot{\gamma}(s_k))$. Since $\varphi_0$ preserves metric on $M^c_j$, for all $k$ odd and $t \in (0, s_{k+1}-s_k)$, $\tilde{\varphi}(\gamma(s_k+t)) = \exp_{\varphi_0(y_k)}^{g_2}(t(\varphi_0)_*w_k)$ is a future pointing lightlike geodesic. By definition of $\varphi(x, v)$, for all $k$ even, $\tilde{\varphi}(\gamma(s_k+t)) = \exp^{g_2}_{\varphi_0(y_k)}(t(\varphi_0)_*w_k)$ where $t \in [0, s_{k+1}-s_k]$. So $\tilde{\varphi}(\gamma)$ is a piecewise future pointing lightlike geodesic. In particular, Proposition \ref{prop: lightlike travel time data} and a simple induction shows that for all $t$,
    \begin{equation}\label{eq: tilde varphi}
        \tilde{\varphi}(\gamma(t)) = \exp^{g_2}_{\varphi_0(\gamma(0))}(t(\varphi_0)_*\dot{\gamma}(0)).
    \end{equation}
    The proof for $\tilde{\varphi}^{-1}$ is the same, and so they preserve causality.
\end{proof}

\begin{pp}
    $\tilde{\varphi}$ is an analytic isometry.
\end{pp}
\begin{proof}
    By Proposition \ref{prop: preserve causality} and \cite[Lemma 19]{Haw14}, $\tilde{\varphi}$ is a conformal diffeomorphism. Then $g_1 = f\tilde{\varphi}^*g_2$ on $\tilde{M}_1$ with positive $f \in C^\infty(\tilde{M}_1)$ (note that so far we do not know if $f$ is analytic). We provide several different proofs from here.

    Since the dimension is at least 3, $g_1 = f\varphi^*g_2$ provides a system of $\frac{n(n+1)}{2}$ equations for $n+1$ unknowns (namely, $\partial_j \varphi^k$ and $f$ in local coordinates), and $g_1$ and $g_2$ are both analytic. By Cauchy-Kowalevski Theorem (see for example \cite[Chapter 1D]{Fol95}), the smooth solution $\varphi$ and $f$ must be analytic. As $f \equiv 1$ on $M^c_1$, it must be 1 everywhere, and the theorem is proved.
    
    We can also prove $f \equiv 1$ using the special form of $\tilde{\varphi}$ in \eqref{eq: tilde varphi}, without using any analyticity assumption. Let $\gamma^1(t)$ be a lightlike geodesic in $\tilde{M}_1$ such that $\gamma^1(0) \in M^c_1$. Then by \eqref{eq: tilde varphi}, $\gamma^1(t) = \tilde{\varphi}^{-1}(\gamma^2(t))$ for some lightlike geodesic $\gamma^2(t)$ in $\tilde{M}_2$. On the other hand, $\tilde{\varphi}^{-1}(\gamma^2(t))$ is also a lightlike geodesic for $\tilde{\varphi}^*g_2$ on $\tilde{M}_1$. Since $g_1 = f\tilde{\varphi}^*g_2$, $\tilde{\varphi}^{-1}(\gamma^2(s(t)))$ is a lightlike geodesic for $g_1$ (see for example \cite[Section 3.2]{Ste25}) with
    \[
    \dot{s} = \frac{1}{f(\tilde{\varphi}^{-1}(\gamma^{2}(s)))}, \quad s(0) = 0.
    \]
    Note that $\gamma^1(t) = \tilde{\varphi}^{-1}(\gamma^2(t))$ and $\tilde{\varphi}^{-1}(\gamma^2(s(t)))$ are both geodesics for $g_1$, and
    \[
    \tilde{\varphi}^{-1}(\gamma^2(s(0))) = \gamma^1(0), \quad \left.\frac{d}{dt}\right|_0\tilde{\varphi}^{-1}(\gamma^2(s(t))) = \frac{1}{f(\gamma^1(0))}\tilde{\varphi}_*^{-1}(\dot{\gamma}^2(0)) = \dot{\gamma}^1(0),
    \]
    where we used the fact that $\gamma^1(0) \in M^c_1$ and $\tilde{\varphi}$ is an isometry between $M^c_j$ so $f|_{M^c_1} \equiv 1$. This means they must be the same geodesic, so
    \[
    \tilde{\varphi}^{-1}(\gamma^2(t)) = \gamma^1(t) = \tilde{\varphi}^{-1}(\gamma^2(s(t))).
    \]
    Then $s(t) = t$ for all $t$, so $f \equiv 1$ on $\gamma^1$. The set of all lightlike geodesics cover the entire $\tilde{M}_1$, so $f \equiv 1$ on $\tilde{M}_1$, $g_1 = \tilde{\varphi}^*g_2$. Since both $g_1$ and $g_2$ are analytic, $\tilde{\varphi}$ is analytic.

    Finally, given the special form of $\tilde{\varphi}$ in \eqref{eq: tilde varphi}, we can modify the proof in \cite[Lemma 19]{Haw14} to show $\tilde{\varphi}$ is an analytic isometry directly. Consider some $p \in \tilde{M}_1$, locally choose an analytic coordinate system $x^1, \cdots, x^n$, we may assume $\partial_{x^1}$ is timelike and $\partial_{x^2}, \cdots, \partial_{x^n}$ are spacelike. Then for $j = 2,\cdots, n$, define
    \[
    a_j = -g_{1j}+\sqrt{g_{1j}^2-g_{11}g_{jj}}
    \]
    which is analytic because $g_{1j}^2-g_{11}g_{jj}$ is uniformly bounded below by a positive number ($g_{11} < 0$ and $g_{jj} > 0$). Then $V_j = a_j\partial_{x^1}+\partial_{x^j}$ are linearly independent analytic lightlike vector fields. Similarly, $V_1 = a_1\partial_{x^1}-\partial_{x^2}$ where
    \[
    a_1 = g_{12}+\sqrt{g_{12}^2-g_{11}g_{22}}
    \]
    is lightlike and $V_1, \cdots, V_n$ form an analytic basis for $T\tilde{M}$ around $p$. By Inverse Function Theorem and analyticity of $g_1$, locally
    \[
    (t_1, \cdots, t_n) \mapsto \exp^{g_1}(t_nV_n) \circ \cdots \circ \exp^{g_1}(t_1V_1)(p)
    \]
    is an analytic diffeomorphism from a neighborhood of 0 in $\mathbb{R}^n$ to neighborhood of $p$ in $\tilde{M}_1$, so it is an analytic chart. Let $W_j = \tilde{\varphi}_*V_j$, by \eqref{eq: tilde varphi} they are linearly independent lightlike directions in $\tilde{M}_2$, so around $\tilde{\varphi}(p)$ we have an analytic chart
    \[
    (t_1, \cdots, t_n) \mapsto \exp^{g_2}(t_nW_n) \circ \cdots \circ \exp^{g_2}(t_1W_1)(\tilde{\varphi}(p)).
    \]
    With respect to these charts, $\tilde{\varphi}$ is simply the identity map between neighborhoods of 0 by \eqref{eq: tilde varphi}, so it is an analytic map. Thus $g_1$ and $\tilde{\varphi}^*g_2$ are both analytic metric on $\tM_1$ that agrees on $M^c_1$, meaning they are the same everywhere.
\end{proof}

\begin{remark}
    We emphasize that, as shown in the proof, the form of $\tilde{\varphi}$ in \eqref{eq: tilde varphi} directly implies isometry, without referring to analyticity. Hence if two smooth Lorentzian manifolds admit a bijective causal preserving map that also preserves parametrization of lightlike geodesics, then they are isometric. 
\end{remark}

Thus the proof of Theorem \ref{thm: exterior data no cut points} can be summarized as follows.
\begin{proof}[Proof of Theorem \ref{thm: exterior data no cut points}:]
    By Lemma \ref{lm: exterior metric}, the metric on $K^c$ is determined. By Lemma \ref{lm: find middle layer} and Proposition \ref{prop: lightlike travel time data}, one can find a middle layer $K \subset\subset M \subset\subset \tM$ with piecewise analytic boundary, such that the lightlike complete travel time data for $M$ are determined. By Proposition \ref{prop: complete light scattering recover metric}, $(\tM, g)$ is determined.
\end{proof}

\subsection{Non-time-orientable case}
Indeed, Section \ref{sec: construction of bijection} and Section \ref{sec: analytic isometry} have proved Proposition \ref{prop: complete light scattering recover metric} for spacetime. When the manifold is non-time-orientable, we only need a small modification.
\begin{proof}[Proof of Proposition \ref{prop: complete light scattering recover metric}:]
    Section \ref{sec: construction of bijection} and Section \ref{sec: analytic isometry} gave the proof for when the manifolds are time-orientable. Suppose now that they are not time-orientable, then $L^+M_j$ is not well-defined. Instead, since the lightlike complete travel time data are equivalent, we can use the same definition for $\varphi$ with domain being the entire $LM_1$. Lemma \ref{lm: well-defined} and Proposition \ref{prop: locally constant} still proves $\varphi$ is well-defined and locally constant, but this time $L_xM_1$ contains two connected components, so we need to show the two agree.

    Consider some $(x, v) \in LM$, again denote \[
    (y, w) = (\gamma^1_{x, -v}(t(x, v)), -\dot{\gamma}^1_{x, -v}(t(x, v)))
    \]
    where $t(x, v)$ is the first time the backward geodesic $\gamma_{x, -v}$ leaves $M$. If we denote
    \[
    (z, u) = (\gamma^1_{x, v}(t(x, -v)), -\dot{\gamma}^1_{x, v}(t(x, -v)))
    \]
    the corresponding one for $(x, -v)$, then by construction
    \[
    (z, -u) = S_1(y, w), \quad t(x, v) + t(x, -v) = \tau_1(y, w).
    \]
    Since the lightlike complete travel time data are equivalent, we immediately have
    \begin{align*}
    \varphi(x, v) &= \exp^{g_2}_{\varphi_0(y)}(t(x, v)(\varphi_0)_*w)\\
    &= \exp^{g_2}_{\varphi_0(z)}((\tau_2(\varphi_0(y), (\varphi_0)_*w) - t(x, v))(\varphi_0)_*u)\\
    &= \exp^{g_2}_{\varphi_0(z)}(t(x, -v)(\varphi_0)_*u)\\
    &= \varphi(x, -v).
    \end{align*}
    Thus $\varphi(x, \cdot)$ is a constant function. The rest of the proof is the same as the time-orientable case.
\end{proof}


\section{Boundary rigidity}\label{sec: boundary rigidity}

We now study the boundary rigidity problem for an analytic globally hyperbolic Lorentzian manifold with timelike boundary. In particular, we prove Theorem \ref{thm: boundary distance main thm} in this section. In the setting of Theorem \ref{thm: exterior data no cut points}, since the time separation function is defined in the ambient manifold, $y \in J^+(x)$ implies the existence of a distance maximizing causal geodesic. 

In the contrast, in a manifold with boundary, the situation becomes much more subtle. 
For example, if the timelike boundary is strictly concave in a small neighborhood, then even locally one can find $x < y$ whose distance maximizing curve is not given by a pregeodesic. This makes identifying the intersection of light cone and boundary more challenging than the exterior case.

\subsection{Cut locus in Lorentzian manifolds with boundaries}\label{sec: cut locus in lorentzian manifolds with boundaries}
In this subsection, we prove a sequence of lemmas about cut points in a Lorentzian manifold $(M,g)$ with timelike boundary.

For $(x,v) \in L^+M$, we define  the null cut locus function
\begin{align}\label{def: null cut bd}
\rho(x,v) = \inf \{ s\in [0, \mathcal{T}(x,v)]:\ d(x, \gamma_{x,v}(s)) > 0 \},
\end{align}
where $\gamma_{x, v}(s)$ is the unique null geodesic starting from $x$ in the direction of $v$ and $\mathcal{T}(x,v)$ is the supremum of the maximal parameter value such that $\gamma_{x,v}(s)$ is defined in $M$.
If there is no $s$ such that $d(x, \gamma_{x,v}(s)) > 0$, then we write $\rho(x, v) = +\infty$.
For timelike $(x,v)$, we define the 
timelike cut locus function
\begin{align}\label{def: time cut bd}
\rho(x,v) = \inf \{ s\in [0, \mathcal{T}(x,v)]:\ d(x, \gamma_{x,v}(s)) > s \},
\end{align}
where $\gamma_{x, v}(s)$ is the unique timelike geodesic starting from $x$ in the direction of $v$.
If there is no such $s$, then we write $\rho(x, v) = +\infty$.
In both cases, when $\rho(x, v) < +\infty$, we call $\gamma_{x,v}(\rho(x,v))$ the first cut point of $x$ along $\gamma_{x,v}$.
Note that these definitions are slightly different from those for Lorentzian manifolds without boundaries, see Section \ref{sec: null and timelike cut locus}.
Indeed, if one still defines $\rho$ using the supremum, then $\rho(x, v) = \mathcal{T}(x, v)$ whenever $\gamma_{x, v}(\mathcal{T}(x, v))$ is the first cut point, or when $\gamma_{x, v}$ has no cut points at all.
We have the following lemmas for manifolds with timelike boundaries.
\begin{lm}\label{lm: cut_conjugate}
Let $(N,g)$ be a spacetime with timelike boundary. 
Let $x, y\in N^\circ$ and $\gamma_{x, v}:[0,1] \rightarrow N^\circ$ be a lightlike geodesic segment  with $y = \gamma_{x,v}(1)$. 
Suppose $\rho(x, v) > 1$.
Then the first conjugate point to $x$ along $\gamma_{x, v}$ comes after $y$.
\end{lm}
\begin{proof}
Assume for contradiction that there contains a conjugate $s_0\in (0, 1]$ to $x$ along $\gamma_{x, v}$.
By considering the same proper variation as in the proof of \cite[Proposition 10.12]{BEE96} and \cite[Proposition 10.72]{BEE96}, there is a timelike path from $x$ to $y$ arbitrarily close to $\gamma_{x, v}([0,1])$. 
As $\gamma_{x, v}([0,1])$ is contained in the interior, such timelike path is there and we can always assume it is a piecewise smooth timelike curve. Thus, we must have $d(x, y)>0$ which contradicts with $\rho(x, v) > 1$. 
\end{proof}

\begin{lm}\label{lm:shortcut}
Let $(N,g)$ be a globally hyperbolic spacetime with timelike boundary. 
Let $x, y\in N^\circ$ and $\gamma:[0,1] \rightarrow N^\circ$ be a lightlike geodesic segment  with $y = \gamma_{x,v}(1)$. 
Suppose there is another $H^1$-causal curve $\mu:[0,1] \rightarrow N$ from $x$ to $y$, which cannot be reparameterized to $\gamma$. 
Then $y$ comes on or after the first cut point along $\gamma_{x,v}$, i.e., $\rho(x,v) \leq 1$.
\end{lm}
\begin{proof}
First we prove that there exists a piecewise smooth causal curve $\mu_0:[0,1]\rightarrow N$ from $x$ to $y$, which cannot be reparameterized to $\gamma$.
Indeed, with $\gamma([0,1]) \in N^\circ$, there exist $p \in \mu((0,1)) \cap N^\circ$ such that $p \notin \gamma([0,1])$.
Note that $p \in J^+_{H^1}(x)$ and $y \in J^+_{H^1}(p)$. 
By Lemma \ref{lm_causalH1}, there exists a piecewise smooth causal curve from $x$ to $p$ and then $p$ to $y$. We refer to it by $\mu_0$ in the following. 
Note that with $p \notin \gamma([0,1])$, 
there exists $\delta > 0$ such that $y_0 = \mu_0(1-\delta) \notin \gamma([0,1])$ and $\mu_0([1-\delta, 1]) \subset N^\circ$.

By Lemma \ref{lm: H1 curve to piecewise},
We consider $y_j  = \gamma(1 + \epsilon_j) \in N$ for a sequence $\epsilon_j \rightarrow 0$. 
With $y \in N^\circ$, we may assume $y_j \in N^\circ$ as well.
Note that the causal path $\mu_j \coloneqq \mu_0([1-\delta,1]) \cup \gamma([1, 1+\epsilon_j])$ from $y_0$ to $y_j$ is not a null pregeodesic. 
By \cite[Proposition 10.46]{One83}, there is a timelike curve from $y_0$ to $y_j$ arbitrarily close to $\mu_j$ and therefore $d(y_0, y_j) > 0$. 
By the reverse inequality, we conclude that $d(x, y_j) \geq d(x,y_0) + d(y_0, y_j) > 0$.
As $y_j \rightarrow y$, one has $\rho(x, v) \leq 1$.
\end{proof}

Recall that the time separation function is defined by considering all piecewise smooth causal curves. We show that global hyperbolicity guarantees any distance maximizing $H^1$-causal curve lying in the interior must be a geodesic. Clearly the same result does not hold for points on the boundary, for example the distance maximizing curve in a strictly concave part of the boundary may be realized by boundary geodesics.
\begin{lm}\label{lm: timelike_pregeodesic}
    Let $(N, g)$ be a globally hyperbolic spacetime with timelike boundary. Let $x, y \in N^\circ$ be such that $y \in I^+(x)$. 
    Let $\alpha([0,1]) \subset N^\circ$ be an $H^1$-causal curve from $x$ to $y$ such that $L(\alpha) \geq d(x, y)$. Then $\alpha$ must be a timelike pregeodesic.
\end{lm}
\begin{proof}
    Assume for contradiction that $\alpha$ is not a timelike pregeodesic. 
    As $\alpha([0,1])$ is contained in the interior, there exist finitely many open convex neighborhoods $U_j$, for $j = 0, \ldots N$, covering it. 
    We can find a sequence $0 = s_0 < s_1 < \ldots < s_N = 1$ such that $\alpha([s_j, s_{j+1}]) \subset U_j$, for $j = 0, \ldots, N$. 
    We set $x_j = \alpha(s_j)$ with $x_N = y$ and note that $x_j < x_{j+1}$.
    As $L(\alpha) \geq d(x, y) > 0$, there exists $j_0 \in \{0, 1, \ldots, N\}$ such that $x_{j_0} \ll x_{j_0 + 1}$. 
    This implies one of the following two scenarios can happen:
    \begin{enumerate}
        \item If $\alpha([s_{j_0}, s_{j_0+1}])$ is the unique timelike pregeodesic connecting $x_{j_0}$ and $x_{j_0+1}$, we take $\alpha_0 = \alpha$. Certainly $L(\alpha_0) = L(\alpha)$.
        \item If $\alpha([s_{j_0}, s_{j_0+1}])$ is not the unique timelike pregeodesic connecting $x_{j_0}$ and $x_{j_0+1}$, we replace this segment of $\alpha$ with the unique timelike geodesic, and denote the new curve as $\alpha_0$. Note that $L(\alpha_0) > L(\alpha)$ by \cite[Proposition 7.2]{Pen72}\footnote{Indeed, the proof there works for any $H^1$-causal curve, see \cite[Remark 7.3]{Pen72}.}.
    \end{enumerate}
    We now inductively modify $\alpha_0$, and by abuse of notation we shall keep using $\alpha_0$ after the modification. The goal is to show that after finitely many steps, $\alpha_0$ is a piecewise timelike pregeodesic from $x$ to $y$ whose length is strictly larger than $\alpha$.
    
    We show the first inductive step in details. Let $j = j_0$, by the previous construction, $\alpha_0([s_j, s_{j+1}])$ is a timelike pregeodesic. Pick $\epsilon>0$ sufficiently small such that $\alpha_0([s_{j+1} - \epsilon, s_{j+2}])$ lies in a geodesically convex neighborhood, denote $x'_{j+1}:=\alpha_0(s_{j+1}-\epsilon)$. Since $\alpha_0([s_{j+1}-\epsilon, s_{j+1}])$ is a timelike pregeodesic, $x'_{j+1} \ll x_{j+1} < x_{j+2}$, so one of the following two scenarios can happen:
    \begin{enumerate}
        \item $\alpha([s_{j+1}-\epsilon, s_{j+2}])$ is the unique timelike pregeodesic connecting $x'_{j+1}$ and $x_{j+2}$. In this case we keep the segment, so $\alpha_0$ is unchanged.
        \item $\alpha([s_{j+1}-\epsilon, s_{j+2}])$ is not the unique timelike pregeodesic connecting $x'_{j+1}$ and $x_{j+2}$. In this case we replace this segment with the unique timelike geodesic, which strictly increases the total length of $\alpha_0$ by \cite[Proposition 7.2]{Pen72}, in particular $L(\alpha_0) > L(\alpha)$.
        See Figure \ref{fig: induction}.
    \end{enumerate}
    \begin{figure}
        \centering
        \includegraphics[width=0.5\linewidth]{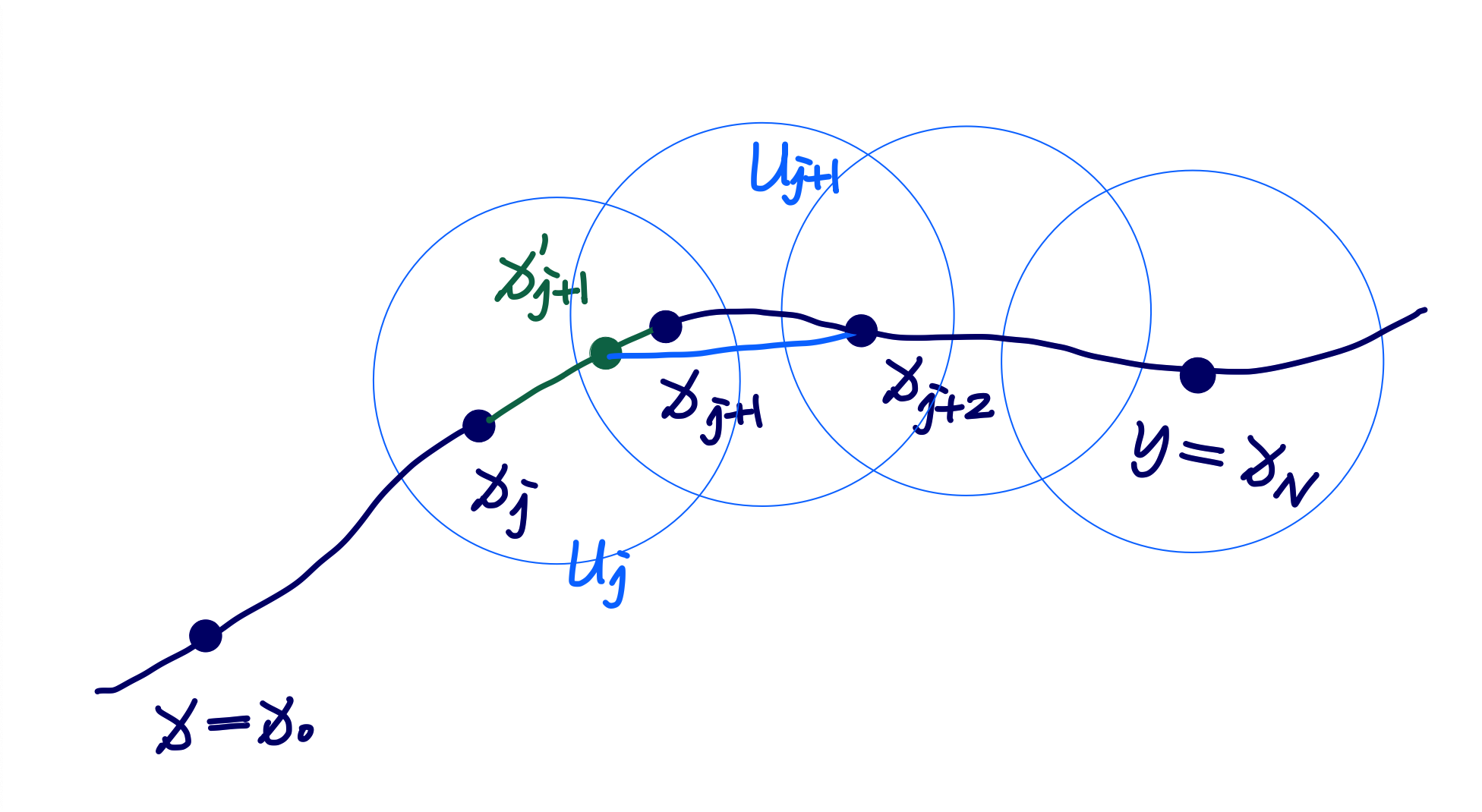}
        \caption{The  inductive step: we choose $x_{j+1}'$  sufficiently close to $x_{j+1}$ such that it is still contained in the convex neighborhood $U_{j+1}$. There exists a unique timelike geodesic segment from $x_{j+1}'$ to $x_{j+1}$.}
       \label{fig: induction}
    \end{figure}
    Now $\alpha_0([s_{j+1}, s_{j+2}])$ is a timelike pregeodesic, and the induction continuous. Certainly it terminates in finitely many steps (0 step if $j_0+1 = N$) when $j+1$ reaches $N$. Similarly, we inductively modify $\alpha_0$ in the reverse direction, by checking if $\alpha_0([s_{j-1}, s_j+\epsilon])$ is the unique timelike pregeodesic. Again, we either keep the segment if it is, or replace it with the unique timelike geodesic to increase the total length of $\alpha_0$. The induction also terminates in finitely many steps when reaching $0$.

    By construction, $\alpha_0$ is now a piecewise timelike pregeodesic. If replacement never happened, then $\alpha_0 = \alpha$, in particular we have $\alpha([s_{j_0}, s_{j_0+1}])$, $\alpha_0([s_k-\epsilon, s_{k+1}])$ and $\alpha_0([s_{l-1}, s_l + \epsilon])$ are all timelike pregeodesic, for all $k > j_0$ and $l < j_0$. As this would imply $\alpha$ is a timelike pregeodesic which contradicts our assumption, the replacement must happen at least once. In particular, $L(\alpha_0) > L(\alpha)$, contradicting $L(\alpha_0) \leq d(x, y) \leq L(\alpha)$.
\end{proof}

By \cite[Theorem 9.33]{BEE96}, the cut locus function is lower semi-continuous on a globally hyperbolic spacetime without boundary. We show similar results on a globally hyperbolic spacetime with timelike boundary, around lightlike geodesics lying in the interior.
\begin{pp}\label{pp: lower semi-cts}
Let $(N,g)$ be a globally hyperbolic spacetime with timelike boundary. 
Let $x, y\in N^\circ$ and $\gamma_{x, v}:[0,1] \rightarrow N^\circ$ be a lightlike geodesic segment with $y = \gamma_{x,v}(1)$. 
Suppose $\rho(x, v) > 1$.
Then there exists $\epsilon>0$  
such that for any $w$ in the open neighborhood 
\[
N(v, \epsilon) = \{w \in T_x M: \|w- v\|_{g^+} <  \epsilon,  \ g(w, w) < 0\},
\]
the cut locus function $\rho(x, w) > 1$.
\end{pp}
\begin{proof}
It suffices to prove $\rho(x, w) \geq 1$, as $(x, w) \in N(v, \epsilon)$ would imply $(x, (1-\delta)w) \in N(v, \epsilon)$ for sufficiently small $\delta$, then $\rho(x, w) = \frac{1}{1-\delta}\rho(x, (1-\delta)w) \geq \frac{1}{1-\delta}$.
By Lemma \ref{lm: cut_conjugate}, the first conjugate point to $x$ along $\gamma_{x, v}$ happens after $y$.
Then there exists $\epsilon>0$ and a small neighborhood $V \subset N^\circ$ of $y$
such that 
the exponential map $\exp_x: N(v, \epsilon) \rightarrow V$ is a diffeomorphism. 

Now assume for contradiction we cannot find $\epsilon>0$ such that
$\rho(x, w) \geq 1$ for any $w \in N(v, \epsilon)$.
Then there exists a sequence $w_j \in T_xM$ such that
\[
g(w_j, w_j) < 0, \quad \text{ and } \quad w_j \rightarrow v \text{ as $j\rightarrow +\infty$},
\]
but with $\rho(x, w_j) < 1$. 
It follows that $y_j = \gamma_{x, w_j}(1) \rightarrow y$ is after the first cut point of $x$ along the timelike geodesic $\gamma_{x, w_j}$. 
We may assume $y_j \in V$ by choosing sufficiently small $\epsilon>0$.
By the definition, one has $d(x, y_j) > |w_j|_g$, 
where $d$ is the Lorentzian distance function.
Then we claim there exists an $H^1$-causal curve $\mu_j$ in $M$ from $x$ to $y_j$ such that  $L(\mu_j) \geq d(x, y_j)$, where recall $L$ denotes the Lorentzian arc length.  
Indeed, by the definition of $d(x, y_j)$, there exists a sequence of future directed piecewise smooth causal curve $\mu_{j_k}:[0,1] \rightarrow N$ from $x$ to $y_j$,  such that $L(\mu_{j_k}) \rightarrow d_M(x, y_j)$. 
By Lemma \ref{lm: limit_H1}, there exists a future directed $H^1$-casual limit curve $\mu_j:[0,1] \rightarrow N$ from $x$ to $y_j$.
As the length function is upper semi-continuous, we have 
\[
L(\mu_j) \geq \limsup L(\mu_{j_k}) = d(x, y_j).
\]
With $y_j$ being the first cut point of $x$, we emphasize that $\mu_j$ cannot be reparameterized to $\gamma_{x, w_j}$.

Next, for the sequence of $H^1$-casual curve $\{\mu_j\}$ in $M$ from $x$ to $y_j$, we apply Lemma \ref{lm: limit_H1} again and there exists an $H^1$-causal limit curve $\mu:[0,1] \rightarrow N$ from $x$ to $y$. 
If such $\mu$ cannot be reparameterized to the null geodesic $\gamma_{x, v}$ from $x$ to $y$, then by Lemma \ref{lm:shortcut}, we must have $y$ is on or after the first cut point along $\gamma_{x, v}$, which contradicts with $\rho(x,v) > 1$. 
Thus, possible limit curves of $\{\mu_j\}$ can always be reparameterized to $\gamma_{x, v}$. 
By Lemma \ref{lm: limit_C0}, 
there is a subsequence $\{\mu_j\}$ that converges to $\mu$ in the $C^0$ topology.
We abuse the notation and denote it still by $\{\mu_j\}$. 
This implies one can find a small tube neighborhood $\Gamma$ of the null geodesic $\gamma_{x,v}([0,1])$ such that $\mu_j([0,1]) \subset \Gamma$, for sufficiently large $j$.
With $\gamma_{x,v}([0,1]) \subset N^\circ$, we may choose $\Gamma \subset N^\circ$ and therefore $\mu_j \subset N^\circ$ for large $j$.
Then with $L(\mu_j) \geq d(x, y_j)$, by Lemma \ref{lm: timelike_pregeodesic}, such $\mu_j$ is a timelike pregeodesic from $x$ to $y$. 
We may reparameterize  and write it as $\gamma_{x, v_j}$ for some timelike $v_j \in T_xM$.
As such $\gamma_{x, v_j}$ from $x$ to $y_j$ is contained in $\Gamma$, one has $v_j \in N(v, \epsilon)$ for large $j$.
Now for $y_j \in V$,  we have found two different timelike geodesic $\gamma_{x,w_j}$ and $\gamma_{x,v_j}$  from $x$ and $y_j$, with both $w_j, v_j \in N(v, \epsilon)$. This contradicts with the fact that $\exp_x$ is diffeomorphism there. 
\end{proof}
Then we can prove the following analog for Lorentzian manifolds with boundaries.
\begin{lm}
Let $(N,g)$ be a globally hyperbolic spacetime with timelike boundary. 
Let $x, y\in N^\circ$ and $\gamma_{x, v}([0,1])$ be a lightlike geodesic segment contained in $N^\circ$ with $y = \gamma_{x,v}(1)$.
Then either one or possibly both of the following hold:
\begin{itemize}
    \item[(1)] The point $y$ is the first conjugate point of $x$ along $\gamma$.
    \item[(2)] There exist at least one other  piecewise smooth causal curves from $x$ to $y$ that cannot be reparameterized to $\gamma$. 
\end{itemize}
\end{lm}
\begin{proof}
As $y$ is the first cut point of $x$ long $\gamma$, there exists $y_j = \gamma(1+ \epsilon_j) \rightarrow y$ in $M$, with $\epsilon_j \rightarrow 0$ and
$d(x, y_j) = \delta_j$ for some $\epsilon_j, \delta_j >0$.
Then there exists a piecewise smooth causal curve $\mu_j:[0, 1] \rightarrow M$ from $x$ to $y_j$, with the Lorentzian arc length  given by $L(\mu_j) = \delta_j$. 
As $(M,g)$ is globally hyperbolic, by Lemma \ref{lm: limit_H1}, there exists a future directed $H^1$-causal limit curve $\mu:[0,1] \rightarrow M$ connecting $x$ and $y$.
If we cannot reparameterize $\mu$ to $\gamma$, 
then by the proof of Lemma \ref{lm: H1 curve to piecewise} and Lemma \ref{lm:shortcut}, 
we may assume $\mu$ is piecewise smooth and
we have case (2). 
Otherwise, such $\mu$ can be reparameterized to $\gamma$.
By Lemma \ref{lm: limit_C0}, we can always find a subsequence of $\{\mu_j\}$ that converges to $\mu$ in the $C^0$ topology. 
Then the same argument as in the proof of Proposition \ref{pp: lower semi-cts} shows that $y$ must be a conjugate point of $x$.
\end{proof}

\subsection{The proof of the boundary rigidity}\label{sec: the proof of the boundary rigidity}
We briefly describe the general steps of proving Theorem \ref{thm: boundary distance main thm}. The first step is to determine the jet of metric on the boundary. For Riemannian manifolds, it was shown in \cite{SU09} that near a strictly convex direction, the jet of the metric can be determined from the lens relation. On the other hand, near a strictly convex direction, it is well-known that the knowledge of the distance function can imply lens data, in a geodesically convex neighborhood. Same result holds for spacetime with timelike boundary if the strictly convex direction is timelike. As their proof is not constructive and is for Riemannian manifolds, we provide a constructive one in Appendix \ref{sec: construction of jet with strictly convex direction}.

Then by analyticity one can obtain an analytic collar neighborhood extension, as explained in Section \ref{sec: normal exponential extension}. By the assumption in the theorem, we may assume the extension $(\tM, \tilde{g})$ is still an analytic globally hyperbolic spacetime with timelike boundary. The goal now is to determine the lightlike complete travel time data of $M$, then we can use Proposition \ref{prop: complete light scattering recover metric} to finish the proof. In order to do that, we first determine the time separation function $\tilde{d}$ for points in the exterior region $M^c = \tM \backslash M$.

\begin{figure}
    \centering
    \includegraphics[width=\linewidth]{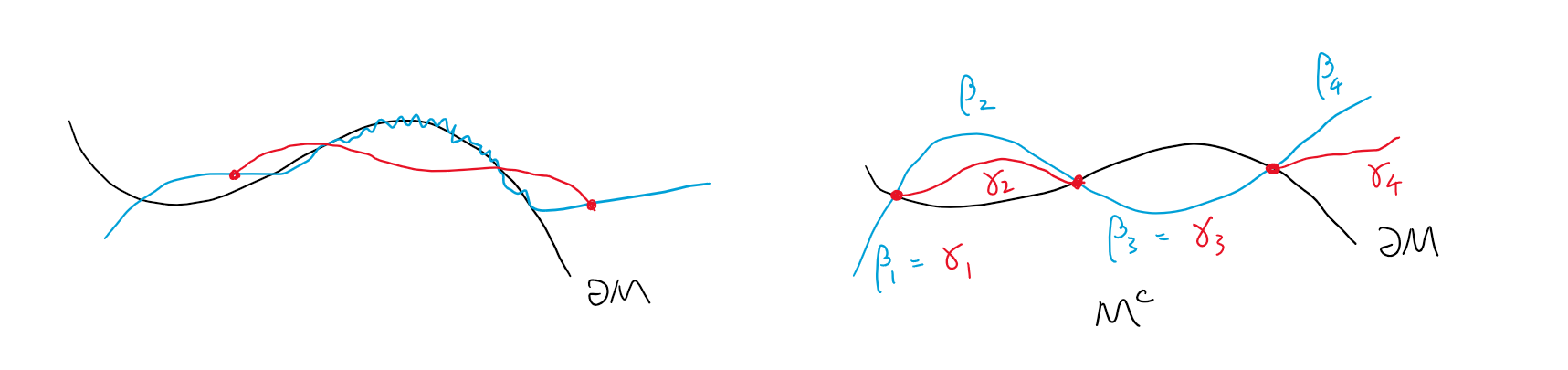}
    \caption{In the left picture, if a causal curve intersects with the boundary infinitely many times, then it can be replaced by a piecewise geodesic curve via a finite cover with geodesically convex neighborhoods. By analyticity, the piecewise geodesic curve switches between $M$ and $M^c$ only finitely many times. In the right picture, for each timelike curve $\beta_1 \circ \cdots \circ \beta_N$, we keep the same $\beta_j$ as $\gamma_j$ if it is in the exterior region, and find a $\gamma_j$ that is at least almost the same length as $\beta_j$ if it is in the interior. This is possible since the boundary time separation functions agree.}
    \label{fig: extend distance function}
\end{figure}

\begin{pp}\label{prop: extend time separation function}
    Let $(M, g)$ be an analytic Lorentzian manifold with timelike boundary, and $d$ is the boundary time separation function. Let $(\tM, \tilde{g})$ be an analytic extension, denote $M^c = \tilde{M} \backslash M$. Suppose $\tilde{g}|_{M^c}$ is given. Then $d$ uniquely determines the exterior time separation function $\tilde{d}$ with respect to $(\tilde{M}, \tilde{g})$ for all $x, y \in M^c \sqcup \partial M$.
\end{pp}
\begin{proof}
    Fix some $x, y \in M^c \sqcup \partial M$, we first show that to compute $\tilde{d}(x, y)$, it suffices to consider all piecewise smooth, future pointing causal curves from $x$ to $y$ that can be decomposed into finitely many segments, with each segment either fully in $M$ or in $M^c$. Suppose not, then it enters and exits $M$ infinitely many times in a bounded time interval. Denote the curve by $\alpha$, then the set of points where this transition happens contains accumulation point $p$. Around $p$, we may consider a geodesically convex neighborhood $W$. By analyticity of the boundary, locally a geodesic either fully stays in the boundary or has only discrete intersection point, so the segment of $\alpha$ in $W$ can not be a geodesic. Then, one can replace it with a geodesic segment to increase the length (the geodesic would be causal because $\alpha$ is causal). This geodesic segment now has finitely many intersection points with the boundary. We may perform this for all accumulation points, since the travel time is bounded, by compactness we will end up with a curve that only switches finitely many times between $M$ and $M^c$. We refer to such a curve as a \textit{valid curve} for simplicity. See the left graph in Figure \ref{fig: extend distance function}.

    Suppose now there are two metrics $\tilde{g}_j$ on the same $\tM$ and $M$ with the same boundary time separation function $d_j$, and $\tilde{g}_1|_{M^c} = \tilde{g}_2|_{M^c}$. We show that their exterior time separation function must be the same. Fix $x$ and $y$ in $M^c \sqcup \partial M$ such that $\tilde{d}_1(x, y) > 0$, suppose $\alpha_k$ is a sequence of valid curves, such that $L_1(\alpha_k) \to \tilde{d}_1(x, y)$, $L_1$ computes the length of the curve with respect to $\tilde{g}_1$. We may assume $\alpha_k$ is timelike almost everywhere with respect to $\tilde{g}_1$: since $\tM$ is open and $\tilde{d}_1(x, y) > 0$, we can perform the same replacement argument as in the proof of Lemma \ref{lm: timelike_pregeodesic} so that the resulting curve is timelike almost everywhere and the length does not decrease. By the previous arguments, we may assume each $\alpha_k$ switches between $M$ and $M^c$ finitely many times, denote the decomposition as $\beta_1, \cdots, \beta_N$. If $\beta_j \subset M^c$, we let $\gamma_j = \beta_j$ be the same curve, then $L_2(\gamma_j) = L_1(\beta_j)$ since the exterior metrics are the same. If $\beta_j \subset M$, then denote the starting and ending point as $z_1, z_2 \in \partial M$, we must have $d_1(z_1, z_2) \geq L_1(\beta_j) > 0$. As a result $d_2(z_1, z_2) = d_1(z_1, z_2) \geq L_1(\beta_j) > 0$, so there exists $\gamma_j \subset M$ piecewise smooth, future pointing causal curve with respect to $g_2$ connecting $z_1$ and $z_2$, such that $L_2(\gamma_j) > d_2(z_1, z_2) - \epsilon$ for arbitrarily small $\epsilon$. As a result, $\gamma_1, \cdots, \gamma_N$ form a piecewise smooth, future pointing causal curve with respect to $\tilde{g}_2$ connecting $x$ to $y$. See the right graph of Figure \ref{fig: extend distance function}. Moreover,
    \[
    L_2(\gamma_1)+\cdots + L_2(\gamma_N) \geq L_1(\beta_1) + \cdots + L_1(\beta_N) - N \cdot \epsilon = L_1(\alpha_k) - N \cdot \epsilon.
    \]
    Since $\epsilon$ is arbitrarily small, we have $\tilde{d}_2(x, y) \geq \tilde{d}_1(x, y) > 0$. By symmetry, we also have $\tilde{d}_1(x, y) \geq \tilde{d}_2(x, y)$. As a result, $\tilde{d}_1$ and $\tilde{d}_2$ have the same support in the exterior region, and in their common support, $\tilde{d}_1 = \tilde{d}_2$.
\end{proof}

Now we have the information of two different time separation functions, $\tilde{d}$ for $(\tM, \tilde{g})$ and $d$ for $(M, g)$, for points in $M^c \sqcup \partial M$ and $\partial M$, respectively. As explained in the beginning of Section \ref{sec: boundary rigidity}, unlike the globally hyperbolic spacetime without boundary case, $x$ and $y \in \partial J^+(x)$ may not be connected by any lightlike geodesic. To select out those $y$ that are connected by a lightlike geodesic in $M$, we use a criterion that involves both $d$ and $\tilde{d}$. Intuitively, these are the points in $J^+(x) \cap \partial M$ whose distance to $x$ remains $0$ after the extension.
\begin{pp}\label{prop: identify light cone on boundary}
    Let $(M, g)$ be an analytic globally hyperbolic Lorentzian manifold with timelike boundary. Suppose there exists an analytic extension $(\tM, \tilde{g})$ such that the lightlike geodesics in $\tM$ do not have cut points. 
    For any $x \in \partial M$ and $v \in \overline{\partial_-L^+_xM}$, denote by $\gamma_{x, v}$ the corresponding inextendible geodesic in $M$. We have
    \[
    \bigcup_{v \in \overline{\partial_- L^+_xM}}\gamma_{x, v} \cap \partial M = \partial(\{y \in \partial M: d(x, y) > 0\}) \cap \{y \in \partial M: \tilde{d}(x, y) = 0\}.
    \]
\end{pp}
\begin{proof}
    Let $y \in \gamma_{x, v} \cap \partial M$ for some $v \in \overline{\partial_-L^+_xM}$. Since the lightlike geodesics do not have cut points in $\tM$, we must have $\tilde{d}(x, y) = 0$. Then $d(x, y) \leq \tilde{d}(x, y)$ implies $d(x, y) = 0$. Since the boundary is timelike, pick any future pointing timelike curve $\alpha$ in the boundary starting at $y$, clearly $d(x, \alpha(s)) > 0$ for any $s > 0$. Thus $y$ belongs to the right hand side.

    Now let $y$ be a point in the set on the right-hand side. Let $z_j \in \partial M$ be a sequence of points converging to $y$ such that $d(x, z_j) > 0$. 
    Then we can find a sequence of piecewise smooth causal curves $\gamma_j \subset M$ from $x$ to $z_j$. 
    By \cite[Proposition 2.19]{AFS21}, there exists a future directed $H^1$-causal limit curve $\gamma \subset M$ by the global hyperbolicity of $M$, going from $x$ to $y$. Note that $\gamma$ may not be piecewise smooth, but it can always be modified into a piecewise smooth causal curve joining $x$ and $y$, see Lemma \ref{lm: H1 curve to piecewise}.
    Now assume for contradiction that $\gamma$ is not a lightlike pregeodesic. 
    Then it is a piecewise smooth causal curve in the interior of $\tM$ which is not a lightlike pregeodesic. By \cite[Proposition 10.46]{One83}, there exists a piecewise smooth timelike curve arbitrarily close to $\gamma$ in $\tilde{M}$ from $x$ to $y$. This contradicts with the assumption that $\tilde{d}(x, y) = 0$. 
    Thus $\gamma$ is a lightlike pregeodesic in $M$, this finishes the proof.
\end{proof}

We are now ready to put these together and prove the main theorem of this section.
\begin{proof}[Proof of Theorem \ref{thm: boundary distance main thm}]

    Throughout the proof, we omit the subscript $j$, and say a quantity $A$ is determined if $A_1$ and $A_2$ are related by $\varphi_0$. We first show the boundary metric can be determined around boundary points with strictly convex directions. Consider  the strictly convex direction $(x, v)$ in $T\partial M$ given in the statement of the theorem, then strict convexity holds for all $(y, w)$ sufficiently close  to $(x,v)$. Pick a smooth curve $\alpha$ in the boundary such that $(\alpha(0), \dot{\alpha}(0)) = (y, w)$ and $d(y, \alpha(s)) > 0$ for $s \ll 1$. Such $(y, w)$ and $\alpha$ exist because the boundary is timelike and $(x, v)$ is a causal direction. In particular, $w$ is timelike, we need to determine its length. By strict convexity, for $s\ll 1$, $d(y, \alpha(s))$ is obtained by a unique timelike geodesic in $M$, and $d(y, \alpha(s)) = |\exp_y^{-1}(\alpha(s))|_g:= |w(s)|_g$, for $w(s) \in T_y\partial M$. By construction, $w(0) = 0$ and $\dot{w}(0) = d\exp_y|_{w(0)}\dot{w}(0) = \dot{\alpha}(0) = w$. In particular,
    \[
    (d(y, \alpha(s)))^2 = -g(\dot{w}(0), \dot{w}(0))s^2+O(s^3),
    \]
    so $g(w, w)$ is determined from the boundary time separation function. Since this holds for all timelike $(y, w)$ sufficiently close to $(x, v)$, the boundary metric around $x$ is determined.

    Next, we determine the scattering relation around strictly convex direction $(x, v)$. Note that because of the strict convexity, the interior and complete scattering relation for these directions coincide, so we simply use the term scattering relation. Denote by $C \subset \partial M$ the connected component that contains $x$. Since the boundary metric has been determined, we may assume $(x, v)$ is timelike, as the strict convexity is an open condition. Then there exists $V \subset C$, the image of a conic neighborhood of $(x, v)$ in $T_xC$, such that for any $y \in V$, $d(x, y) > 0$ is the length of the unique timelike geodesic connecting $x$ and $y$ in $M$. Consider a local extension around $x$ and $U$ a geodesically convex neighborhood of $x$ in the extension. Let $\bar{d}$ be the time separation function for points in $U$. Then $\bar{d}(x, y) = d(x, y)$ for all $y \in V$, since the unique timelike geodesic from $x$ to $y$ is in $M \cap U$, by shrinking $V$ if necessary. In particular, $d(x, \cdot)$ is smooth around $y$; similarly, $d(\cdot, y)$ is smooth around $x$ for every $y \in V$. As a result, by \cite[Lemma 5]{LOY16}, the two differentials
    \[
    \eta:=\diff(d(x, \cdot))|_y = [\diff (\tilde{d}(x, \cdot))|_y]|_{T_yC}, \quad \xi:= \diff (d(\cdot, y))|_x = [\diff(\tilde{d}(\cdot, y))|_x]|_{T_xC}
    \]
    are the projection of unit timelike covectors at $x$ and $y$ respectively, whose corresponding vectors are related by the projected scattering relation. In boundary normal coordinates around $x$, the scattering relation is thus given by
    \[
    (x, g^{\alpha\beta}\xi_\alpha \partial_\beta + \sqrt{-1-|\xi|_g}\partial_n) \to (y, g^{\alpha\beta}\eta_\alpha \partial_\beta + \sqrt{-1-|\eta|_g}\partial_n).
    \]
    As a result, the scattering relation is determined for causal directions sufficiently close to $(x, v)$.

    Since the metric and scattering relation are determined around the strictly convex direction $(x, v)$ (recall we can assume $(x, v)$ timelike), we can now apply Theorem \ref{thm: construct jet}, which determines the jet of metric in boundary normal coordinates at $x$. Consider the extension of a collar neighborhood using normal exponential map, denote the extended manifold as $(\tM, \tilde{g})$. By the same argument as \cite[Section 2]{Var09}, the metric $\tilde{g}$ on $M^c = \tM \backslash M$ is determined, with the only difference being that the collar neighborhood is not uniform in size as the manifold is non-compact. Moreover, by the assumption in the statement of the theorem, we can make the extension sufficiently small so that $(\tM, \tilde{g})$ is a globally hyperbolic spacetime with timelike boundary and lightlike geodesics do not have cut point in $\tilde{M}$.

    By Proposition \ref{prop: extend time separation function}, the time separation function $\tilde{d}$ for $(\tM, \tilde{g})$ is determined for points $x, y$ in the exterior region $M^c \sqcup \partial M$. By Proposition \ref{prop: identify light cone on boundary}, we can also identify the intersection between lightlike geodesics with the boundary:
    \[
    L_x := \bigcup_{v \in \overline{\partial_- L^+_xM}}\gamma_{x, v} \cap \partial M, \quad x \in \partial M.
    \]
    Let $y \in L_x$ reached by $\gamma_{x, v}$ at time $T$, by rescaling we may assume $y = \gamma_{x, v}(1)$. We first recover $w = d\exp_x|_{Tv}v$.

    Since the boundary is timelike and we know the metric in $M^c$, there exists a sequence of $z_j \to y$ such that $z_j \in I^+(y) \cap M^c$. By the assumption that lightlike geodesics do not have cut points, so they do not have conjugate points by Lemma \ref{lm: cut_conjugate}. Locally, $\exp_x$ is then a diffeomorphism around $v$. Denote by
    \[
    v_j:= \exp_x^{-1}(z_j), \quad u_j := d\exp_x|_{v_j}v_j.
    \]
    For any $v'$ sufficiently close to $v$, $\gamma_{x, v'}([0, 1])$ are in the interior of $\tM$, since $\gamma_{x, v}([0, 1])$ is in $M$. By Proposition \ref{pp: lower semi-cts}, for $v'$ sufficiently close to $Tv$, the cut point on $\gamma_{x, v'}$ comes after $z':= \gamma_{x, v'}(1)$, which implies $\tilde{d}(x, z') = |v'|_{\tilde{g}} = |\exp^{-1}_x(z')|_{\tilde{g}}$. In particular, $|v_j|_{\tilde{g}} = |u_j|_{\tilde{g}} = \tilde{d}(x, z_j)$.
    
    Moreover, $\tilde{d}(x, \cdot)$ is smooth around $z_j$ for sufficiently large $j$. By the same computation as \cite[Lemma 5]{LOY16}, $u_j = \tilde{d}(x, z_j) \cdot \nabla_{\tilde{g}}(\tilde{d}(x, \cdot))|_{z_j}$ is determined by $\tilde{d}$, where $\nabla_{\tilde{g}}$ is the gradient. As $z_j \to y$, $w$ is determined via $u_j \to w$.

    To determine the corresponding $v$ at $x$, we perform the same computation around $x$. Specifically, there exists a sequence of $z_j' \to x$ such that $z_j' \in I^-(x) \cap M^c$. Same argument shows $\tilde{d}(\cdot, y)$ is smooth around $z'_j$, and $|\tilde{d}(z', y)| = |\exp_y^{-1}(z')|_{\tilde{g}}$ for $z'$ close to $x$. Then $v$ is recovered as the limit of $-\tilde{d}(z'_j, y) \cdot \nabla_{\tilde{g}}(\tilde{d}(\cdot, y))|_{z'_j}$.

    Since this can be computed for all $y \in L_x$, for any $v \in \overline{L_x^+M}$, we group the $(y_j, w_j)$ whose corresponding $(x, v_j)$ is a scaling of $(x, v)$. The final exit point is thus the $(y_j, w_j)$ whose corresponding $v_j$ is the longest. This gives us the range of the lightlike complete travel time data $(x, v) = (x, \frac{|v|}{|v|_j}v_j) \to (y_j, \frac{|v|}{|v_j|}w_j, \frac{|v_j|}{|v|})$, where $|\cdot|$ is any measure of vector length. Finally, we apply Proposition \ref{prop: complete light scattering recover metric} to obtain isometry between $M_1$ and $M_2$.  
\end{proof}


\section{Determination of jet}\label{sec: determination of jet}

In this section, we study the determination of the boundary jet from both the interior and the complete scattering relation. The main idea is similar to \cite[Theorem 1]{SU09}, but under different assumptions.
Specifically, whereas \cite{SU09} used lens data, which under their definition does not require a priori knowledge of the boundary metric, see  \cite[Section 1]{SU09}, here we work with scattering information. 
Since we are interested in the scattering rigidity problem, we do not assume the length function or the travel time function. Nevertheless, we prove that the same result can be achieved if one has knowledge of the boundary metric, using the first variation of the length function. 
We first state the results for Riemannian manifolds, then extend the result to Lorentzian setting. Finally, we refer to Section \ref{sec: construction of jet with strictly convex direction} for a constructive proof for a strictly convex direction.

\begin{thm}
    Let $(M, g)$ be a compact Riemannian manifold with boundary. Let $(x_0, v_0) \in S\partial M$ be such that the maximal geodesic $\gamma_0$ through it leaves $M$ in finite time. Suppose $x_0$ is not conjugate to any point in $\gamma_0 \cap \partial M$. If $S^{in}$ is known on some neighborhood of $(x_0, v_0)$ and the boundary metric $g|_{T\partial M \times T\partial M}$ is known on some neighborhood of $\gamma_0 \cap \partial M$, then the jet of $g$ at $x_0$ in boundary normal coordinates is determined uniquely.
\end{thm}
\begin{proof}
    The proof stays close to the proof of \cite[Theorem 1]{SU09}. By the same argument there, this non-conjugacy condition holds for any $(x, v)$ sufficiently close to $(x_0, v_0)$. Now consider $v_\epsilon = v_0 + \epsilon \partial_n$ in boundary normal coordinates, and denote $(y_\epsilon, w_\epsilon) = S^{in}(x_0, v_\epsilon)$. By compactness there exists a subsequence $y_{\epsilon_j} \to y_0^*$ for some $y_0^*$, and $y_0^*$ lies in $\gamma_0 \cap \partial M$. By the non-conjugacy assumption, $y_0^*$ is not conjugate to $x_0$.
    
    Denote by $\gamma_j$ the geodesic corresponding to $(x_0, v_{\epsilon_j})$. The same argument there implies there exists $H_j$ hypersurfaces on $\partial M$ containing $x_0$ such that a small neighborhood of $x_0$ on the boundary, denoted $U_j \subset \partial M$, can be decomposed into $U^+_j \sqcup (H_j \cap U_j) \sqcup U^-_j$, where $U^+_j$ is visible from $y_{\epsilon_j}$.
    Here visible means any point in $U^+_j$ can be connected to $y_{\epsilon_j}$ by a geodesic that is in the interior except end points. Since the non-conjugacy condition is an open condition, it holds for $(x, u_j(x))$ where $x \in U^+_j$ and $u_j(x)$ is the unique direction near $(x_0, v_0)$ such that the base point of $S(x, u_j(x))$ is $y_{\epsilon_j}$. We can thus define $\tau_j(x):= \ell^{in}(x, u_j(x))$ as a function on $U^+_j$. It is smooth in $U^+_j$ because the geodesic from $y_{\epsilon_j}$ connects to $x$ is transversal at $x$. Note that so far every step is the same as \cite[Theorem 1]{SU09}, with the only difference being we do not have knowledge of $\tau_j$, since we do not have the length function. This is not a problem because we do not need to recover the boundary metric.

    Let $C_j := \tau_j(x_0, v_{\epsilon_j})$ be some positive constant. 
    Although we do not know this constant, it does not affect the future computations. By the first variation formula of the length function, for any $x(s)$ smooth curve in $U^+_j \cup \{x_0\}$ connecting $x_0$ to $x$, we obtain
    \[
    \tau_j(x) = C_j - \int_0^1 g(x'(s), u_j(x(s)))ds.
    \]
    Because we have the scattering relation and the boundary metric, this is saying $\tau_j$ can be determined up to a constant (this fact is also stated in the proof of \cite[Theorem 1]{SU09}). In particular, the derivative of $\tau_j$ can be determined from our assumptions. The remaining proof follows exact like the proof of \cite[Theorem 1]{SU09}, because the Eikonal equation
    \[
    g^{\alpha\beta}\partial_\alpha \tau_j \partial_\beta \tau_j + (\partial_n\tau_j)^2 = 1
    \]
    only involves the derivatives of $\tau_j$.
\end{proof}

In \cite[Theorem 1]{SU09}, the authors used interior lens data. Next we show that one can also use complete scattering relation. We emphasize that the existence of $U^+_j$ requires the geodesic connecting $x_0$ and $y_{\epsilon_j}$ to stay in the interior of $M$ except for end points, which can not be guaranteed if $y_{\epsilon_j}$ is obtained via $S(x_0, v_{\epsilon_j})$. On the other hand, a generic perturbation would solve the problem.

\begin{thm}
    Let $(M, g)$ be a compact Riemannian manifold with boundary. Let $(x_0, v_0) \in S\partial M$ be such that the maximal geodesic $\gamma_0$ through it leaves $M$ in finite time. Suppose $x_0$ is not conjugate to any point in $\gamma_0 \cap \partial M$. If $S$ is known on some neighborhood of $(x_0, v_0)$ and the boundary metric $g|_{T\partial M \times T\partial M}$ is known on some neighborhood of $\gamma_0 \cap \partial M$, then the jet of $g$ at $x_0$ in boundary normal coordinates is determined uniquely.
\end{thm}
\begin{proof}
    We show that we can still find $\tv_\epsilon$ arbitrarily close to $v_\epsilon = v+\epsilon \partial_n$ such that the geodesic from $(x_0, \tv_\epsilon)$ to $S(x_0, \tv_\epsilon)$ stays in the interior of $M$ except for two end points, whose directions are transversal to the boundary. In particular, this means $S(x_0, \tv_\epsilon) = S^{in}(x_0, \tv_\epsilon)$. Indeed, consider arbitrarily small open set of directions containing $v_\epsilon$ in $\partial_-T_xM$. By non-trapping assumption, there exists some $T$ such that all of them leaves $M$ before time $T$. Let $\tM$ be any extension of $M$, by Lemma \ref{lm: transversal_global}, the corresponding geodesic for all except a zero measure set of them will intersect with the boundary transversally every time it intersects with the boundary. Let $\tv_\epsilon$ be any such direction, then it is in $\partial_- TM$, and $\gamma_{x_0, \tv_\epsilon}$ intersects with the boundary only transversally. In particular, this means the first time it leaves the manifold is transversal, and does not touch the boundary in between. This proves what we need. We have shown existence of such $\tv_\epsilon$, to explicitly find one, use the fact that $\tv_\epsilon$ is such a direction if and only if
    \[
    \{(z, u): S(z, u) = S(x_0, \tv_\epsilon)\}
    \]
    has cardinality 2 with both elements transversal to the boundary. We can now pick $\tv_\epsilon = v_\epsilon + O(\epsilon^N)$ for any large $N$ so that $\tv_\epsilon \to v$, the rest of the proof is now the same with $\tv_\epsilon$ replacing $v_\epsilon$. The higher order error does not affect the computation for lower order jet, and since we may choose $N$ to be arbitrarily large, we may determine all the jet.
\end{proof}

Since we are mainly interested in Lorentzian manifolds with timelike boundary, we prove the corresponding statements in the Lorentzian setting. The ideas are the same when the non-conjugacy condition holds for a tangential timelike direction.

\begin{thm}\label{thm: lorentzian determine jet}
    Let $(M, g)$ be a Lorentzian manifold with timelike boundary. Let $(x_0, v_0) \in T\partial M$ be unit timelike vector tangential to the boundary, such that the maximal geodesic $\gamma_0$ through it leaves $M$ in finite time. Suppose $x_0$ is not conjugate to any point in $\gamma_0 \cap \partial M$. If $S^{in}$ or $S$ is known on some neighborhood of $(x_0, v_0)$ and the boundary metric $g|_{T\partial M \times T\partial M}$ is known on some neighborhood of $\gamma_0 \cap \partial M$, then the jet of $g$ at $x_0$ in boundary normal coordinates is determined uniquely.
\end{thm}
\begin{proof}
    The proof is almost identical with the Riemannian case, we point out some small adjustments below.
    
    Let $\tM$ be any extension of $M$ and suppose $\gamma_0(T) \notin M$. Then there exists a small tubular neighborhood $U$ of $\gamma_0([0, T])$ and a sufficiently small open set of directions $\mathcal{U} \subset \partial IM$ containing $(x_0, v_0)$ such that the geodesics from $(x, v) \in \mathcal{U}$ will stay in $\bar{U} \cap M$ before leaving $M$. Hence we may perform all the steps in this tubular neighborhood, which is compact. Since we are only working with geodesics very close to $\gamma_0$, we may assume all the geodesics appear in the proof to be timelike geodesics. For timelike geodesics, one can check that the step involving the existence of visible set still holds, specifically the proof of the statement in \cite[Lemma 2.3]{Sha02} works for timelike geodesics as well. For timelike geodesics, the first variation formula for the length function has the opposite sign
    \[
    \tau_j(x) = C_j + \int_0^1 g(x'(s), u_j(x(s)))ds.
    \]
    The Eikonal equation also has the opposite sign
    \[
    g^{\alpha\beta}\partial_\alpha \tau_j \partial_\beta \tau_j + (\partial_n\tau_j)^2 = -1.
    \]
    Both of them do not affect the proof.

    We will rewrite the induction step in \cite[Theorem 1]{SU09} for the case $y_0^* = x_0$, since timelike geodesics locally maximizing distances instead of minimizing distances. The main idea is exactly the same. Suppose for the two metrics $g, h$, there exists $k \geq 1$ such that at $(x_0, v_0)$,
    \[
    \partial_n^jg^{\alpha\beta} v_\alpha v_\beta = \partial_n^jh^{\alpha\beta} v_\alpha v_\beta, \quad j < k; \quad \partial_n^kg^{\alpha\beta} v_\alpha v_\beta > \partial_n^jh^{\alpha\beta} v_\alpha v_\beta.
    \]
    Then $g^{\alpha\beta} v_\alpha v_\beta > h^{\alpha\beta} v_\alpha v_\beta$ for all $(x, v)$ close to $(x_0, v_0)$, excluding the boundary points. This is equivalent to
    \[
    \sqrt{-g^{\alpha\beta} v_\alpha v_\beta} < \sqrt{-h^{\alpha\beta} v_\alpha v_\beta}.
    \]
    Denote $L^g, L^h$ the length of a curve with respect to $g, h$. Then integrating along $\gamma^g_j$, the geodesic connecting $x_0$ to $y_{\epsilon_j}$ with respect to $g$, we obtain
    \[
    L^g(\gamma^g_j) < L^h(\gamma^g_j)
    \]
    since these are short geodesics with initial direction close to $(x_0, v_0)$. Use the fact that they share the same scattering relation near $(x_0, v_0)$, and the boundary metrics agree around $x_0$, the first variation formula for the length function with initial data $L^g(x_0 \mapsto y_0^*) = L^h(x_0 \mapsto y_0^*) = 0$ gives
    \[
    L^h(\gamma^h_j) = L^g(\gamma^g_j) < L^h(\gamma^g_j).
    \]
    Note that $\gamma^g_j$ is still a timelike curve with respect to $h$ because $0 > g^{\alpha\beta} v_\alpha v_\beta > h^{\alpha\beta} v_\alpha v_\beta$ along it. This is a contradiction because $\gamma^h_j$ should be the distance maximizing causal curve between $x_0$ and $y_{\epsilon_j}$ with respect to $h$, if we consider sufficiently large $j$ so that $y_{\epsilon_j}$ is in the geodesic convex neighborhood of $x_0$.
\end{proof}

\section{Interior and complete scattering information}\label{sec: interior and complete scattering information}

In this section, we study the relationship between interior and complete scattering information. As explained in Section \ref{sec: introduction}, both definitions have made their appearance in the literature, and neither of them can trivially imply the other one. Even though the paper is focusing on analytic Lorentzian manifold, we choose to study the Riemannian setting first. The techniques used in the Riemannian setting will then be used on Lorentzian manifolds.

\subsection{Riemannian setting}\label{sec: riemannian setting}
For a compact Riemannian manifold with boundary, we first show that the interior lens data can be constructed from the complete lens data.
\begin{lm}\label{lm: riemannian recover interior from complete}
    Let $(M, g)$ be a compact Riemannian manifold with boundary, suppose it is non-trapping. Then $(S^{in}, \ell^{in})$ can be recovered from $(S, \ell)$.
\end{lm}
\begin{proof}
    Let $(x, v) \in \partial_- SM$, then $S(x, v) = (y, w)$ is the final exit point. If we denote $\varphi(x, v, \cdot)$ as the complete geodesic flow with respect to $(x, v)$, then
    \[
    \varphi(x, v, \cdot) \cap \partial TM = \{(z, u): S(z, u) = (y,w)\}.
    \]
    One can then order them via $\ell(z, u)$. We thus have
    \[
    \ell^{in}(x, v) = \ell(x, v) - \min\{\ell(z, u) < \ell(x, v): (z, u) \in \varphi(x, v, \cdot)\cap \partial TM\}
    \]
    and $S^{in}(x, v)$ is the $(z, u)$ that obtained the minimum.
\end{proof}

\begin{figure}
    \centering
    \includegraphics[width=\linewidth]{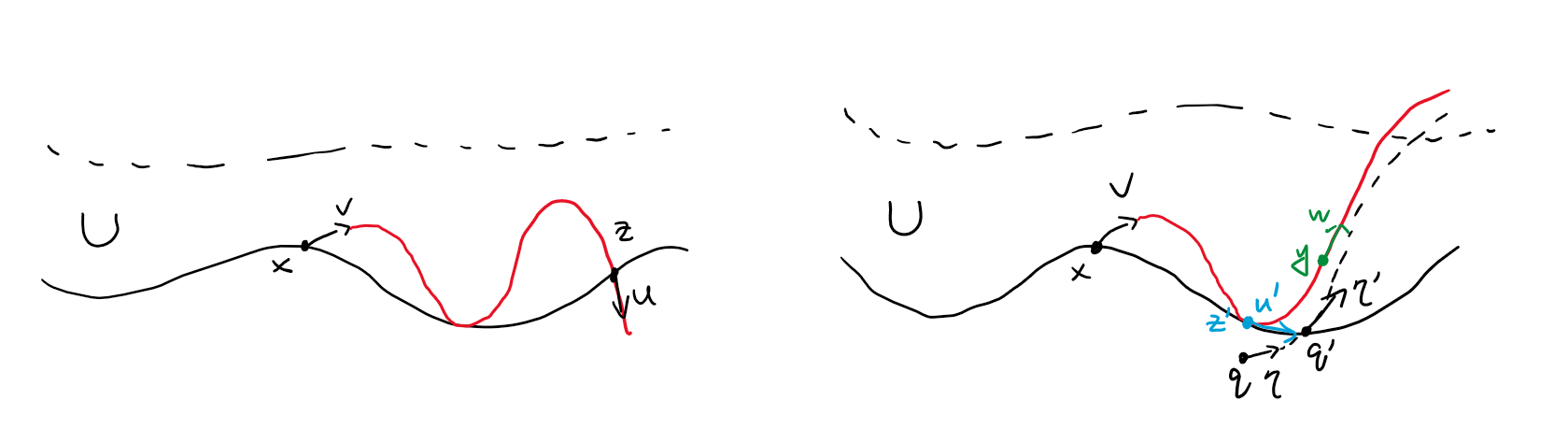}
    \caption{There are two possibilities for the geodesic to leave $U$: either it leaves $M$ or it enters $M\backslash U$. For the second case we use nearby transversal starting points to find a future boundary intersection point of $(x, v)$.}
    \label{fig: leaving U}
\end{figure}

Note that the above proof does not work when one only has the interior scattering relation. 
The other direction, however, is more challenging. This is because the interior scattering relation is defined in $\partial_-SM$, which excludes the tangential directions. As a result, if a geodesic tangentially enters $M$ and then tangentially exits $M$, then it is invisible with respect to the interior scattering information, but should be contained in the complete scattering relation.
On the other hand, if one has a prior information of the metric around the boundary, then the tangential geodesics are no longer invisible. 
Note that we do not claim this to be the optimal result: for example, one may try to approximate the tangential directions using limits of transversal directions without requiring the knowledge of the boundary metric. 
We prove this weaker version because it suffices for the scattering rigidity result in the analytic setting.
\begin{lm}\label{lm: riemannian recover complete from interior}
    Let $(M, g)$ be a compact Riemannian manifold with boundary, suppose it is non-trapping. Suppose there exists $U \subset M$ an open neighborhood of $\partial M$ such that $g|_U$ is given. Then
    \begin{itemize}
        \item $S$ can be recovered from $S^{in}$.
        \item $(S, \ell)$ can be recovered from $(S^{in}, \ell^{in})$.
    \end{itemize}
\end{lm}
\begin{proof}
    We start with the first statement. Since we know the metric on $U$, we may extend $M$ to some $\tM$, then we have knowledge of the metric on $V := U \cup (\tM \backslash M)$. For every $(x, v) \in \partial SM$, since we know the metric in $U$ and the geodesics are non-trapping, we can keep track of how the geodesic from $(x, v)$ leaves $U$ for the first time. There are two possibilities (see Figure \ref{fig: leaving U}):
    \begin{enumerate}[label=(\roman*)]
        \item it leaves $M$;
        \item it enters $M \backslash U$.
    \end{enumerate}
    Denote the time of leaving by $T$ and direction by $(z, u)$. For case (i), we refer to $(z, u)$ as the exiting direction of $(x, v)$.

    
    We now deal with case (ii). Denote $\varphi$ the geodesic flow and let $(z', u') = \varphi(x, v, T')$ be such that $T' = \inf \{t > 0: \varphi(x, v, s) \in U \backslash \partial M \text{ for } s \in [t, T)\}$. In other words, $(z', u')$ is the last time the geodesic touches the boundary before entering $M \backslash U$, note that $(z', u')$ may just be $(x, v)$ itself. Consider a geodesically convex neighborhood $W$ of $z'$. 
    Then there exists small $\delta$ such that $\varphi(z', u', t)$ stays in $W \cap M$ for all $t \in [0, \delta]$, denote $(y, w) = \varphi(z', u', \delta)$. Pick $p \in W \backslash M$, then the unique geodesic from $p$ to $y$ intersects with $\partial M$, let $\xi$ be the direction from $p$, $(p, \xi)$ is close to $(z', u')$. Using Lemma \ref{lm: transversal_local}, there exists $(q, \eta)$ sufficiently close to $(p, \xi)$ such that the geodesic from $(q, \eta)$ intersects with $\partial M$ at least once before leaving $W$, and all intersections are transversal. For $(q, \eta)$ sufficiently close to $(p, \xi)$ and $p$ sufficiently close to $z'$, we have $\varphi(q, \eta, \cdot)$ stays close to $\varphi(z', u', \cdot)$ and eventually enters $M \backslash U$. In particular, we can assume that $\varphi(q, \eta, t)$ stays in the interior of $U$ for $t_W(q, \eta) \leq t \leq t_U(q, \eta)$, where $t_W(q, \eta)$ is the time it leaves $W$ and $t_U(q, \eta)$ is the time it enters $M\backslash U$. Backtracking from $\varphi(q, \eta, t_W(q, \eta))$ until it hits $\partial M$ for the first time, denote that time by $t'(q, \eta)$ and direction by $(q', \eta') = \varphi(q, \eta, t'(q, \eta))$.
    
    To summarize, we have shown that $\varphi(q, \eta, t) \in U\backslash \partial M$ for $t \in (t'(q, \eta), t_U(q, \eta))$ before entering $M \backslash U$, and is transversal to the boundary at time $t'(q, \eta)$ with direction $(q', \eta')$. In particular, $S^{in}(q', \eta')$ is well-defined. We may find a sequence of such $(p_j, \xi_j)$ converging to $(z', u')$, and a sequence of corresponding $S^{in}(q'_j, \eta'_j) = (y_j, w_j)$. By compactness and passing down to subsequence, we may assume $(y_j, w_j) \to (y_*, w_*) \in \partial SM$. By our construction, $(y_*, w_*)$ will be on the flow of $\varphi(z', u', \cdot)$ because $(q_j, \eta_j)$ will converge to $(z', u')$. Moreover, if we denote $T_j$ as the time $\varphi(q_j, \eta_j, T_j) = (y_j, w_j)$, then $T_j > t_U(q_j, \eta_j) > \delta$. Hence the limit of $T_j$ is bounded below by $\delta$, meaning $(y_*, w_*) = \varphi(z', u', t)$ for some $t > \delta$. We emphasize that this means $(y_*, w_*) = \varphi(x, v, t)$ for some $t > \delta$, and $\delta$ can be uniformly chosen as the minimum of: half of the injective radius; and minimum times for a boundary direction to enter $M\backslash U$.

    We are now ready to prove the lemma. For any $(x, v) \in \overline{\partial_- SM}$, it falls in one of the two categories. If it is case (i), we set $S(x, v)$ to be the exiting direction and we are done. If it is case (ii), then we will find some $(y_*, w_*)$ and we repeat the procedure on $(y_*, w_*)$. Note that we can guarantee $(y_*, w_*)$ is at least $\delta$ time after $(x, v)$, so the procedure will terminate in finite steps by the non-trapping assumption. In other words, we will reach case (i) in finitely many steps, and we set $S(x, v)$ to be the final exiting direction $(z, u)$. $S$ is recovered from $S^{in}$.

    To recover $(S, \ell)$ from $(S^{in}, \ell^{in})$, it now suffices to show $\ell$ can be recovered as well. During the above procedure, we can recover $L_j = \ell^{in}(q_j', \eta_j') + L((q_j, \eta_j) \to (q_j', \eta_j'))$. So $\ell^{in}(z', u') = \lim L_j$. We just need to add them up when recursively finding the next $(y_*, w_*)$, which as explained before terminates in finitely many steps. Finally, the projected versions are exactly the same.
\end{proof}

\subsection{The first variation of the travel time}\label{sec: the first variation of the travel time}
Before going into the Lorentzian setting, we compute the first variation of the travel time function. Even though on Riemannian manifolds, the travel time function is almost the same as the length function, it contains more information in the Lorentzian setting, especially for lightlike geodesics. We will show that this allows us to build stronger relationship between the interior and complete scattering information for lightlike geodesics.

Let $(M,g)$ be a Riemannian or Lorentzian manifold and $H_1, H_2 \subset M$ be smooth hypersurfaces. 
In the following, we compute the variation of the travel time function of a fixed geodesic between $H_1$ and $H_2$.
Let $(x_0, v_0) \in T_{H_1} M$ be fixed. 
In this part, for simplification, we use $\gamma_0: [0, \tau_0] \rightarrow M$ to denote the unique geodesic segment starting from $x_0 \in H_1$ in the direction of $v_0$. 
Suppose $\gamma_0$ hits $H_2$ transversally at $y_0 = \gamma_0(\tau_0)$. 
For intervals $I_t, I_\lambda \subset \mathbb{R}$ to be specified later, we consider a smooth one-parameter family of geodesics
\[
\Gamma: I_\lambda \times I_t \rightarrow M
\]
near $\gamma_0$, such that for each $\lambda \in I_\lambda$, the smooth curve $\gamma_\lambda(\cdot) = \Gamma(\lambda, \cdot)$ is the unique geodesic starting from $x(\lambda)  = \Gamma(0, \lambda)$ in the direction of $v(\lambda) = \partial_t \Gamma(0, \lambda)$, 
with $x(0) = x_0$ and $v(0) = v_0$. 
Suppose each $\gamma_\lambda(t)$ intersects $H_2$ at $t = \tau_\lambda$. 
As $\gamma_0$ intersects $H_2$ transversally, for sufficiently small $|\lambda|$ we can expect $\gamma_\lambda$ to intersect $H_2$ transversally as well. 
Thus, we may assume the interval $I_\lambda$ is a small open neighborhood of $\lambda = 0$ such that each 
$\gamma_\lambda$ intersect $H_2$ transversally
and the interval $I_t$ contains $[0, \tau_0]$ such that $\tau_\lambda \in I_t$ for every $\lambda \in I_\lambda$. 
Note that $\tau_\lambda$ is the solution to $\gamma_\lambda(\tau_\lambda) \in S_2$. 
By the Implicit Function Theorem, the travel time function $\tau_\lambda$ is smooth for $\lambda \in I_\lambda$. 

Now we consider the variation field
\[
X_\lambda(t) \coloneqq \partial_\lambda \Gamma(\lambda, t),
\]
which is a Jacobi field along the geodesic $\gamma_\lambda(t)$. 
We denote by 
$\dot{\gamma}_\lambda(t) = \partial_t \Gamma(\lambda, t)$ with
$\dot{\gamma}_\lambda(0) = v(\lambda)$. 
Along each geodesic, the quantity $g(\dot{\gamma}_\lambda(t), \dot{\gamma}_\lambda(t))$ is invariant and therefore we denote it by the level set
\begin{equation}\label{def_h}
h(\lambda) \coloneqq g(\dot{\gamma}_\lambda(t), \dot{\gamma}_\lambda(t)).
\end{equation}
In particular, we can find $h(\lambda)$ by checking the initial data of this family of geodesics at $H_1$, i.e., by $h(\lambda) = g(\dot{\gamma}_\lambda(0), \dot{\gamma}_\lambda(0))$. 
We compute its derivative as below:
\begin{align*}
     h'(\lambda) = \partial_{\lambda}\big(g(\dot{\gamma}_\lambda(t), \dot{\gamma}_\lambda(t))\big)
     = 2g(D_\lambda \dot{\gamma}_\lambda(t), \dot{\gamma}_\lambda(t))
     = 2g(D_t X_\lambda(t), \dot{\gamma}_\lambda(t)),
\end{align*}
where $D_\lambda, D_t$ are covariant derivatives along these curves and the last equality uses the symmetry lemma, for example, see \cite[Lemma 6.2]{Lee18}. 
With $D_t \dot{\gamma}_\lambda(t) = 0$, we have 
\[
 h'(\lambda) = 2g(D_t X_\lambda(t), \dot{\gamma}_\lambda(t)) 
 + 2g(X_\lambda(t), D_t\dot{\gamma}_\lambda(t)) 
 =
 2\partial_t \big(g(X_\lambda(t), \dot{\gamma}_\lambda(t))\big)
\]
Integrating it with respect to $t$ from $0$ to $\tau_\lambda$, we have
\begin{equation}\label{eq_h}
h'(\lambda) \tau_\lambda  = 2g(X_\lambda( \tau_\lambda), \dot{\gamma}_\lambda(\tau_\lambda)) 
- 2g(X_\lambda(0), \dot{\gamma}_\lambda(0)).
\end{equation}
Moreover, if we denote by
\[
y(\lambda) = \gamma_\lambda(\tau_\lambda)
\]
the endpoints of the geodesic segment $\gamma_\lambda(t)$ at $S_2$, then by the chain rules we have
\begin{equation}\label{eq_y}
y'(\lambda) = \frac{\diff }{\diff \lambda} \gamma_\lambda(\tau_\lambda) 
=\partial_\lambda \gamma_\lambda(t)|_{t = \tau_\lambda} +  \partial_t \gamma_\lambda(t)\frac{\diff \tau_\lambda}{\diff \lambda}
=X_\lambda(\tau_\lambda) + \dot{\gamma}_\lambda(\tau_\lambda) \frac{\diff \tau_\lambda}{\diff \lambda}. 
\end{equation}
Combining (\ref{eq_h}) and (\ref{eq_y}), one has
\begin{equation*}
 h'(\lambda) \tau_\lambda  =  2g(y'(\lambda), \dot{\gamma}_\lambda(\tau_\lambda))
 - 2g(\dot{\gamma}_\lambda(\tau_\lambda), \dot{\gamma}_\lambda(\tau_\lambda))\frac{\diff \tau_\lambda}{\diff \lambda} - 2g(X_\lambda(0), \dot{\gamma}_\lambda(0)).
\end{equation*}
Using (\ref{def_h}), we prove the following proposition.
\begin{pp}\label{pp_traveltime}
Let $H_1, H_2 \subset M$ be smooth hypersurfaces and $\gamma_0: [0, T_0] \rightarrow M$ be the unique geodesic segment starting from $x_0 \in H_1$ in the direction of $v_0$. 
Suppose $\gamma_0$ hits $H_2$ transversally at $y_0 = \gamma_0(\tau_0)$.
Consider a one-parameter family of geodesics $\gamma_\lambda(t)$ near $\gamma_0$ with initial data $(x(\lambda), v(\lambda))$ smoothly depending on $\lambda$, with $x(0) = x_0$ and $v(0) = v_0$. 
Let $\tau_\lambda$ be the travel time such that $\gamma_\lambda(\tau_\lambda) \in S_2$. 
Then near $\lambda = 0$, 
one has 
\begin{equation}\label{eq_T}
2h(\lambda)\frac{\diff \tau_\lambda}{\diff \lambda} + h'(\lambda) \tau_\lambda =  2g(y'(\lambda), \dot{\gamma}_\lambda(\tau_\lambda))
 - 2g(x'(\lambda), \dot{\gamma}_\lambda(0)),
\end{equation}
where $h(\lambda) = g(\dot{\gamma}_\lambda(0), \dot{\gamma}_\lambda(0))$ and  $y(\lambda) =  \gamma_\lambda(\tau_\lambda)$. 
\end{pp}
This proposition relates the change of the travel time with the observations we could get from the scattering relation on $H_1$ and $H_2$.
Note that $y'(\lambda)$ and $x'(\lambda)$ is always tangent to the hypersurfaces, so the equation can give the first variation of travel time using projected scattering relation as well.
In the Riemannian setting or for timelike geodesics in the Lorentzian setting, equation \eqref{eq_T} is the same as the first variation of the length function:
\begin{align*}
&l'(\lambda) = \frac{d}{d\lambda}\left(\sqrt{\pm h(\lambda)}\tau(\lambda)\right) = \pm \frac{1}{\sqrt{\pm h}}\left( g(y', \dot{\gamma}_\lambda(\tau_\lambda)) - g(y', \dot{\gamma}_\lambda(0)) \right),
\end{align*}
with $+$ in the Riemannian case and $-$ in the Lorentzian timelike case.
However, if we consider a fixed null-geodesic, the length function fails to be differentiable there and we no longer have the first variation of the length. 
Instead, we may use \eqref{eq_T} to recover the travel time of this null-geodesic directly, provided we have timelike or spacelike scattering relation so that $h'(0) \neq 0$. Note that if $h'(0) = 0$ and $h(0) = 0$, then the equality becomes
\[
g(y'(0), \dot{\gamma}_0(\tau_0)) = g(x'(0), \dot{\gamma}_0(0)).
\]
Or equivalently, $\dot{\gamma}_0(\tau_0)^\flat (y'(0)) = \dot{\gamma}_0(0)^\flat(x'(0))$. In particular, lightlike scattering relation along can not recover the travel time.

\subsection{Lorentzian setting}\label{sec: lorentzian setting}
We are now ready for the Lorentzian setting. Let $(M, g)$ be a Lorentzian manifold with timelike boundary, in this section we do not require analyticity. Suppose the lightlike geodesics are non-trapping. Then there exists a conic open neighborhood $\mathcal{U}'$ of $\partial LM$ such that the corresponding geodesics are also non-trapping. Denote by $\mathcal{U} = \mathcal{U}' \cap \partial JM$ the causal ones. We use $\mathcal{U}_\pm$ to denote inward $(-)$ and outward $(+)$ pointing ones. Then $S$ is well-defined on $\overline{\mathcal{U}_-}$ and $S^{in}$ is well-defined on $\mathcal{U}_-$. We show that interior lightlike travel time can be recovered directly from interior scattering relation. This requires the knowledge of nearby timelike directions, so that we can apply the first variation of travel time formula \eqref{eq_T} in a non-trivial way.

\begin{lm}\label{lm: lorentzian recover travel time from interior}
    Let $(M, g)$ and $\mathcal{U}_\pm$ be defined as above. Suppose $g|_{T\partial M \times T\partial M}$ is given. Then the lightlike interior travel time data $(S^{in}, \tau^{in})|_{\partial_- LM}$ can be recovered from $S^{in}|_{\mathcal{U}_-}$. 
\end{lm}
\begin{proof}
    Given $(x, v) \in \partial_- LM$, denote $(y, w) = S^{in}(x, v)$, $w$ is either tangential to $\partial M$ or outward pointing. Then the lightlike geodesic corresponding to $(y, -w)$ leaves the manifold at $(x, -v)$ transversally, and stays in the interior of $M$ except for two ends. Pick a smooth one parameter family of outward pointing timelike vectors $w(\lambda)$ for $\lambda \in (0, 1)$ such that $w(\lambda) \to w$ as $\lambda \to 0$. Denote $h(\lambda) = g(w(\lambda), w(\lambda))$, we first show that we may choose $w(\lambda)$ such that $h'(0) = -2$.

    Consider some normal coordinate at $y$ such that $g$ is Minkowski at $y$, $\partial_t$ is tangent to $\partial M$, and $w = \partial_t + \partial_1$. Then pick $w(\lambda) = \partial_t+(1-\lambda)\partial_1$. If $w$ is outward pointing then this holds for $w(\lambda)$ when $\lambda \ll 1$; if $w$ is tangent to $\partial M$ then so are all the $w(\lambda)$. Then $h'(0) = -2$.
    
    For $\lambda \ll 1$, there exists unique $(x(\lambda), v(\lambda))$ such that $S^{in}(x(\lambda), v(\lambda)) = (y, w(\lambda))$. Moreover, by the Implicit Function Theorem, $(x(\lambda), v(\lambda))$ is a smooth curve converging to $(x, v)$ when $\lambda \to 0$. By Proposition \ref{pp_traveltime}, the first variation formula for travel time gives
    \[
    \tau^{in}(x, v) = \frac{-2}{h'(0)}g(x'(0), v) = g(x'(0), v).
    \]

\end{proof}

Comparing to Lemma \ref{lm: riemannian recover interior from complete}, we show that the lightlike interior travel time data can be recovered from complete scattering relation, without the need of complete travel time data.
\begin{lm}\label{lm: lorentzian recover travel time interior from complete}
    Let $(M, g)$ and $\mathcal{U}_\pm$ be defined as above. Suppose $g|_{T\partial M \times T\partial M}$ is given. Then the lightlike interior travel time data $(S^{in}, \ell^{in})|_{\partial_- LM}$ can be recovered from $S|_{\overline{\mathcal{U}_-}}$.
\end{lm}
\begin{proof}
    Given $(x, v) \in \partial_- LM$, consider
    \[
    \varphi(x, v, \cdot) \cap \partial TM = \{(z, u) \in \overline{\partial_+LM}: S(z, u) = (x, v)\}.
    \]
    If the set has cardinality 2, then besides $(x, v)$ itself, the other element must be $S^{in}(x, v) = S(x, v)$. Otherwise $S^{in}(x, v) \neq S(x, v)$, and for each
    \[
    (z, u) \in (\varphi(x, v, \cdot) \cap \partial TM) \backslash\{(x, v), S(x, v)\},
    \]
    $(z, u)$ is tangential to $\partial M$. There are two possibilities:
    \begin{enumerate}
        \item there exists a one parameter family of outward pointing timelike $(z, u(\lambda))$ converging to $(z, u)$ as $\lambda \to 0$, and a one parameter family of inward pointing timelike $(x(\lambda), v(\lambda))$ converging to $(x, v)$ as $\lambda \to 0$, such that $S(x(\lambda), v(\lambda)) = (z, u(\lambda))$;
        \item or such situation does not happen.
    \end{enumerate}
    If it is case (1), then we may apply Proposition \ref{pp_traveltime}, the travel time from $(x, v)$ to $(z, u)$ is thus
    \[
    \tau(z, u) = \frac{-2}{h'(0)}g(x'(0), v).
    \]
    It now suffices to show that $S^{in}(x, v)$ is the point that falls in case (1), and that it is the one minimizing $\tau(z, u)$. The fact that it falls in case (1) is proven in Lemma \ref{lm: lorentzian recover travel time from interior}, and certainly it is the minimizing one.
\end{proof}

The recovery of complete travel time data from the interior scattering relation needs more delicate treatment. If we are given complete scattering relation for causal directions sufficiently close to the light cone, then we use Proposition \ref{pp_traveltime} to recover the travel time. The situation is more complicated if we are given interior scattering relation for causal directions sufficiently close to the light cone. In this case, we first use similar strategy as Lemma \ref{lm: riemannian recover complete from interior} to recover the complete scattering relation for causal directions sufficiently close to the light cone, and then use Proposition \ref{pp_traveltime}.


\begin{figure}
     \centering
    \begin{subfigure}{0.7\textwidth}
    \centering
    \def\svgwidth{\linewidth}
    \def\svgwidth{\columnwidth}
    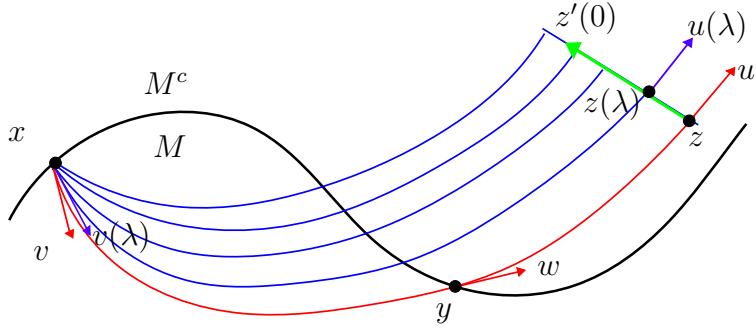

  \end{subfigure}  
    \caption{For every $(x, v) \in \overline{\partial_-LM}$, one can find a smooth one parameter family of inward pointing timelike directions converging to $v$. Their flow forms a 2-dimensional band, and in the exterior region we choose a transversal slice $z(\lambda)$ that approaches $z$ as $\lambda \to 0$. The first variation formula of travel time then recovers the travel time of $(x, v)$ to the exterior point $z$.}
    \label{fig: step 2}
\end{figure}

\begin{lm}\label{lm: lorentzian recover complete from interior}
    Let $(M, g)$ and $\mathcal{U}_\pm$ be defined as above. Suppose there exists $U \subset M$ an open neighborhood of $\partial M$ such that $g|_U$ is given. Then the lightlike complete travel time data $(S, \tau)|_{\overline{\partial_-LM}}$ can be recovered from $S^{in}|_{\mathcal{U}_-}$ 
    or $S|_{\overline{\mathcal{U}_-}}$.
\end{lm}
\begin{proof}
    We first assume $S^{in}|_{\mathcal{U}_-}$ is given. Again extend $M$ to some $\tM$, and then we have knowledge of the metric on $V := U \cup (\tM \backslash M)$. We will prove the statement in two steps:
    \begin{enumerate}
        \item we use the same idea as the Riemannian setting, to show that we can recover the complete scattering relation for all causal directions sufficiently close to the light cone;
        \item we use the complete scattering relation to compute the complete travel time of lightlike geodesics using first variation formula of travel time. 
    \end{enumerate}

    \textbf{Step (1):} consider some causal vector $(x, v) \in \overline{\partial_- TM}$ that are close to the light cone. Denote by $\varphi$ the geodesic flow. By the non-trapping assumption of lightlike geodesics, the geodesic $\varphi(x, v, \cdot)$ is non-trapping as well when $(x, v)$ is sufficiently close to the light cone. As a result, $\varphi(x, v, \cdot)$ is uniformly close to $LM$. If $\varphi(x, v, t)$ is tangential to the boundary for some $t$, then we may assume it is in $\overline{U_-}$ by the closeness to $LM$. In particular, this means for inward pointing causal directions sufficiently close to $\varphi(x, v, t)$, it will be in $\mathcal{U}_-$. We denote $\mathcal{V}$ as the set of such $(x, v)$, note that it contains $\overline{\partial_- LM}$. 
    
    Now pick some $(x, v) \in \mathcal{V}$, since we know the metric in $U$ and $\varphi(x, v, \cdot)$ is non-trapping, we can keep track of how the geodesic from $(x, v)$ leaves $U$ for the first time (see Figure \ref{fig: leaving U}). There are two possibilities:
    \begin{enumerate}[label=(\roman*)]
        \item it leaves $M$;
        \item it enters $M \backslash U$.
    \end{enumerate}
    Denote the time of leaving by $T$ and direction by $(z, u)$. For case (i), we refer to $(z, u)$ as the exiting direction of $(x, v)$.
    
    For case (ii), let $(z', u') = \varphi(x, v, T')$ be such that $T' = \inf \{t > 0: \varphi(x, v, s) \in U \backslash \partial M \text{ for } s \in [t, T)\}$. In other words, $(z', u')$ is the last time the geodesic touches the boundary before entering $M \backslash U$, note that $(z', u')$ may just be $(x, v)$ itself. Consider $W$ a geodesically convex neighborhood of $z'$, there exists small $\delta$ such that $\varphi(z', u', t)$ stays in $W \cap M$ for all $t \in [0, \delta]$, denote $(y, w) = \varphi(z', u', \delta)$. Pick $p \in (W \backslash M) \cap I^-(z')$, such $p$ exists because the boundary is timelike. Then the unique timelike geodesic from $p$ to $y$ intersects with $\partial M$, let $\xi$ be the direction from $p$. Using Lemma \ref{lm: transversal_local}, there exists timelike $(q, \eta)$ sufficiently close to $(p, \xi)$ such that the timelike geodesic from $(q, \eta)$ intersects with $\partial M$ at least once before leaving $W$, and all intersections are transversal. For $(q, \eta)$ sufficiently close to $(p, \xi)$ and $p$ sufficiently close to $z'$, we have $\varphi(q, \eta, \cdot)$ stays close to $\varphi(z', u', \cdot)$ and eventually enters $M \backslash U$. In particular, we can assume that $\varphi(q, \eta, t)$ stays in the interior of $U$ for $t_W(q, \eta) \leq t \leq t_U(q, \eta)$, where $t_W(q, \eta)$ is the time it leaves $W$ and $t_U(q, \eta)$ is the time it enters $M\backslash U$. Backtracking from $\varphi(q, \eta, t_W(q, \eta))$ until it hits $\partial M$ for the first time, denote that time by $t'(q, \eta)$ and direction by $(q', \eta') = \varphi(q, \eta, t'(q, \eta))$.
    
    To summarize, we have shown that $\varphi(q, \eta, t) \in U\backslash \partial M$ for $t \in (t'(q, \eta), t_U(q, \eta))$ before entering $M \backslash U$, and is transversal to the boundary at time $t'(q, \eta)$ with direction $(q', \eta')$. Moreover, $(q', \eta')$ will be sufficiently close to the light cone if $(q, \eta)$ is sufficiently close to $(z', u')$, because $\varphi(x, v,\cdot)$ is uniformly close to light cone. As a result, $(q', \eta') \in \mathcal{U}_-$, and $S^{in}(q', \eta')$ is well-defined. We may find a sequence of such $(p_j, \xi_j)$ converging to $(z', u')$, and a sequence of corresponding $S^{in}(q'_j, \eta'_j) = (y_j, w_j)$. By non-trapping of $\varphi(x, v, \cdot)$, $\varphi(q', \eta', \cdot)$ will be non-trapping as the two are close. Hence we may assume they live in a tubular neighborhood of $\varphi(x, v, \cdot)$, whose closure is thus compact. By compactness and passing down to subsequence, we may assume $(y_j, w_j) \to (y_*, w_*) \in \partial SM$. By our construction, $(y_*, w_*)$ will be on the flow of $\varphi(z', u', \cdot)$ because $(q_j, \eta_j)$ will converge to $(z', u')$. Moreover, if we denote $T_j$ as the time $\varphi(q_j, \eta_j, T_j) = (y_j, w_j)$, then $T_j > t_U(q_j, \eta_j) > \delta$. Hence the limit of $T_j$ is bounded below by $\delta$, meaning $(y_*, w_*) = \varphi(z', u', t)$ for some $t > \delta$. We emphasize that this means $(y_*, w_*) = \varphi(x, v, t)$ for some $t > \delta$, and $\delta$ can be uniformly chosen as the minimum of: half of the injective radius; and minimum times for a boundary direction to enter $M\backslash U$.

    We can finish the first step as the Riemannian setting. For any $(x, v) \in \mathcal{V}$, it falls in one of the two categories. If it is case (i), we set $S(x, v)$ to be the exiting direction and we are done. If it is case (ii), then we will find some $(y_*, w_*)$. By definition of $\mathcal{V}$, $(y_*, w_*)$ will still be in $\mathcal{V}$, and we can repeat the procedure on $(y_*, w_*)$. Note that we can guarantee $(y_*, w_*)$ is at least $\delta$ time after $(x, v)$, so the procedure will terminate in finite steps by the non-trapping assumption. In other words, we will reach case (i) in finitely many steps, and we set $S(x, v)$ to be the final exiting direction $(z, u)$. $S|_{\mathcal{V}}$ is recovered from $S^{in}|_{\mathcal{U}_-}$.

    \textbf{Step (2):} now we use $S|_{\mathcal{V}}$ to recover the travel time of lightlike geodesics. For any $(x, v) \in \mathcal{V}$, since we know $S(x, v)$, we may identify all the $\varphi(x, v, \cdot) \cap T(\tM \backslash M)$. That is, how the inextendible geodesic from $(x, v)$ travels in the exterior region. Fix some $(x, v) \in \overline{\partial_- LM}$, denote $(y, w) = S(x, v)$ and pick some $(z, u) = \varphi(y, w, t_0)$ in the exterior region. If the boundary is analytic then $t_0$ can be computed; otherwise $t_0$ may be unknown because $\varphi(y, w, \cdot)$ may have entered and exited $M$ infinitely many times before reaching $(z, u)$. Let $v(\lambda) \in \partial_-TM$ be a smooth one parameter of timelike directions converging to $v$ as $\lambda \to 0$. Again we may choose the parameterization such that $h(\lambda) = g(v(\lambda), v(\lambda))$ satisfies $h'(0) = -2$. Since $z \in \tM \backslash M$ and we can identify $\varphi(x, v(\lambda), \cdot) \cap T(\tM \backslash M)$, we may pick a smooth one parameter family of timelike directions $(z(\lambda), u(\lambda))$ in the exterior region such that:
    \begin{itemize}
        \item $(z(\lambda), u(\lambda)) \to (z, u)$ as $\lambda \to 0$;
        \item $z'(0)$ is not scaling of $u$;
        \item and $(z(\lambda), u(\lambda)) = \varphi(x, v(\lambda), t(\lambda))$ for some unknown time $t(\lambda)$.
    \end{itemize}
    See Figure \ref{fig: step 2}. By Proposition \ref{pp_traveltime},
    \[
    t(0) = \frac{-2}{h'(0)}g(z'(0), u(0)) = g(z'(0), u).
    \]
    We gather all such $t(0)$, and take the infimum, denoted by $\tau$. By construction we know $\tau \geq \tau(x, v)$ because each $(z, u)$ is after $(y, w)$. On the other hand, there exists a sequence of such $(z_j, u_j)$ in the exterior region converging to $(y, w)$, so the infimum is precisely $\tau(x, v)$.
    
    Note that if we are given $S|_{\overline{\mathcal{U}_-}}$, then we directly apply step 2 to recover the lightlike complete travel time. Thus the proof is finished.
\end{proof}

\section{Scattering rigidity}\label{sec: scattering rigidity}
We can now prove Theorem \ref{thm: interior scattering main thm} and Theorem \ref{thm: complete scattering main thm}, using the results from Section \ref{sec: determination of jet} and Section \ref{sec: interior and complete scattering information}.

\begin{proof}[Proof of Theorem \ref{thm: interior scattering main thm}]
    Since the lightlike geodesics are non-trapping, this holds for any timelike geodesics that are sufficiently close to the light cone, so $\mathcal{U}_j$ does exist. By the assumptions stated in the theorem, the non-conjugacy condition holds for at least one tangential direction in each connected component of $\partial M_1$, where $S^{in}_1$ is known for all transversal directions close to it. We can then apply Theorem \ref{thm: lorentzian determine jet} to recover the jet of $g_1$ at that point. 
    By analyticity of the boundary and the metric, this determines the jet of $g_1$ on each component and thus the entire $\partial M_1$.

    Similar to \cite{Var09}, even though the non-conjugacy condition only holds with respect to $g_1$, the jet of $g_2$ can be computed in the same way. This is because the non-conjugacy assumption is only used to find $y_{\epsilon_j}$, $y_0^*$, and a sequence of diffeomorphisms between $U^+_j$ and $\partial_- T_{y_{\epsilon_j}}M_1$ for each $j$ in the proof of Theorem \ref{thm: lorentzian determine jet}. Then, since the scattering data are the same, these diffeomorphisms still hold between $\varphi_0(U^+_j)$ and $(\varphi_0)_* (\partial_- T_{y_{\epsilon_j}}M_1)$ (recall we can view $(\varphi_0)_*$ as a map from $\partial TM_1$ to $\partial TM_2$ in boundary normal coordinates). One can then perform the same computation for the Eikonal equation with respect to $g_2$ and use the same scattering data, which shows that the jet of $g_2$ agrees with the jet of $g_1$ (through $\varphi_0$ of course). That is, for any $(x, v) \in T\partial M_1$, $k \in \mathbb{N}$, in boundary normal coordinates we have
    \[
    (\partial_n^k (g_1)_{\alpha \beta})(x) v^{\alpha} v^{\beta} = (\partial_n^k (g_2)_{\alpha \beta})(\varphi_0(x)) ((\varphi_0)_*v)^{\alpha} ((\varphi_0)_*v)^{\beta}.
    \]
    
    As a result, one can analytically extend $M_j$ to a slightly larger manifold $\tM_j$, similar to Section 2 of \cite{Var09}, also see Section \ref{sec: normal exponential extension}. The only difference being that the extension may not be uniform in size. Nevertheless, the determination of jet gives the following: there exists $U_j \subset \tM_j$ open set of $\partial M_j$ such that $\varphi_0$ can be extended to an analytic isometry on $U_j$ via normal exponential maps. We may assume $\tM_j = U_j \cup M_j$ and by abuse of notation denote $\varphi_0$ also as the isometry on $U_j$. Now we can apply Lemma \ref{lm: lorentzian recover complete from interior}. Since $\varphi_0$ clearly maps $\overline{\partial_- LM_1}$ to $\overline{\partial_- LM_2}$, we obtain from the lemma that the lightlike complete travel time data are equivalent. Finally we apply Proposition \ref{prop: complete light scattering recover metric} to obtain isometry between $M_1$ and $M_2$.
    
\end{proof}

\begin{proof}[Proof of Theorem \ref{thm: complete scattering main thm}]
    The proof is mostly the same as the proof of Theorem \ref{thm: interior scattering main thm}. For the recovery of the jet of boundary metric, we still use Theorem \ref{thm: lorentzian determine jet}. After extending to $U_j \cup M_j$, we again use Lemma \ref{lm: lorentzian recover complete from interior} to recover the lightlike complete travel time data. The rest of the proof is the same as Theorem \ref{thm: interior scattering main thm}.
\end{proof}

\appendix

\section{Transversal geodesics}\label{sec: transversal geodesics}
We include two lemmas here for finding transversal geodesics.
\begin{lm}\label{lm: transversal_local}
Let $(M,g)$ be a Riemannian or semi-Riemannian manifold and $H \subset M$ be a smooth hypersurface. 
Let $(x_0, v_0) \in SM$ and $W \subset M$ be a convex neighborhood of $x_0$ such that $H \cap W = \{x: \phi(x) = 0\}$, where $\phi$ is the smooth defining function.
Suppose $\phi(x_0) < 0$ and $\phi(y_0) > 0$ with $y_0 = \gamma_{x_0, v_0}(t_0)$ for some $t_0> 0$. 
Then there exists a small conic neighborhood $V_0 \subset S_{x_0}M$ of $v_0$ such that for almost every $v \in V_0$, the geodesic segment $\gamma_{x_0, v}((0, t_0))$ intersects with $H$ and the intersections are always transversal. 
\end{lm}
\begin{proof}
We prove the statement using Sard's Theorem. 
With the assumption $\phi(x_0) > 0$ and $\phi(y_0) < 0$, we can find a sufficiently small conic neighborhood $V_0 \subset S_{x_0}M$ of $v_0$, such that $\phi(\gamma_{x_0, v}(t_0)) < 0$ for any $v \in V_0$. 
Then we consider the smooth function
\[
f(v, t) = \phi(\gamma_{x_0, v}(t)), \quad \text{ for $(v, t) \in V_0 \times (0, t_0)$}
\]
and its zero set
\[
Z = \{(v,t) \in V_0 \times (0, t_0): f(v, t)  = 0\}.
\]
Note that $Z$ is a smooth one-codimensional submanifold of $V_0 \times (0, t_0)$, as we have $\diff_{(t,v)} f \neq 0$. 
Let $\pi: Z \rightarrow V_0$ be the projection given by $\pi(v, t) = v$.
By our choice of $V_0$, for each $v$ there, we can find $t_v$ such that $\phi(\gamma_{x_0, v}(t_v)) = 0$, which implies $(v, t_v) \in Z$ and therefore $\pi(Z) = V_0$.

Now consider the differential $\diff \pi_{(v,t)}: T_{(v,t)} Z \rightarrow T_v V_0$. 
Let $(\delta v, \delta t) \in T_{(v,t)} Z$, which satisfies
\[
\diff_v f (\delta v) + \diff_t f(\delta_t) = 0. 
\]
On the one hand, when $\gamma_{x_0,v}$ intersects $H$ transversally at $t = t_v$, the differential $\diff \pi$ is surjective with $\mathrm{rank}(\diff \pi) = n-1$. 
Indeed, in this case, one has $\partial_t f(v, t) = \diff \phi(\dot{\gamma}_{x_0,t}(t_v)) \neq 0$. 
Then we can solve $\delta_t$ in terms of $\delta_v$ from the condition above. 
This implies the map $\diff \pi_{(v,t)}(\delta v, \delta t) = \delta_v$ is surjective and we have the desired rank. 
On the other hand, when $\gamma_{x_0,v}$ intersects $H$ tangentially at $t = t_v$, all $(\delta v, \delta t) \in T_{(v,t)} Z$ must satisfy
\[
0 = \diff_v f (\delta v) = \diff \phi \circ \diff_v \exp_{x_0}|_{t_v v} (\delta v). 
\]
With $\gamma_{x_0, v}(t_v) \in W \cap H$, the differential $\diff_v \exp_{x_0}|_{t_v v}$ is invertible and therefore $\diff \phi \circ \diff_v \exp_{x_0}|_{t_v v} \neq 0$. 
It follows that the condition above gives restriction for $\delta v$ and the map $\diff \pi_{(v,t)}(\delta v, \delta t) = \delta_v$ has rank less than $n-1$. 
If we denote the critical set of $\pi$ by
\[
T_0 = \{(t,v) \in Z: \mathrm{rank}(\diff \pi) < n-1\}, 
\]
then the arguments above proves
\[
T_0 = \{(t, v) \in Z: \partial_t f(t, v) = 0\}. 
\]
By Sard's Theorem, the image $\pi(T_0)$ has Lebesgue measure zero in $V_0 = \pi(Z)$, where $\pi(T_0)$ is the set of directions of which the geodesics intersect $H$ tangentially before $t = t_0$. 
Thus, we prove the desired result.
\end{proof}

\begin{lm}\label{lm: transversal_global}
Let $(M,g)$ be a Riemannian or semi-Riemannain manifold with smooth boundary and $(\tM, \tilde{g})$ be its extension.
Let $(x_0, v_0) \in S \tM$ with $x_0 \in \tM \setminus M$. 
Suppose the geodesic $\gamma_{x_0, v_0}$ satisfies $\gamma_{x_0, v_0}((0, \delta)) \subset M^\circ$ and then $\gamma(t_0) \in \tM \setminus M$ for some $t_0 > \delta > 0$. 
Moreover, suppose $x_0$ is not conjugate to any point in $\gamma_{x_0, v_0}([0, t_0]) \cap \partial M$.  
Then there exists a small conic neighborhood $V_0 \subset S_{x_0}M$ of $v_0$ such that for almost every $v \in V_0$, the geodesic segment $\gamma_{x_0, v}((0, t_0))$ intersects with $\partial M$ and the intersections are always transversal. 
\end{lm}
\begin{proof}
We consider the set 
\[
P_0 = \gamma_{x_0, v_0}([0, t_0]) \cap \partial M.
\]
It is compact and we can pick $N$ points $p_1, \ldots, p_N \in P_0$ such that $P_0$ is covered by the convex neighborhoods of these points, i.e., 
$
P_0 \subset \cap_{j = 1}^N W_j, 
$
where $W_j$ is the convex neighborhood of $p_j$. 
By choosing sufficiently small $V_0 \subset S_{x_0} M$, we can assume $\gamma_{x_0, v}((0, t_0))$ always intersects $\partial M$ at some points in $\cap_{j = 1}^N W_j$. 

With the assumption on conjugate points, we may shrink these neighborhood $W_j$ and $V_0$ such that for any $v \in V_0$, the exponential map $\exp_{x_0}: t v \mapsto y = \exp_{x_0}(tv) \in W_j \cap \partial M$ for some $t> 0$ and $1 \leq j \leq N$ has non-degenerate differential and therefore is a local diffeomorphism. 
Now for fixed $j$, let $\phi_j$ be the smooth boundary defining function for $W_j \cap \partial M$, when $W_j$ is sufficiently small. 
We perform the same argument as in Lemma \ref{lm: transversal_local}.
More explicitly, we define
\[
f_j(v, t) = \phi_j(\gamma_{x_0, v}(t)), \quad \text{ for $(v, t) \in V_0 \times (0, t_0)$}
\]
and its zero set
\[
Z_j = \{(v,t) \in V_0 \times (0, t_0): f_j(v, t)  = 0\}.
\]
Similarly, $Z_j$ is a smooth one-codimensional submanifold of $V_0 \times (0, t_0)$, as we have $\diff_{(t,v)} f_j \neq 0$ by the  assumption. 
Note that $\pi(Z_j)$ might be empty but we always have $\pi(Z_j) \subset V_0$ by its definition. 
If $\pi(Z_j) \neq \emptyset$,
then as we denote the critical set of $\pi$ by
\[
T_{0,j} = \{(t,v) \in Z: \mathrm{rank}(\diff \pi) < n-1\}.
\]
The same arguments as in the proof of Lemma \ref{lm: transversal_local} proves
$
T_{0,j} = \{(t, v) \in Z: \partial_t f(t, v) = 0\}. 
$
By the Sard's Theorem, 
the set $\pi(T_{0,j})$ has Lebesgue measure zero in $\pi(Z_j)$ and therefore has measure zero in $V_0$. 
Thus, we consider all elements $v \in V_0 \backslash \bigcup_{j=1}^N\pi(T_{0,j})$. This proves the desired result.

\end{proof}

\section{Construction of jet with strictly convex direction}\label{sec: construction of jet with strictly convex direction}

We include here a constructive proof of recovering the normal jet of the metric on the boundary. 
We derive the one-sided Taylor expansion of the travel time, and we inductively recover the higher order jet from the expansion.
In \cite{SU05}, similar ideas are used to prove the stability of recovery.

\subsection{Setup}
Let $(M, g)$ be a Riemannian manifold with boundary or a Lorentzian manifold with boundary. Suppose the metric around some $p \in \partial M$ and $v \in T_p\partial M$ is strictly convex in the $v$ direction. For the Lorentzian setting, we assume the boundary near $p$ is either timelike or spacelike, and $v$ is either timelike or spacelike (we do not consider the case $v$ is lightlike: since strict convexity is an open condition, one can always choose a timelike or spacelike direction that is also strictly convex). Here we define strict convexity to be
\[
\sff(v,v) = g(\nabla_v \nu, v) > 0,
\]
where $\sff$ is the second fundamental form and $\nu$ is the unit outward normal vector. Since the manifold is either Riemannian or Lorentzian with timelike or spacelike boundary, we can write the metric locally in semi-geodesic normal coordinate around $p$ as
\[
g = g_{\alpha\beta}dx^\alpha dx^\beta + (dx^n)^2, \quad \alpha,\beta = 1,\dots, n-1,
\]
where the interior is $x^n > 0$. Then strict convexity translates to
\[
0 < \sff(v, v) = g(\nabla_v \nu, v) = -g(\nabla_v \partial_n, v)  = -\frac{1}{2}\partial_ng_{\alpha\beta}v^\alpha v^\beta.
\]
Note that in this case, for directions near $(p, v)$, interior and complete scattering information coincide. We will first show that if the scattering relation is known in a neighborhood of $v \in T_p\partial M$, then the jet of $g$ can be constructed from the scattering relation $S$ around $(p, v)$ and the boundary metric $g|_{T\partial M \times T\partial M}$ around $p$.


Let $v_\varepsilon \in T_p\partial M$ be a one parameter family of directions approaching $v$, the scattering relation is known for $\varepsilon \in [0, \delta)$ where $\delta \ll 1$. In the Riemannian setting we choose $v_\varepsilon = \sqrt{1-\varepsilon^2}v+\varepsilon \partial_n$; in the Lorentzian setting we choose $v_\varepsilon = \sqrt{1+\varepsilon^2}v+\varepsilon \partial_n$.
Then the travel time (exit time) $\tau(\varepsilon) := \tau(x, v_\varepsilon)$ as a function of $\varepsilon$ is smooth in $(0, \delta)$ by the Implicit Function Theorem, since the geodesics corresponding to $v_\varepsilon$, denoted by $\gamma_\varepsilon$, leave the manifold transversally by strict convexity.

To avoid confusion, we use the notation $V_\varepsilon(s) = \dot{\gamma}_\varepsilon(s)$ to mean the vector field, in which case $V_\varepsilon^j(f)$ means $V_\varepsilon$ acting on the function $f$ for $j$ times; while $v^j_\varepsilon(s) = dx^j(V_\varepsilon)(s)$ is used as the $j$-th component of the geodesic $\gamma_\varepsilon$ at time $s$. We shall also omit $s = 0$ most of the time and simply write $v^j_\varepsilon = v^j_\varepsilon(0)$. Then
\begin{align}
    x^n(\gamma_\varepsilon(s)) &= \sum_{j=0}^N V_\varepsilon^j(x^n)|_{s=0}\cdot \frac{1}{j!}s^j + O(s^{N+1})\\
    &= \varepsilon s + \frac{1}{4}\partial_ng_{\alpha\beta}v^\alpha_\varepsilon v^\beta_\varepsilon \cdot s^2 + \sum_{j=3}^N V_\varepsilon^{j-1}(v^n_\varepsilon)|_{s=0}\cdot \frac{1}{j!}s^j + O(s^{N+1}).
\end{align}
For the last equality, we used the fact that $\Gamma^n_{\alpha\beta} = -\frac{1}{2}\partial_n g_{\alpha\beta}$, and
\begin{equation}\label{eq: V_ep(v^j_ep)}
    0 = dx^\alpha(\nabla_{V_\varepsilon}V_\varepsilon) = V_\varepsilon(v^\alpha_\varepsilon)+\Gamma^\alpha_{kl}v^k_\varepsilon v^l_\varepsilon, \quad 0 = dx^n(\nabla_{V_\varepsilon}V_\varepsilon) = V_\varepsilon(v^n_\varepsilon)+\Gamma^n_{\alpha\beta}v^\alpha_\varepsilon v^\beta_\varepsilon
\end{equation}
because $\Gamma^n_{nl} = 0$ for $l = 1, \dots, n$.
In particular, by the choice of $v_\varepsilon$, we have
\[
V_\varepsilon(0) = v_\varepsilon^\alpha \partial_\alpha + \varepsilon \partial_n.
\]
Denote
\begin{equation}\label{eq: K_varepsilon}
    K_\varepsilon = \frac{1}{2}\partial_ng_{\alpha\beta}v^\alpha_\varepsilon v^\beta_\varepsilon = \frac{1}{2}\partial_ng_{\alpha\beta}v^\alpha v^\beta + O(\varepsilon^2) := K + O(\varepsilon^2).
\end{equation}
Note that by strict convexity, we know $K, K_\varepsilon \neq 0$ for $\varepsilon \ll 1$, in fact they are uniformly bounded away from 0. Then at $s = \tau(\varepsilon)$, we have
\begin{align}
    \tau(\varepsilon) &= -2K_\varepsilon^{-1}\left[ \varepsilon + \sum_{j=2}^{N-1} V_\varepsilon^j(v^n_\varepsilon)|_{s=0}\cdot \frac{1}{(j+1)!}\tau(\varepsilon)^j  \right]+ O(\tau(\varepsilon)^N)\\
    &= -2K_\varepsilon^{-1}\left[ \varepsilon + \sum_{j=2}^{N-1} V_\varepsilon^j(v^n_\varepsilon)|_{s=0}\cdot \frac{1}{(j+1)!}\tau(\varepsilon)^j  \right]+ O(\varepsilon^N)\label{eq: expansion}
\end{align}
by repeatedly substituting the left hand side to the right hand side. Note that the terms are all finite because the only term appearing in the denominator at each level is powers of $K_\varepsilon$ which is bounded away from 0.

The construction can be outlined as follow. The above computation shows that the one-sided Taylor expansion of the travel time (exit time) function is well-defined at $\varepsilon=0^+$. We then show that the $m$-th normal derivative can be recovered from the $\varepsilon^{2m-1}$ term. To do this, we need to show that the term in the $\varepsilon^{2m-1}$ level that contains $\partial_n^mg_{\alpha\beta}$ has non-zero coefficient, and that all the other terms contains only $\partial_n^kg_{\alpha\beta}$ with $k < m$. Tangential derivatives do not pose problem: if $\partial_n^k g_{\alpha\beta}$ can be recovered at $p$, then it can be recovered in a neighborhood of $p$ since strict convexity holds for a neighborhood of $(p, v) \in T\partial M$.

\begin{example}
    Consider $B_1(0,1) \subset \mathbb{R}^2$ the unit ball centered at $(0,1)$, equipped with the Euclidean metric. Consider lines of slope $k$ that passes through the origin, we may write the length of the segment inside of $B_1(0,1)$ as a function of $k$, denote by $l(k)$. Straightforward computation shows that $l(k) = \frac{2|k|}{\sqrt{k^2+1}}$. Even though it is not differentiable at $k=0$, the coefficients of the Taylor expansion at $k >0$ will converge to a finite number when $k \to 0$. That is, the one-sided Taylor expansion at 0 is well-defined.
\end{example}

\subsection{Analyze contribution of each term}
In this subsection, we analyze what terms may contribute to the coefficient of $\varepsilon^{2m-1}\partial_n^mg_{\alpha\beta}v^\alpha v^\beta$ for $m \geq 2$.

First of all, it is straightforward to see that if $j < m$, then $V^j_\varepsilon(v^n_\varepsilon)|_{s=0}$ does not contain any term related to $\partial_n^mg_{\alpha\beta}$. On the other hand, $\tau(\varepsilon)$ is $\varepsilon$ level, so if $j > 2m-1$, then $\tau(\varepsilon)^j = O(\varepsilon^{2m})$. Therefore the contribution must come from
\begin{align}
    &-2K_\varepsilon^{-1}\sum_{j=m}^{2m-1} V_\varepsilon^j(v^n_\varepsilon)|_{s=0}\cdot \frac{1}{(j+1)!}\tau(\varepsilon)^j \\
    &= -2K_\varepsilon^{-1}\tau(\varepsilon)^m\sum_{j=0}^{m-1} V_\varepsilon^{m+j}(v^n_\varepsilon)|_{s=0}\cdot \frac{1}{(m+j+1)!}\tau(\varepsilon)^j.
\end{align}

We now make the following observation. The term $V^k_\varepsilon(v^n_\varepsilon)$ contains a finite sum of the product of tangential derivatives of $v^n_\varepsilon, v^\alpha_\varepsilon, \partial^l_ng_{\alpha\beta}$, and
\begin{align*}
    &V_\varepsilon(v_\varepsilon^n) = \frac{1}{2}\partial_ng_{\alpha\beta}v_\varepsilon^\alpha v_\varepsilon^\beta, \quad V_\varepsilon(v^\alpha_\varepsilon) = -\frac{1}{2}g^{\alpha\beta}\partial_ng_{\beta\gamma}v^n_\varepsilon v^\gamma_\varepsilon - \Gamma^\alpha_{\beta\gamma}v^\beta_\varepsilon v^\gamma_\varepsilon,\\
    &V_\varepsilon(\partial_n^lg_{\gamma\beta}) = v^n_\varepsilon\partial_n^{l+1}g_{\alpha\beta} + v^\alpha_\varepsilon\partial_\alpha\partial_n^lg_{\alpha\beta}.
\end{align*}
Since $\tau(\varepsilon)^{m+j} = O(\varepsilon^{m+j})$, to obtain the desired term $\varepsilon^{2m-1}\partial_n^mg_{\alpha\beta}v^\alpha v^\beta$, the contribution from $V^{m+j}_\varepsilon(v^n_\varepsilon)|_{s=0}$ can contain at most $\varepsilon^{m-j-1}$. Meanwhile, \eqref{eq: V_ep(v^j_ep)} tells us that
\begin{equation}\label{eq: V^m+j and V^m+j-1}
    V^{m+j}_\varepsilon(v^n_\varepsilon) = \frac{1}{2}V_\varepsilon^{m+j-1}(\partial_ng_{\alpha\beta}v^\alpha_\varepsilon v^\beta_\varepsilon),
\end{equation}
so to obtain $\partial_n^m g_{\alpha\beta}$, we need $v^n_\varepsilon \partial_n$ to repeatedly apply to $\partial_ng_{\alpha\beta}$ exactly $m-1$ times. One byproduct of this action is $(v^n_\varepsilon)^{m-1}$ which is $\varepsilon^{m-1}$ when evaluated at $s = 0$, and exceeds the $\varepsilon^{m-j-1}$ allowance for $j > 0$. Fortunately, applying $V_\varepsilon$ to $(v^n_\varepsilon)^l$ gives
\begin{equation}
    V_\varepsilon(v^n_\varepsilon)^l = l(v^n_\varepsilon)^{l-1} \cdot \frac{1}{2}\partial_ng_{\alpha\beta}v^\alpha_\varepsilon v^\beta_\varepsilon = lK_\varepsilon(v^n_\varepsilon)^{l-1} = lK(v^n_\varepsilon)^{l-1} + \varepsilon^2 lK(v^n_\varepsilon)^{l-1}
\end{equation}
by \eqref{eq: K_varepsilon} and \eqref{eq: V_ep(v^j_ep)}, which effectively brings down the power of $\varepsilon$ from $l$ to $l-1$ when evaluated at $s = 0$. To bring it down from $m-1$ to $m-j-1$ we thus need at least $j$ times $V_\varepsilon$ acting on powers of $v^n_\varepsilon$. Since we only have $m+j-1$ times $V_\varepsilon$ acting on $\partial_ng_{\alpha\beta}$, we need $m-1$ of them to repeatedly act on the derivative of $\partial_ng_{\alpha\beta}$, and $j$ of them to bring down the power of $v^n_\varepsilon$. In other words, every single $V_\varepsilon$ must be used to either raise the degree of normal derivative, or bring down the power of $v^n_\varepsilon$, and can not be wasted on any other terms.

Combine the above analysis with the fact that
\begin{align}
    K_\varepsilon = K+O(\varepsilon^2), \quad v^\alpha_\varepsilon v^\beta_\varepsilon = v^\alpha v^\beta + O(\varepsilon^2),
\end{align}
in each $V_\varepsilon^{m+j}(v^n_\varepsilon)$, the term that contributes to $\varepsilon^{2m-1}\partial_n^mg_{\alpha\beta}v^\alpha v^\beta$ must only come directly from the zeroth order of the coefficient of $(v^n_\varepsilon)^{m-j-1}\partial_n^mg_{\alpha\beta}v^\alpha_\varepsilon v^\beta_\varepsilon$. We denote the zeroth order of the coefficient as $C_j$. After expanding $K_\varepsilon$, $v^\alpha_\varepsilon$ and use the fact that $\tau(\varepsilon) = -2K^{-1}\varepsilon + O(\varepsilon^2)$, the coefficient of $\varepsilon^{2m-1}\partial_n^mg_{\alpha\beta}v^\alpha v^\beta$ is
\begin{equation}\label{eq: full coeff before computation}
    \sum_{j=0}^{m-1}\frac{C_j}{(m+j+1)!}(-2K^{-1})^{m+j+1}.
\end{equation}

\subsection{Computation of the coefficient}
We compute $C_j$ in this subsection. By the analysis from the previous section, we only need to keep track of the coefficient when $V_\varepsilon$ acts on either the power of $v^n_\varepsilon$ or highest normal derivative of $g$. To make this more clear, we compute the first several terms to illustrate what we are keeping track of.

When $V_\varepsilon$ first acts on $v^n_\varepsilon$ we obtain $\frac{1}{2}\partial_n g_{\alpha\beta}v^\alpha_\varepsilon v^\beta_\varepsilon$. Then the second $V_\varepsilon$ must directly apply to $\partial_ng_{\alpha\beta}$, specifically only the $v^n_\varepsilon\partial_n$ part of $V_\varepsilon$ would matter, which gives $\frac{1}{2}v^n_\varepsilon \partial_n^2g_{\alpha\beta}v^\alpha_\varepsilon v^\beta_\varepsilon$ plus irrelevant terms. These irrelevant terms will stay irrelevant when applying all future $V_\varepsilon$ by the analysis from the previous section. The next $V_\varepsilon$ now has a choice, it can either apply to $v^n_\varepsilon$, thus acting as lowering the power; or it can apply to $\partial_n^2g_{\alpha\beta}$, thus acting as raising the degree of normal derivative. All the other terms would stay irrelevant, and we keep applying $V_\varepsilon$ to the relevant terms as such. For a detailed computation, see Section \ref{subsection: m=2}.

Thus we need to gather the different coefficients caused by lowering the power or increasing the degree of derivative in different orders. Denote $R^{j, k}$ the relevant term that contains $(v^n_\varepsilon)^j\partial_n^kg_{\alpha\beta}$, then the relevant term in $V_\varepsilon(R^{j,k})$ is thus
\begin{equation}
    jK_\varepsilon R^{j-1,k} + R^{j+1,k+1},
\end{equation}
where the first term is obtained from acting on $(v^n_\varepsilon)^j$ (lowering the power) and use \eqref{eq: V_ep(v^j_ep)}; and the second term is obtained from the $v^n_\varepsilon \partial_n$ part of $V_\varepsilon$ acting on $\partial_n^kg_{\alpha\beta}$ (raising the degree of derivative). Ignore the irrelevant term, we obtain the following relation
\begin{equation}
    V(R^{j,k}) = jKR^{j-1,k} + R^{j+1,k+1}
\end{equation}
where for simplicity we omit the lower index in $V_\varepsilon$ and use the zeroth order term of $K_\varepsilon$ from \eqref{eq: K_varepsilon} as it is the term that would contribute to the final coefficient.

\begin{lm}
    Given the following rule
    \begin{equation}
        V(R^{j,k}) = jKR^{j-1,k} + R^{j+1,k+1}.
    \end{equation}
    For $l \geq 0$, one has
    \begin{equation}
        V^l(R^{0,1}) = \sum_{d=0}^{\lfloor l/2\rfloor}\frac{l!}{(l-2d)!d!2^d}K^dR^{l-2d, 1+l-d}.
    \end{equation}
\end{lm}
\begin{proof}
    We prove by induction, $l=0$ is trivial. Suppose this holds for some $l$, then
    \begin{align*}
        V^{l+1}(R^{0,1}) &= \sum_{d=0}^{\lfloor l/2\rfloor}\frac{l!}{(l-2d)!d!2^d}K^dV(R^{l-2d, 1+l-d})\\
        &=\sum_{d=0}^{\lfloor l/2\rfloor}\frac{l!}{(l-2d)!d!2^d} (l-2d) K^{d+1}R^{l-2d-1, 1+l-d}\\
        &+\sum_{d=0}^{\lfloor l/2\rfloor}\frac{l!}{(l-2d)!d!2^d}K^d R^{l-2d+1, 2+l-d}\\
        &=\left. \frac{l!}{(l-2d)!d!2^d} (l-2d) K^{d+1} R^{l-2d-1, 1+l-d} \right|_{\text{if }l \text{ odd and } d=\frac{l-1}{2}}\\
        &+ \sum_{d=0}^{\lfloor l/2\rfloor-1}\frac{l!}{(l-2d)!d!2^d} (l-2d) K^{d+1} R^{l-2d-1, 1+l-d}\\
        &+R^{l+1,2+l} + \sum_{d=1}^{\lfloor l/2\rfloor}\frac{l!}{(l-2d)!d!2^d}K^d R^{l-2d+1, 2+l-d}\\
        &=\left. \frac{l!}{(\frac{l-1}{2})!2^{\frac{l-1}{2}}} K^{\frac{l-1}{2}+1} R^{0, 1+(l+1)-\frac{l+1}{2}} \right|_{\text{if }l \text{ odd}} + R^{l+1,1+(l+1)}\\
        &+\sum_{d=1}^{\lfloor l/2 \rfloor} \frac{l!}{(l-2d)!(d-1)!2^{d-1}} \left[ \frac{1}{l-2d+1} + \frac{1}{2d} \right]K^dR^{l+1-2d, 1+(l+1)-d}\\
        &=\left. \frac{(l+1)!}{(\frac{l+1}{2})!2^{\frac{l+1}{2}}} K^{\frac{l-1}{2}+1} R^{0, 1+(l+1)-\frac{l+1}{2}} \right|_{\text{if }l \text{ odd}}\\
        &+\sum_{d=0}^{\lfloor l/2 \rfloor} \frac{(l+1)!}{(l+1-2d)!d!2^d} K^dR^{l+1-2d, 1+(l+1)-d}.
    \end{align*}
    For the last equality, we used the fact that $R^{l+1,1+(l+1)}$ is precisely the $d=0$ term. When $l$ is even, we do not have the first term and $\lfloor \frac{l+1}{2}\rfloor = \lfloor \frac{l}{2} \rfloor$; and when $l$ is odd, the $d = \lfloor \frac{l+1}{2}\rfloor = \frac{l+1}{2}$ term is the first term.
\end{proof}

Since we start from $R^{0,1}$, by \eqref{eq: V^m+j and V^m+j-1} we substitute $l = m + j-1$ and $d = l+1-m = j$. Indeed for this $d$ and $l$ we have $R^{l-2d,1+l-d} = R^{m-j-1, m}$ which corresponds to $(v^n_\varepsilon)^{m-j-1}\partial^m_ng_{\alpha\beta}$ as required. From \eqref{eq: V^m+j and V^m+j-1}, we have $C_j$ is the coefficient of the $d = j$ term in $V^{m+j-1}(R^{0,1})$ multiply by $1/2$:
\begin{equation}
    C_j = \frac{1}{2}\frac{(m+j-1)!}{(m-j-1)!j!2^j}K^j.
\end{equation}
Plug into \eqref{eq: full coeff before computation}, the entire coefficient is
\begin{equation}\label{eq: full coeff}
    \frac{1}{2}(-2K^{-1})^{m+1}\sum_{j=0}^{m-1}\frac{(-1)^j}{(m+j)(m+j+1)(m-j-1)!j!} = \frac{1}{2}(-2K^{-1})^{m+1}\frac{m!}{(2m)!}
\end{equation}
by Lemma \ref{lm: series 2}.

\begin{lm}\label{lm: series 2}
For any integer  $m\ge 1$, we have 
\[
\sum_{j=0}^{m-1} \frac{(-1)^j}{(m+j)(m+j+1)(m-j-1)!j!}
    = \frac{m!}{(2m)!}.
    \]
\end{lm}
\begin{proof}
We denote the left-hand side by $S_m$.
First, we observe
\[
\frac{1}{(m+j)(m+j+1)} = \frac{1}{m+j} - \frac{1}{m+j+1} = \int_0^1 t^{\,m+j-1}(1-t)dt.
\]
This implies 
\begin{align*}
S_m &=\sum_{j=0}^{m-1}\frac{(-1)^j}{(m-j-1)!j!}
\int_0^1 t^{\,m+j-1}(1-t) dt\\
& =\int_0^1 t^{m-1}(1-t)
\sum_{j=0}^{m-1}\frac{(-t)^j}{(m-1-j)!j!}dt,
\end{align*}
where we swap the finite sum with the integral. 
On the other hand, using the binomial identity, we have
\begin{align*}
(1-t)^{m-1} = \sum_{j=0}^{m-1}\binom{m-1}{j} (-t)^j =  (m-1)!\sum_{j=0}^{m-1}\frac{(-t)^j}{(m-1-j)!j!}.
\end{align*}
It follows that
\[
S_m=\frac{1}{(m-1)!}\int_0^1 t^{m-1}(1-t)(1-t)^{m-1} dt
=\frac{1}{(m-1)!}\int_0^1 t^{m-1}(1-t)^m dt.
\]
A straightforward computation shows
\[
\int_0^1 t^{m-1}(1-t)^m dt =\frac{(m-1)! m!}{(2m)!}.
\]
Thus, we have 
\[
S_m=\frac{1}{(m-1)!}\cdot\frac{(m-1)!m!}{(2m)!}
=\frac{m!}{(2m)!}.
\]
\end{proof}

\subsection{Recover normal jet}
In this subsection we recover the normal jet of the metric. From the previous sections, we have computed the expansion of the travel time $\tau$ in terms of $\varepsilon$, see \eqref{eq: expansion}. We have also computed that the coefficient for $\varepsilon^{2m-1}\partial_n^mg_{\alpha\beta}v^\alpha v^\beta$ is given by \eqref{eq: full coeff} which is bounded away from 0. Now it suffices to show that all the other terms of $\varepsilon^{2m-1}$ level only contains lower order normal derivative of the metric. This can be achieved by using the same argument we used to find contributions.

Certainly contribution from $V^j_\varepsilon(v^n_\varepsilon)$ for $j < m$ is fine since it has at most $m-1$ degree of normal derivative of $g$. When $j = m$ we have contribution from $m$-th normal derivative, which is included in \eqref{eq: full coeff} already, all the other terms have lower order normal derivative. When $m < j< 2m-1$, again the contribution from $m$-th normal derivative term has already been computed; as for higher order normal derivative term, this requires at least $m+1$ times $V_\varepsilon$ for raising the degree of normal derivative, but the byproduct is $m-1$ power of $v^n_\varepsilon$, which requires at least $j-m$ times $V_\varepsilon$ to bring down to $2m-1-j$ level (because $\tau(\varepsilon)^j = O(\varepsilon^j)$). However, there are only $j < (m+1)+(j-m)$ times $V_\varepsilon$, so the only other terms they can contribute to coefficient of $\varepsilon^{2m-1}$ are normal derivatives of the metric with degree strictly smaller than $m$. Finally, $j > 2m-1$ is at least $O(\varepsilon^{2m})$.

Since $K$ can be read from the coefficient of $\varepsilon$ by \eqref{eq: expansion} and gives $\partial_ng_{\alpha\beta}v^\alpha v^\beta$, the first order normal derivative can be recovered in a neighborhood of $p$ (this is because the boundary is strictly convex with respect to all nearby $(x, w)$). Inductively, we have that $\partial_n^mg_{\alpha\beta}v^\alpha v^\beta$ can be computed from the coefficient of $\varepsilon^{2m-1}$ term for all $m \geq 2$. Thus one can similarly compute $\partial_n^mg_{\alpha\beta}(x)w^\alpha w^\beta$ for an open set of tangential directions $(x, w)$ near $(p, v)$. This is enough information to construct the symmetric two tensor $\partial^m_ng_{\alpha\beta}$ at a neighborhood of $p$.

To conclude, we have proved the following result.
\begin{thm}\label{thm: construct jet}
    Let $(M, g)$ be either a Riemannian manifold with boundary, or a Lorentzian manifold with timelike boundary. Suppose the boundary is strictly convex with respect to some $(p, v) \in \partial TM$, with $v$ being timelike in the Lorentzian case. Suppose the scattering relation $S$ around $(p, v)$ and the boundary metric $g|_{T\partial M \times T\partial M}$ around $p$ are given. Then the normal jet of $g$ at $p$ can be constructed.
\end{thm}

\begin{remark}
    We claim that the same result holds even if $\sff$ vanishes around $p$, as long as some higher level of convexity holds. Specifically, suppose there exists $k \in \mathbb{Z}^{\geq 1}$ such that $\partial^j_ng_{\alpha\beta} v^\alpha v^\beta = 0$ for all $j< k$, but there exists some $v \in T_p\partial M$ such that $\partial^k_ng_{\alpha\beta} v^\alpha v^\beta < 0$. What we computed was the $k=1$ case, when $k > 1$, the expansion of $x^n(\gamma_\varepsilon(s))$ will be
    \[
    \varepsilon s+\sum_{j=k+1}^NV_\varepsilon^j(x^n)|_{s=0}\cdot \frac{1}{j!}s^j + O(s^{N+1}).
    \]
    Then we obtain the expansion of $\tau(\varepsilon)^k$ with the first term being $C_k^{-1}\varepsilon$. Here $C_k$ is a nonzero constant multiple of $\partial^k_ng_{\alpha\beta} v^\alpha v^\beta$, and by the assumption it is nonzero. One can analyze the higher order terms to recover the jet of $g$ in a similar way. We do not prove this here.
\end{remark}

\subsection{Example}\label{subsection: m=2}
As an example, let us compute for $m=2$. For simplicity we omit the $\varepsilon$ subindex, and we use $R^{j,k}$ to denote any term with $j$-th power of $v^n$ and at most $k$-th normal derivative of $g$.
\begin{align*}
    V(v^n) &= \frac{1}{2}\partial_ng_{\alpha\beta}v^\alpha v^\beta\\
    V^2(v^n) &= \frac{1}{2}v^n\partial_n^2g_{\alpha\beta}v^\alpha v^\beta + \frac{1}{2} \partial_\gamma\partial_ng_{\alpha\beta}v^\gamma v^\alpha v^\beta + \partial_ng_{\alpha\beta} V(v^\alpha)v^\beta\\
    &=\frac{1}{2}v^n\partial_n^2g_{\alpha\beta}v^\alpha v^\beta + R^{0,1} + R^{1,1}\\
    V^3(v^n) &= \frac{1}{2}K\partial_n^2g_{\alpha\beta}v^\alpha v^\beta + \frac{1}{2}(v^n)^2\partial_n^3g_{\alpha\beta}v^\alpha v^\beta + \frac{1}{2}v^n\partial_\gamma\partial_n^2g_{\alpha\beta}v^\gamma v^\alpha v^\beta + v^n\partial_n^2g_{\alpha\beta}V(v^\alpha)v^\beta\\
    &+\frac{1}{2}v^n\partial_\gamma \partial_n^2g_{\alpha\beta}v^\gamma v^\alpha v^\beta + \frac{1}{2}\partial_\omega\partial_\gamma \partial_ng_{\alpha\beta}v^\omega v^\gamma v^\alpha v^\beta \\
    &+ \frac{1}{2}\partial_\gamma \partial_n g_{\alpha\beta} V(v^\gamma)v^\alpha v^\beta + \partial_\gamma \partial_n g_{\alpha\beta} V(v^\alpha)v^\beta\\
    &+ v^n\partial_n^2g_{\alpha\beta}V(v^\alpha)v^\beta + \partial_\gamma \partial_n g_{\alpha\beta}v^\gamma V(v^\alpha) v^\beta +\partial_ng_{\alpha\beta}V^2(v^\alpha)v^\beta + \partial_n g_{\alpha\beta} V(v^\alpha)V(v^\beta)\\
    &= \frac{1}{2}K\partial_n^2g_{\alpha\beta}v^\alpha v^\beta + O(v^n) + R^{0, 1}.
\end{align*}
In the above computation we used the fact that $V$ apply to tangential $v^\alpha$ belongs to $R^{1,1} + R^{0,0}$. Use $R^j$ to denote the term that contains normal derivative of at most $j$, substitute the above computation into the expansion \eqref{eq: expansion}:
\begin{align*}
    \tau(\varepsilon) &= -2K_\varepsilon^{-1}\left[ \varepsilon + (\frac{1}{2}\varepsilon\partial_n^2g_{\alpha\beta}v^\alpha_\varepsilon v^\beta_\varepsilon + R^1)\cdot \frac{1}{6}(-2K^{-1}\varepsilon+O(\varepsilon^2))^2 \right. \\
    &\left. \qquad + (\frac{1}{2}K_\varepsilon \partial_n^2g_{\alpha\beta}v^\alpha_\varepsilon v^\beta_\varepsilon + O(\varepsilon) + R^1) \cdot \frac{1}{24}(-2K^{-1}\varepsilon + O(\varepsilon^2))^3\right] + O(\varepsilon^4)\\
    &= -2K^{-1}\varepsilon + R^1\varepsilon^2 \\
    &+ \left[ (-2K^{-1})^3\frac{1}{12} \partial_n^2g_{\alpha\beta}v^\alpha v^\beta -(-2K^{-1})^3\frac{1}{24}\partial_n^2g_{\alpha\beta}v^\alpha v^\beta + R^1 \right]\varepsilon^3 + O(\varepsilon^4)\\
    &= -2K^{-1}\varepsilon + R^1\varepsilon^2 + (-2K^{-1})^3\frac{1}{24}\partial_n^2g_{\alpha\beta}v^\alpha v^\beta \varepsilon^3 + R^1 \varepsilon^3 + O(\varepsilon^4).
\end{align*}
Plug $m = 2$ into \eqref{eq: full coeff}, we obtain
\[
\frac{1}{2}(-2K^{-1})^{2+1} \frac{2!}{4!} = (-2K^{-1})^3\frac{1}{24},
\]
which agrees with the explicit computation.

\begin{footnotesize}
    \bibliographystyle{abbrv}
    \bibliography{ref.bib}
\end{footnotesize}

\end{document}

%% file: cut_point.pdf_tex
\begingroup%
  \makeatletter%
  \providecommand\color[2][]{%
    \errmessage{(Inkscape) Color is used for the text in Inkscape, but the package 'color.sty' is not loaded}%
    \renewcommand\color[2][]{}%
  }%
  \providecommand\transparent[1]{%
    \errmessage{(Inkscape) Transparency is used (non-zero) for the text in Inkscape, but the package 'transparent.sty' is not loaded}%
    \renewcommand\transparent[1]{}%
  }%
  \providecommand\rotatebox[2]{#2}%
  \newcommand*\fsize{\dimexpr\f@size pt\relax}%
  \newcommand*\lineheight[1]{\fontsize{\fsize}{#1\fsize}\selectfont}%
  \ifx\svgwidth\undefined%
    \setlength{\unitlength}{365.0372747bp}%
    \ifx\svgscale\undefined%
      \relax%
    \else%
      \setlength{\unitlength}{\unitlength * \real{\svgscale}}%
    \fi%
  \else%
    \setlength{\unitlength}{\svgwidth}%
  \fi%
  \global\let\svgwidth\undefined%
  \global\let\svgscale\undefined%
  \makeatother%
  \begin{picture}(1,0.93462635)%
    \lineheight{1}%
    \setlength\tabcolsep{0pt}%
    \put(0,0){\includegraphics[width=\unitlength,page=1]{cut_point.pdf}}%
    \put(0.07976805,0.82461033){\color[rgb]{0.10196078,0.10196078,0.10196078}\makebox(0,0)[lt]{\lineheight{1.25}\smash{\begin{tabular}[t]{l}$\gamma(T)$\end{tabular}}}}%
    \put(0.55645938,0.67822781){\color[rgb]{0.10196078,0.10196078,0.10196078}\makebox(0,0)[lt]{\lineheight{1.25}\smash{\begin{tabular}[t]{l}$t_3$\\\end{tabular}}}}%
    \put(0.51288872,0.5716206){\color[rgb]{0.10196078,0.10196078,0.10196078}\makebox(0,0)[lt]{\lineheight{1.25}\smash{\begin{tabular}[t]{l}$s_2 = t_3'$\end{tabular}}}}%
    \put(0.64206685,0.47963128){\color[rgb]{0.10196078,0.10196078,0.10196078}\makebox(0,0)[lt]{\lineheight{1.25}\smash{\begin{tabular}[t]{l}$s_1$\end{tabular}}}}%
    \put(0.74748514,0.4895345){\color[rgb]{0.10196078,0.10196078,0.10196078}\makebox(0,0)[lt]{\lineheight{1.25}\smash{\begin{tabular}[t]{l}$t_2' = t_2$\end{tabular}}}}%
    \put(0.6372529,0.37788014){\color[rgb]{0.10196078,0.10196078,0.10196078}\makebox(0,0)[lt]{\lineheight{1.25}\smash{\begin{tabular}[t]{l}$t_2 - \epsilon/2$\end{tabular}}}}%
    \put(0.83990259,0.24070678){\color[rgb]{0.10196078,0.10196078,0.10196078}\makebox(0,0)[lt]{\lineheight{1.25}\smash{\begin{tabular}[t]{l}$t_1$\end{tabular}}}}%
    \put(0.73243099,0.14890874){\color[rgb]{0.10196078,0.10196078,0.10196078}\makebox(0,0)[lt]{\lineheight{1.25}\smash{\begin{tabular}[t]{l}$t_1-\epsilon/2$\\\end{tabular}}}}%
    \put(0.78585522,0.06748787){\color[rgb]{0.10196078,0.10196078,0.10196078}\makebox(0,0)[lt]{\lineheight{1.25}\smash{\begin{tabular}[t]{l}$\gamma(0)$\end{tabular}}}}%
  \end{picture}%
\endgroup%

%% file: cut_point_2.pdf_tex
\begingroup%
  \makeatletter%
  \providecommand\color[2][]{%
    \errmessage{(Inkscape) Color is used for the text in Inkscape, but the package 'color.sty' is not loaded}%
    \renewcommand\color[2][]{}%
  }%
  \providecommand\transparent[1]{%
    \errmessage{(Inkscape) Transparency is used (non-zero) for the text in Inkscape, but the package 'transparent.sty' is not loaded}%
    \renewcommand\transparent[1]{}%
  }%
  \providecommand\rotatebox[2]{#2}%
  \newcommand*\fsize{\dimexpr\f@size pt\relax}%
  \newcommand*\lineheight[1]{\fontsize{\fsize}{#1\fsize}\selectfont}%
  \ifx\svgwidth\undefined%
    \setlength{\unitlength}{365.0372747bp}%
    \ifx\svgscale\undefined%
      \relax%
    \else%
      \setlength{\unitlength}{\unitlength * \real{\svgscale}}%
    \fi%
  \else%
    \setlength{\unitlength}{\svgwidth}%
  \fi%
  \global\let\svgwidth\undefined%
  \global\let\svgscale\undefined%
  \makeatother%
  \begin{picture}(1,0.93462635)%
    \lineheight{1}%
    \setlength\tabcolsep{0pt}%
    \put(0,0){\includegraphics[width=\unitlength,page=1]{cut_point_2.pdf}}%
    \put(0.06205695,0.81723071){\color[rgb]{0.10196078,0.10196078,0.10196078}\makebox(0,0)[lt]{\lineheight{1.25}\smash{\begin{tabular}[t]{l}$\gamma(T)$\end{tabular}}}}%
    \put(0.64058709,0.61181121){\color[rgb]{0.10196078,0.10196078,0.10196078}\makebox(0,0)[lt]{\lineheight{1.25}\smash{\begin{tabular}[t]{l}$t_3'$\\\end{tabular}}}}%
    \put(0.53207574,0.59375953){\color[rgb]{0.10196078,0.10196078,0.10196078}\makebox(0,0)[lt]{\lineheight{1.25}\smash{\begin{tabular}[t]{l}$s_2$\end{tabular}}}}%
    \put(0.68044114,0.38222025){\color[rgb]{0.10196078,0.10196078,0.10196078}\makebox(0,0)[lt]{\lineheight{1.25}\smash{\begin{tabular}[t]{l}$s_1$\end{tabular}}}}%
    \put(0.72091838,0.51610111){\color[rgb]{0.10196078,0.10196078,0.10196078}\makebox(0,0)[lt]{\lineheight{1.25}\smash{\begin{tabular}[t]{l}$t_2$\end{tabular}}}}%
    \put(0.59592697,0.4442968){\color[rgb]{0.10196078,0.10196078,0.10196078}\makebox(0,0)[lt]{\lineheight{1.25}\smash{\begin{tabular}[t]{l}$t_2 - \epsilon/2$\end{tabular}}}}%
    \put(0.83990259,0.24070678){\color[rgb]{0.10196078,0.10196078,0.10196078}\makebox(0,0)[lt]{\lineheight{1.25}\smash{\begin{tabular}[t]{l}$t_1$\end{tabular}}}}%
    \put(0.72652728,0.15628835){\color[rgb]{0.10196078,0.10196078,0.10196078}\makebox(0,0)[lt]{\lineheight{1.25}\smash{\begin{tabular}[t]{l}$t_1-\epsilon/2$\\\end{tabular}}}}%
    \put(0.7784756,0.06158417){\color[rgb]{0.10196078,0.10196078,0.10196078}\makebox(0,0)[lt]{\lineheight{1.25}\smash{\begin{tabular}[t]{l}$\gamma(0)$\end{tabular}}}}%
  \end{picture}%
\endgroup%

%% file: fermi.pdf_tex
\begingroup%
  \makeatletter%
  \providecommand\color[2][]{%
    \errmessage{(Inkscape) Color is used for the text in Inkscape, but the package 'color.sty' is not loaded}%
    \renewcommand\color[2][]{}%
  }%
  \providecommand\transparent[1]{%
    \errmessage{(Inkscape) Transparency is used (non-zero) for the text in Inkscape, but the package 'transparent.sty' is not loaded}%
    \renewcommand\transparent[1]{}%
  }%
  \providecommand\rotatebox[2]{#2}%
  \newcommand*\fsize{\dimexpr\f@size pt\relax}%
  \newcommand*\lineheight[1]{\fontsize{\fsize}{#1\fsize}\selectfont}%
  \ifx\svgwidth\undefined%
    \setlength{\unitlength}{383.4651965bp}%
    \ifx\svgscale\undefined%
      \relax%
    \else%
      \setlength{\unitlength}{\unitlength * \real{\svgscale}}%
    \fi%
  \else%
    \setlength{\unitlength}{\svgwidth}%
  \fi%
  \global\let\svgwidth\undefined%
  \global\let\svgscale\undefined%
  \makeatother%
  \begin{picture}(1,1.00359999)%
    \lineheight{1}%
    \setlength\tabcolsep{0pt}%
    \put(0,0){\includegraphics[width=\unitlength,page=1]{fermi.pdf}}%
    \put(0.08678216,0.32448097){\color[rgb]{0,0,0}\makebox(0,0)[lt]{\lineheight{1.25}\smash{\begin{tabular}[t]{l}$U_1$\end{tabular}}}}%
    \put(0.25195555,0.13454859){\color[rgb]{0,0,0}\makebox(0,0)[lt]{\lineheight{1.25}\smash{\begin{tabular}[t]{l}$z$\end{tabular}}}}%
    \put(0.13287464,0.20569921){\color[rgb]{0,0,0}\makebox(0,0)[lt]{\lineheight{1.25}\smash{\begin{tabular}[t]{l}$z'$\\\end{tabular}}}}%
    \put(0.45797248,0.35200365){\color[rgb]{0,0,0}\makebox(0,0)[lt]{\lineheight{1.25}\smash{\begin{tabular}[t]{l}$y$\\\end{tabular}}}}%
    \put(0.51764519,0.44807096){\color[rgb]{0,0,0}\makebox(0,0)[lt]{\lineheight{1.25}\smash{\begin{tabular}[t]{l}$w$\end{tabular}}}}%
    \put(0.37233274,0.23056111){\color[rgb]{0,0,0}\makebox(0,0)[lt]{\lineheight{1.25}\smash{\begin{tabular}[t]{l}$u$\end{tabular}}}}%
    \put(0.64768392,0.77505915){\color[rgb]{0,0,0}\makebox(0,0)[lt]{\lineheight{1.25}\smash{\begin{tabular}[t]{l}$x$\end{tabular}}}}%
    \put(0.53791968,0.93152569){\color[rgb]{0,0,0}\makebox(0,0)[lt]{\lineheight{1.25}\smash{\begin{tabular}[t]{l}$v$\end{tabular}}}}%
    \put(0.74771807,0.83206484){\color[rgb]{0,0,0}\makebox(0,0)[lt]{\lineheight{1.25}\smash{\begin{tabular}[t]{l}$v'$\\\end{tabular}}}}%
    \put(0,0){\includegraphics[width=\unitlength,page=2]{fermi.pdf}}%
    \put(0.68211066,0.33132186){\color[rgb]{0,0,0}\makebox(0,0)[lt]{\lineheight{1.25}\smash{\begin{tabular}[t]{l}$M_1$\end{tabular}}}}%
  \end{picture}%
\endgroup%

%% file: flow.pdf_tex
\begingroup%
  \makeatletter%
  \providecommand\color[2][]{%
    \errmessage{(Inkscape) Color is used for the text in Inkscape, but the package 'color.sty' is not loaded}%
    \renewcommand\color[2][]{}%
  }%
  \providecommand\transparent[1]{%
    \errmessage{(Inkscape) Transparency is used (non-zero) for the text in Inkscape, but the package 'transparent.sty' is not loaded}%
    \renewcommand\transparent[1]{}%
  }%
  \providecommand\rotatebox[2]{#2}%
  \newcommand*\fsize{\dimexpr\f@size pt\relax}%
  \newcommand*\lineheight[1]{\fontsize{\fsize}{#1\fsize}\selectfont}%
  \ifx\svgwidth\undefined%
    \setlength{\unitlength}{540.58517336bp}%
    \ifx\svgscale\undefined%
      \relax%
    \else%
      \setlength{\unitlength}{\unitlength * \real{\svgscale}}%
    \fi%
  \else%
    \setlength{\unitlength}{\svgwidth}%
  \fi%
  \global\let\svgwidth\undefined%
  \global\let\svgscale\undefined%
  \makeatother%
  \begin{picture}(1,0.38460092)%
    \lineheight{1}%
    \setlength\tabcolsep{0pt}%
    \put(0,0){\includegraphics[width=\unitlength,page=1]{flow.pdf}}%
    \put(-0.00006272,0.22275952){\color[rgb]{0,0,0}\makebox(0,0)[lt]{\lineheight{1.25}\smash{\begin{tabular}[t]{l}$x$\end{tabular}}}}%
    \put(0.03113868,0.07592933){\color[rgb]{0,0,0}\makebox(0,0)[lt]{\lineheight{1.25}\smash{\begin{tabular}[t]{l}$v$\end{tabular}}}}%
    \put(0.10307976,0.09427781){\color[rgb]{0,0,0}\makebox(0,0)[lt]{\lineheight{1.25}\smash{\begin{tabular}[t]{l}$v(\lambda)$\end{tabular}}}}%
    \put(0.52440038,0.00641066){\color[rgb]{0,0,0}\makebox(0,0)[lt]{\lineheight{1.25}\smash{\begin{tabular}[t]{l}$y$\end{tabular}}}}%
    \put(0.67018259,0.36300628){\color[rgb]{0,0,0}\makebox(0,0)[lt]{\lineheight{1.25}\smash{\begin{tabular}[t]{l}$z'(0)$\end{tabular}}}}%
    \put(0.17800479,0.20320892){\color[rgb]{0,0,0}\makebox(0,0)[lt]{\lineheight{1.25}\smash{\begin{tabular}[t]{l}$M$\end{tabular}}}}%
    \put(0.16613802,0.2800789){\color[rgb]{0,0,0}\makebox(0,0)[lt]{\lineheight{1.25}\smash{\begin{tabular}[t]{l}$M^c$\end{tabular}}}}%
    \put(0.70506491,0.25851189){\color[rgb]{0,0,0}\makebox(0,0)[lt]{\lineheight{1.25}\smash{\begin{tabular}[t]{l}$z(\lambda)$\\\end{tabular}}}}%
    \put(0.83298775,0.35232352){\color[rgb]{0,0,0}\makebox(0,0)[lt]{\lineheight{1.25}\smash{\begin{tabular}[t]{l}$u(\lambda)$\end{tabular}}}}%
    \put(0.83360839,0.21683393){\color[rgb]{0,0,0}\makebox(0,0)[lt]{\lineheight{1.25}\smash{\begin{tabular}[t]{l}$z$\end{tabular}}}}%
    \put(0.89394713,0.2988525){\color[rgb]{0,0,0}\makebox(0,0)[lt]{\lineheight{1.25}\smash{\begin{tabular}[t]{l}$u$\end{tabular}}}}%
    \put(0.64880055,0.05803403){\color[rgb]{0,0,0}\makebox(0,0)[lt]{\lineheight{1.25}\smash{\begin{tabular}[t]{l}$w$\end{tabular}}}}%
    \put(0,0){\includegraphics[width=\unitlength,page=2]{flow.pdf}}%
  \end{picture}%
\endgroup%